\documentclass[12pt,british,reqno]{amsart}

\usepackage{babel}

\usepackage{amsthm,amssymb,amsmath} 
\usepackage[latin1]{inputenc}
\usepackage{tikz}
\usepackage{tikz-cd}
\usepackage{enumerate}

\newtheorem{teo}{Theorem}[section]
\newtheorem{lemma}[teo]{Lemma}

\newtheorem{coro}[teo]{Corollary}
\newtheorem{prop}[teo]{Proposition}
\newtheorem{lem}[teo]{Lemma}

\theoremstyle{remark}
\newtheorem{remark}[teo]{Remark}

\numberwithin{equation}{section}
\allowdisplaybreaks

\usepackage[pdftex,%
bookmarks=true,
breaklinks=true,
bookmarksnumbered = true,
colorlinks= true,
urlcolor= green,
anchorcolor = yellow,
citecolor=blue,
]{hyperref}                   

\begin{document}

\title[Lower central and derived series of surface braid groups]{The lower central and derived series of the braid groups of compact surfaces}

\author[J.~Guaschi]{John Guaschi}
\address{Normandie Univ., Unicaen, CNRS, Laboratoire de Math\'ematiques Nicolas Oresme UMR CNRS~6139, 14000 Caen, France.}
\email{john.guaschi@unicaen.fr}

\author[C.~M.~Pereiro]{Carolina de Miranda e Pereiro}
\address{Universidade Federal do Esp\'{i}rito Santo, UFES, Departamento de Matem\'{a}tica, 29075-910, Vit\'{o}ria, Esp\'{i}rito Santo, Brazil}
\email{carolinapereiro@gmail.com}
%
\date{\today}
%


\begin{abstract}
\noindent Let $M$ be a compact surface, either orientable or non-orientable. We study the lower central and derived series of the braid and pure braid groups of $M$ in order to determine the values of $n$ for which $B_{n}(M)$ and $P_{n}(M)$ are residually nilpotent or residually soluble. First, we solve this problem for the case where $M$ is the $2$-torus. We then give a general description of these series for an arbitrary semi-direct product that allows us to calculate explicitly the lower central series of $P_{2}(\mathbb{K})$, where $\mathbb{K}$ is the Klein bottle, and to give an estimate for the derived series of $P_{n}(\mathbb{K})$. Finally, if $M$ is a non-orientable compact surface without boundary, we determine the values of $n$ for which $B_{n}(M)$ is residually nilpotent or residually soluble in the cases that were not already known in the literature.
\end{abstract}

\maketitle

\section{Introduction}

The braid groups of the disc, also called the Artin braid groups, were introduced by E.~Artin~\cite{Ar0}. If $n\geq1$, the $n$-string Artin braid group, denoted by $B_{n}$, is generated by elements $\sigma_{1},\ldots,\sigma_{n-1}$ that are subject to the Artin relations:\[\left\{\begin{aligned}&\sigma_{i}\sigma_{i+1}\sigma_{i}=\sigma_{i+1}\sigma_{i}\sigma_{i+1} & \text{for all $1\leq i\leq n-2$}\\ &\sigma_{j}\sigma_{i}=\sigma_{i}\sigma_{j}& \text{if $|i-j|\geq 2$ and $1\leq i,j\leq n-1$.}\end{aligned}\right.\]
The notion of braid group was generalised to surfaces by Fox and Neuwirth using configuration spaces as follows~\cite{FoN}. Let $M$ be a compact, connected surface, and let $n\in \mathbb{N}$. The \textit{$n^{\text{th}}$ configuration space of $M$}, denoted by $F_{n}(M)$, is defined by:\[F_{n}(M)=\left\{(x_{1},\ldots,x_{n})\,: \text{$x_{i}\in M$, and $x_{i}\neq x_{j}$ if $i\neq j$, $i,j=1,\ldots,n$}\right\}.\]
The \textit{$n$-string pure braid group $P_{n}(M)$} of $M$ is defined by $P_{n}(M)=\pi_{1}(F_{n}(M))$. The symmetric group $S_{n}$ on $n$ letters acts freely on $F_{n}(M)$ by permuting coordinates, and the \textit{$n$-string braid group $B_{n}(M)$} of $M$ is defined by $B_{n}(M)=\pi_{1}(F_{n}(M)/S_{n})$.  This gives rise to the following short exact sequence: \begin{equation}\label{seq1}1\longrightarrow P_{n}(M)\longrightarrow B_{n}(M)\longrightarrow S_{n}\longrightarrow 1.\end{equation}

If $m\geq 1$, the projection $p\colon\thinspace F_{n+m}(M)\longrightarrow F_{n}(M)$ defined by $p(x_{1},\ldots,x_{n},\ldots,x_{n+m})=(x_{1},\ldots,x_{n})$ induces a homomorphism $p_{\ast}\colon\thinspace P_{n+m}(M)\longrightarrow P_{n}(M)$. Geometrically, $p_{\ast}$ is the homomorphism that `forgets' the last $m$ strings. If $M$ is without boundary, Fadell and Neuwirth showed that $p$ is a locally-trivial fibration~\cite{FN}, with fibre $F_{m}(M\setminus\left\{x_{1},\ldots,x_{n}\right\})$ over the point $(x_{1},\ldots,x_{n})$, which we consider to be a subspace of the total space via the map \mbox{$i\colon\thinspace F_{m}(M\setminus\left\{x_{1},\ldots,x_{n}\right\})\longrightarrow F_{n+m}(M)$} defined by $i((y_{1},\ldots,y_{m}))=(x_{1},\ldots,x_{n},y_{1},\ldots,y_{m})$. Applying the associated long exact sequence in homotopy to this fibration, we obtain the Fadell-Neuwirth short exact sequence of pure braid groups:
\begin{equation}\label{seqFN}
1\longrightarrow P_{m}(M\setminus\left\{x_{1},\ldots,x_{n}\right\})\stackrel{i_{\ast}}{\longrightarrow} P_{n+m}(M)\stackrel{p_{\ast}}{\longrightarrow} P_{n}(M)\longrightarrow 1,\end{equation}
where $n\geq 3$ if $M$ is the sphere $\mathbb S^{2}$~\cite{Fa,FVB}, $n\geq 2$ if $M$ is the projective plane $\mathbb RP^{2}$~\cite{FVB}, and $n\geq1$ otherwise~\cite{FN}, and $i_{\ast}$ is the homomorphism induced by the map $i$. This sequence has been widely studied. If $M$ is the torus $\mathbb{T}$ or the Klein bottle $\mathbb{K}$, the existence of a non-vanishing vector field on $M$ allows one to construct a section for $p$~\cite{FN}. This implies that the short exact sequence~(\ref{seqFN}) splits for all $n,m\in\mathbb{N}$, and that $P_{n}(M)$ may be decomposed as an iterated semi-direct product (see Proposition~\ref{sectionPn} for an explicit section for $p_{\ast}$ in the case $M=\mathbb{K}$).

Let $G$ be a group. If $g,g'\in G$ then $[g,g']=gg'g^{-1}g'^{-1}$ denotes their \emph{commutator}, and if $H$ and $K$ are subgroups of $G$, then the \emph{commutator subgroup of $H$ and $K$}, denoted by $[H,K]$, is defined by $[H,K]=\langle [h,k] : \text{$h\in H$ and $k\in K$}\rangle$, the subgroup of $G$ generated by the commutators of $H$ and $K$. The \textit{lower central series} $\left\{\Gamma_{i}(G)\right\}_{i\geq1}$ of $G$ is defined inductively by $\Gamma_{1}(G)=G$, and for $i\geq1$, $\Gamma_{i+1}(G)=\left[\Gamma_{i}(G),G\right]$, and the \textit{derived series} $\left\{G^{(i)}\right\}_{i\geq0}$ of $G$ is defined inductively by $G^{(0)}=G$, and for $i\geq0$, $G^{(i+1)}=\left[G^{(i)},G^{(i)}\right]$. The quotient $G/\Gamma_{2}(G)$ is the \textit{Abelianisation} of $G$ that we denote by $G^{\text{Ab}}$. Following P.~Hall, for any group-theoretic property $\mathcal P$, a group $G$ is said to be residually $\mathcal P$ if for any (non-trivial) element $x\in G$, there exists a group $H$ that possesses property $\mathcal P$ and a surjective homomorphism $\varphi:G\longrightarrow H$ such that $\varphi(x)\neq1$. It is well known that a group $G$ is \textit{residually nilpotent} (resp.\ \textit{residually soluble}) if and only if $\bigcap_{i\geq1}\Gamma_{i}(G)=\left\{1\right\}$ (resp.\ $\bigcap_{i\geq 0}G^{(i)}=\left\{1\right\}$).

If $p$ is a prime number, the \emph{lower $\mathbb{F}_{p}$-linear central filtration} $\left\{\gamma^{p}_{i}(G)\right\}_{i\geq1}$ of $G$ is defined inductively by $\gamma^{p}_{1}(G)=G$, and for $i\geq1$, $\gamma^{p}_{i+1}(G)=\left\langle \left[\gamma^{p}_{i}(G),G\right],\,x^{p}\,:\,x\in\gamma^{p}_{i}(G)\right\rangle$~\cite{Pa}. If the group $G$ is finitely generated, then $G$ is \emph{residually $p$-finite} if and only if $\bigcap_{i\geq1}\gamma^{p}_{i}(G)=\left\{1\right\}$~\cite[Proposition~2.3(2)]{Pa}. For any group $G$, $G^{(i)}\subset\Gamma_{i+1}(G)\subset \gamma^{p}_{i+1}(G)$, so residually $p$-finite implies residually nilpotent, which in turn implies residually soluble.

In this paper, we are interested in computing the lower central and derived series of $B_{n}(M)$ and $P_{n}(M)$, the aim being to determine the values of $n$ for which these groups are residually nilpotent or residually soluble. We also give some partial results for the residual $p$-finiteness of $P_{n}(\mathbb{K})$.

The first results about the lower central series of the braid groups were obtained by Gorin and Lin who gave a presentation of the commutator subgroup $\Gamma_{2}(B_{n})$ of $B_{n}$ for $n\geq 3$, and who showed that $(B_{n})^{(1)}=(B_{n})^{(2)}$ for all $n\geq 5$, which implies that $(B_{n})^{(1)}$ is perfect~\cite{GL}. As a consequence, $\Gamma_{2}(B_{n})=\Gamma_{3}(B_{n})$, for all $n\geq 3$, so $B_{n}$ is not residually nilpotent. The lower central series of the pure braid group $P_{n}$ was studied by Falk and Randell~\cite{FR} and by Kohno~\cite{K}, who proved independently that $P_{n}$ is residually nilpotent for all $n\geq 1$.

The braid groups of orientable surfaces were studied by Bellingeri, Gervais and Guaschi~\cite{BGG}. If $M_{g,m}$ is a compact, connected, orientable surface of genus $g\geq1$ with $m\geq 0$ boundary components, then $B_{n}(M_{g,m})$ is not residually nilpotent if $n\geq 3$, and $B_{2}(\mathbb{T})$ is residually nilpotent. In the case of the pure braid groups, $P_{n}(M_{g,m})$ is residually torsion-free nilpotent for all $n\geq1$ if $m\geq1$, or if $g=1$ and $m=0$ (the torus). If $m=0$ and $g\geq1$, Bardakov and Bellingeri proved that $P_{n}(M_{g,m})$ is residually torsion-free nilpotent for all $n\geq1$, and the braid group $B_{2}(M_{g,m})$ is residually $2$-finite, in particular, it is residually nilpotent~\cite{BB}. Gon\c{c}alves and Guaschi studied the lower central and derived series of the braid groups of the sphere $\mathbb S^{2}$ and the projective plane $\mathbb RP^2$~\cite{GG, GG1}. For the sphere, $B_{n}(\mathbb S^{2})$ is residually nilpotent if and only if $n\leq2$, and residually soluble if and only if $n\leq4$. In the case of the projective plane, $B_{n}(\mathbb RP^{2})$ is residually nilpotent if and only if $n\leq3$, and if $n\neq 4$, $B_{n}(\mathbb RP^{2})$ is residually soluble if and only if $n<4$. More recently, if $M$ is a non-orientable surface different from $\mathbb RP^{2}$, Bellingeri and Gervais showed that $P_{n}(M)$ is residually $2$-finite, and so is residually nilpotent~\cite{BG}.

In this paper, we study the derived series of the torus and the lower central series and derived series of non-orientable surfaces. Our main results are as follows. 

\begin{teo}\label{toro} 
The group $B_{n}(\mathbb{T})$ is residually soluble if and only if $n\leq5$.
\end{teo}

For non-orientable surfaces, we first study the case of the Klein bottle. For an arbitrary semi-direct product, the following result describes its lower central series, and gives some information about its derived series. It will be used in the computation of the lower central and derived series of $P_n(\mathbb{K})$, but we believe that it may be applicable to other groups.

\begin{teo}\label{serie} Let $G$ and $H$ be groups, and let $\varphi:G\longrightarrow Aut (H)$ be an action of $G$ on $H$. We define recursively the following subgroups of $H$: $L_{1}=V_{1}=H$, and if $n\geq 2$: 
\begin{align*}
K_{n} & =\left\langle \varphi(g)(h).h^{-1}\,:\, g\in \Gamma_{n-1}(G), h\in H\right\rangle, & 
H_{n} &=\left\langle \varphi(g)(h).h^{-1}\,:\, g\in G, h\in L_{n-1}\right\rangle,\\
\widetilde{H}_{n}& =\left\langle \varphi(g)(h).h^{-1}\,:\, g\in G, h\in V_{n-1}\right\rangle, &
L_{n}&=\left\langle K_{n},H_{n},[H,L_{n-1}]\right\rangle,\\
V_{n}&=\bigl\langle \widetilde{H}_{n},[H,V_{n-1}]\bigr\rangle. &&
\end{align*}
Then $\varphi$ induces an action, which we also denote by $\varphi$, of $\Gamma_{n}(G)$ on $L_{n}$ (resp.\ of $G^{(n+1)}$ on $V_{n+2}$), and for all $n\in \mathbb{N}$, we have:
\begin{enumerate}
\item\label{it:lcspart1} $\Gamma_{n}(H\rtimes_{\varphi}G)=L_{n}\rtimes_{\varphi}\Gamma_{n}(G)$.

\item\label{it:lcspart2} $(H\rtimes_{\varphi}G)^{(n-1)}\subset V_{n}\rtimes_{\varphi}G^{(n-1)}$.
\end{enumerate}
\end{teo}

For the case of the commutator subgroup ($n=2$), part~(\ref{it:lcspart1}) was obtained in~\cite[Proposition~3.3]{GG}. Using Theorem~\ref{serie}, we compute explicitly $\Gamma_{n}(P_{2}(\mathbb{K}))$ and $\gamma^{2}_{n}(P_{2}(\mathbb{K}))$ in Theorems~\ref{GammaP2} and~\ref{gammaP2} respectively, and and it will follow that $P_{2}(\mathbb{K})$ is residually nilpotent and residually $2$-finite. If $n\geq 3$, in Theorem~\ref{soluvel}, we give an estimate of the derived series of $P_{n}(\mathbb{K})$, which will enable us to show that $P_{n}(\mathbb{K})$ is residually soluble, and will allow us to determine the values of $n$ for which $B_{n}(\mathbb{K})$ is residually nilpotent or residually soluble.

\begin{teo}\label{klein} $B_{n}(\mathbb{K})$ is residually nilpotent if and only if $n\leq2$, and residually soluble if and only if $n\leq4$.
\end{teo}

For a non-orientable surface $M$ without boundary of higher genus, we may decide whether $B_{n}(M)$ is residually nilpotent or residually soluble using results of~\cite{BG,GG1}.

\begin{teo}\label{final} Let $M$ be a compact non-orientable surface of genus $g$ without boundary. If $g\geq 2$ (resp.\ $g=1$), then $B_{n}(M)$ is residually nilpotent if and only if $n\leq2$ (resp.\ $n\leq3$). For all $g\geq1$, $B_{n}(M)$ is residually soluble if and only if $n\leq4$.
\end{teo}

This paper is organised as follows. In Section~\ref{sec:gens}, we give presentations of the braid groups used in this paper, as well as the statement of Theorem~\ref{TG} due to Gruenberg that will be required in the proofs of some of our results. Theorem~\ref{serie} is proved in Section~\ref{sec:lcssemi}. In Section~\ref{sec:torus}, we study the case of the torus and we prove Theorem~\ref{toro}. In Section~\ref{sec:klein}, our focus is on the braid groups of the Klein bottle, and we use Theorem~\ref{serie} for the proof of Theorem~\ref{klein}. Theorem~\ref{final} is proved in Section~\ref{sec:higher}. 
  If $M$ is a compact surface different from $\mathbb{K}$ and the M\"obius band, the centre $Z(B_{n}(M))$ of $B_{n}(M)$ is known~\cite{B1,chow,GVB,Mu,PR,VB}. We determine $Z(B_{n}(\mathbb{K}))$ in Proposition~\ref{centro}, and for the sake of completeness, in Proposition~\ref{prop:centremobius} of the  Appendix, we compute the centre of the braid groups of the M\"obius band.

\subsubsection*{Acknowledgements}
The authors would like to thank P.~Bellingeri, S.~Gervais, D.~Gon\c{c}alves, L.~Paris and D. Vendr\'{u}scolo for stimulating conversations. C.~M.~Pereiro was supported by project grant n\textsuperscript{o}~2010/18930-6 and~2012/01740-5 from FAPESP. During the writing of this paper, J.~Guaschi was partially supported by the CNRS/FAPESP PRC project n\textsuperscript{o}~275209.

\section{Generalities}\label{sec:gens}

In this section, we give the presentations of the braid and pure braid groups that will be used in this paper. If $M=\mathbb{T}$ or $\mathbb{K}$, we will make use of the following presentations of $P_{n}(M)$ and $B_{n}(M)$.

\begin{teo}[\cite{P}]\label{puras}
Let $n\geq1$, and let $M$ be the torus $\mathbb{T}$ or the Klein bottle $\mathbb{K}$. The following constitutes a presentation of the pure braid group $P_n(M)$ of $M$:

generators: $\left\{a_i,\, b_i,\,i=1,\ldots,n\right\}\cup\left\{C_{i,j},\,1\leq i < j\leq n\right\}$.

relations:
\begin{enumerate}
\item\label{it:puras1} $a_ia_j=a_ja_i$, 		 $(1\leq i<j\leq n)$;
	
\item\label{it:puras2} $a^{-1}_{i}b_ja_i=b_ja_jC^{-1}_{i,j}C_{i+1,j}a^{-1}_{j}$, 	 $(1\leq i<j\leq n)$;
	
\item\label{it:puras3} $a^{-1}_{i}C_{j,k}a_i=\left\{\begin{array}{l} C_{j,k},\,\,  (1\leq i<j<k\leq n)\,\, \mbox{or}\,\,(1\leq j<k<i\leq n),\\
	a_kC^{-1}_{i+1,k}C_{i,k}a^{-1}_{k}C_{j,k}C^{-1}_{i,k}C_{i+1,k},\,\,  (1\leq j\leq i<k\leq n);\end{array}\right.$
	
\item\label{it:puras4} $C^{-1}_{i,l}C_{j,k}C_{i,l}=\left\{\begin{array}{l} C_{j,k},\,\,  (1\leq i<l<j<k\leq n)\,\, \mbox{or}\,\,(1\leq j\leq i<l<k\leq n),\\
	C_{i,k}C^{-1}_{l+1,k}C_{l,k}C^{-1}_{i,k}C_{j,k}C^{-1}_{l,k}C_{l+1,k},\,\,  (1\leq i< j\leq l<k\leq n);\end{array}\right.$
	
\item\label{it:puras5} $\left\{ \begin{array}{lll}
\prod^{n}_{j=i+1}C^{-1}_{i,j}C_{i+1,j}=a_{i}b_{i}C_{1,i}a^{-1}_{i}b^{-1}_{i},& 		 (1\leq i\leq n), &\text{if}\,\, M=\mathbb{T},\\
 \prod^{n}_{j=i+1}C_{i,j}C^{-1}_{i+1,j}=b_{i}C_{1,i}a^{-1}_{i}b^{-1}_{i}a^{-1}_{i},& 		 (1\leq i\leq n), &\text{if}\,\, M=\mathbb{K};\end{array}\right.$
	
\item\label{it:puras6} $\left\{\begin{array}{ll} b_{j}b_{i}=b_{i}b_{j},\,\, (1\leq i<j\leq n), &\text{if}\,\, M=\mathbb{T},\\
b_{j}b_{i}=b_{i}b_{j}C_{i,j}C^{-1}_{i+1,j},\,\, 		 (1\leq i<j \leq n), &\text{if}\,\, M=\mathbb{K};\end{array}\right.$
	
\item\label{it:puras7} $\left\{\begin{array}{ll} b^{-1}_{i}a_jb_i=a_jb_jC_{i,j}C^{-1}_{i+1,j}b^{-1}_{j},\,\, (1\leq i<j\leq n), &\text{if}\,\, M=\mathbb{T},\\
	b^{-1}_{i}a_jb_i=a_jb_j(C_{i,j}C^{-1}_{i+1,j})^{-1}b^{-1}_{j},\,\, (1\leq i<j\leq n), &\text{if}\,\, M=\mathbb{K}; \end{array}\right.$
	
\item\label{it:puras8} $\left\{\begin{array}{ll}{b^{-1}_{i}C_{j,k}b_i=\left\{\begin{array}{l} C_{j,k},\,\,  (1\leq i<j<k\leq n)\,\, \mbox{or}\,\,(1\leq j<k<i\leq n),\\
	C_{i+1,k}C^{-1}_{i,k}C_{j,k}b_kC_{i,k}C^{-1}_{i+1,k}b^{-1}_{k},\,\,  (1\leq j\leq i<k\leq n);\end{array}\right.} & \text{if}\,\, M=\mathbb{T},\\
	{b^{-1}_{i}C_{j,k}b_i=\left\{\begin{array}{l} C_{j,k},\,\,  (1\leq i<j<k\leq n)\,\, \mbox{or}\,\,(1\leq j<k<i\leq n),\\
	C_{i+1,k}C^{-1}_{i,k}C_{j,k}b_k(C_{i,k}C^{-1}_{i+1,k})^{-1}b^{-1}_{k},\,\,  (1\leq j\leq i<k\leq n).\end{array}\right.} &\text{if}\,\, M=\mathbb{K}. \end{array}\right.$
	
	\end{enumerate}

\end{teo}

\begin{teo}[\cite{P}]\label{total} Let $n\geq1$, and let $M$ be the torus $\mathbb{T}$ or the Klein bottle $\mathbb{K}$. The following constitutes a presentation of the braid group $B_n(M)$ of $M$:

generators: $a, b, \sigma_1, \ldots, \sigma_{n-1}$.

relations:
\begin{enumerate}
\item\label{it:full1} $\sigma_{i}\sigma_{i+1}\sigma_{i}=\sigma_{i+1}\sigma_{i}\sigma_{i+1}$;
	
\item\label{it:full2} $\sigma_{j}\sigma_{i}=\sigma_{i}\sigma_{j}$, if $|i-j|\geq 2$;

\item\label{it:full3} $a\sigma_{j}=\sigma_{j}a$, if $j\geq 2$;
	
\item\label{it:full4} $b\sigma_{j}=\sigma_{j}b$, if $j\geq 2$;
	
\item\label{it:full5} $b^{-1}\sigma_{1}a=\sigma_{1}a\sigma_{1}b^{-1}\sigma_{1}$;

\item\label{it:full6} $a(\sigma_1a\sigma_1)=(\sigma_1a\sigma_1)a$;

\item\label{it:full7} $\left\{\begin{array}{lr} b(\sigma^{-1}_{1}b\sigma^{-1}_{1})=(\sigma^{-1}_{1}b\sigma^{-1}_{1})b,& \text{if}\,\, M=\mathbb{T},\\
b(\sigma^{-1}_{1}b\sigma_{1})=(\sigma^{-1}_{1}b\sigma^{-1}_{1})b,& \text{if}\,\, M=\mathbb{K};\end{array}\right.$
	
\item\label{it:full8} $\sigma_{1}\sigma_{2}\cdots\sigma_{n-2}\sigma^{2}_{n-1}\sigma_{n-2}\cdots\sigma_{2}\sigma_{1}=\left\{\begin{array}{lr} bab^{-1}a^{-1}& \text{if}\,\, M=\mathbb{T},\\
ba^{-1}b^{-1}a^{-1}&\,\, \text{if}\,\, M=\mathbb{K}.\end{array}\right.$
\end{enumerate}
\end{teo}

We consider the torus and the Klein bottle to be a square whose edges are identified as indicated in Figure~\ref{tgk}. Geometric representatives of the generators of $P_{n}(\mathbb{T})$ and $P_{n}(\mathbb{K})$ given in Theorem~\ref{puras} are illustrated in Figure~\ref{fig:geradores}, and may be interpreted as follows. For $1\leq i \leq n$, the $i^{\text{th}}$ string is the only non-trivial string of the braid $a_{i}$ (resp.\ of $b_{i}$), and it passes through the edge $\alpha$ (resp.\ $\beta$). If $1\leq i< j \leq n$, the $j^{\text{th}}$ string is the only non-trivial string of the braid $C_{i,j}$, and it encircles all of the basepoints between the $i^{\text{th}}$ and $j^{\text{th}}$ points. If $i=j$, it will be convenient to define $C_{i,i}$ to be the trivial braid. The figures represent the projection of the braids onto $M$, so the constant paths in each figure correspond to vertical strings of the braid.
\begin{figure}[h!]
\hfill
\begin{tikzpicture}[scale=1, very thick]

{\draw[line width=2pt][-] (0,1) -- (-2,0);}
{\draw[line width=2pt][->>] (0,1) -- (-1,0.5);}
{\draw[line width=2pt][-] (2,0) -- (0,-1);}
{\draw[line width=2pt][-<<] (0,-1) -- (1,-0.5);}
{\draw[line width=2pt][-<] (0,1) -- (1,0.5);}
{\draw[line width=2pt][-] (2,0) -- (0,1);}
{\draw[line width=2pt][>-] (-1,-0.5) -- (-2,0);}
{\draw[line width=2pt][-] (0,-1) -- (-2,0);}

{\draw[line width=2pt][-] (7,1) -- (5,0);}
{\draw[line width=2pt][-] (9,0) -- (7,-1);}
{\draw[line width=2pt][->>] (7,1) -- (6,0.5);}
{\draw[line width=2pt][->>] (7,-1) -- (8,-0.5);}

{\draw[line width=2pt][-<] (7,1) -- (8,0.5);}
{\draw[line width=2pt][-] (9,0) -- (7,1);}
{\draw[line width=2pt][>-] (6,-0.5) -- (5,0);}
{\draw[line width=2pt][-] (7,-1) -- (5,0);}

%

{\node at (0,-1.5) {$M=\mathbb{T}$};};
{\node at (7,-1.5) {$M=\mathbb{K}$};};

\foreach \j  in {1,8}
{\node at (\j,-0.85) {$\alpha$};};

\foreach \j  in {-1,6}
{\node at (\j,0.85) {$\alpha$};};

\foreach \j  in {1,8}
{\node at (\j,0.85) {$\beta$};};

\foreach \j  in {-1,6}
{\node at (\j,-0.85) {$\beta$};};

\end{tikzpicture}
\hspace*{\fill}
\caption{Squares representing $\mathbb{T}$ and $\mathbb{K}$}\label{tgk}
\end{figure}
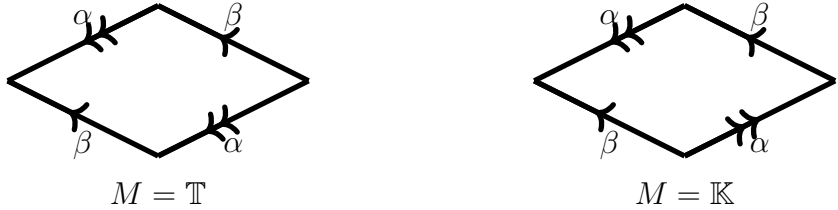
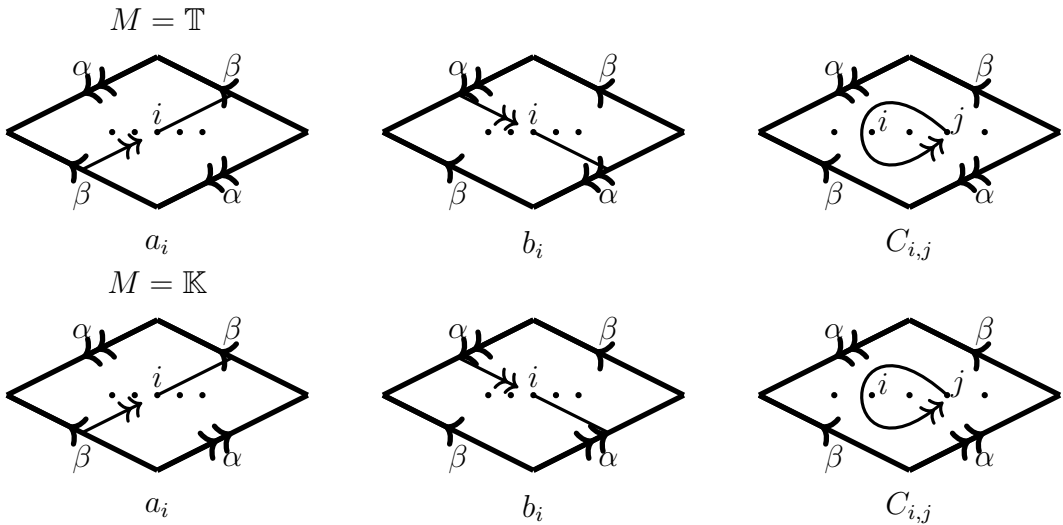
\begin{figure}[h!]
\hfill
\begin{tikzpicture}[scale=1, very thick]

{\draw[line width=2pt][->>] (0,1)-- (-1,0.5);}
{\draw[line width=2pt][-] (0,1)-- (-2,0);}
{\draw[line width=2pt][-<<] (0,-1) -- (1,-0.5);}
{\draw[line width=2pt][-] (2,0) -- (0,-1);}
{\draw[line width=2pt][-<] (0,1) -- (1,0.5);}
{\draw[line width=2pt][-] (2,0) -- (0,1);}
{\draw[line width=2pt][>-] (-1,-0.5) -- (-2,0);}
{\draw[line width=2pt][-] (0,-1) -- (-2,0);}

{\draw (0,0) circle (0.2mm);}
{\draw (0.3,0) circle (0.2mm);}
{\draw (0.6,0) circle (0.2mm);}
{\draw (-0.3,0) circle (0.2mm);}
{\draw (-0.6,0) circle (0.2mm);}

\draw (0,0) -- (1,0.5);
\draw[->>] (-1,-0.5) -- (-0.2,-0.1);

\node at (0,0.25) {$i$};
\node at (0,-1.5) {$a_{i}$};
{\draw[line width=2pt][->>] (5,1) -- (4,0.5);}
{\draw[line width=2pt][-] (5,1) -- (3,0);}
{\draw[line width=2pt][-<<] (5,-1) -- (6,-0.5);}
{\draw[line width=2pt][-] (7,0) -- (5,-1);}

{\draw[line width=2pt][-<] (5,1) -- (6,0.5);}
{\draw[line width=2pt][-] (7,0) -- (5,1);}
{\draw[line width=2pt][>-] (4,-0.5) -- (3,0);}
{\draw[line width=2pt][-] (5,-1) -- (3,0);}

{\draw (5,0) circle (0.2mm);}
{\draw (5.3,0) circle (0.2mm);}
{\draw (5.6,0) circle (0.2mm);}
{\draw (4.7,0) circle (0.2mm);}
{\draw (4.4,0) circle (0.2mm);}

\draw (5,0) -- (6,-0.5);
\draw[->>] (4,0.5) -- (4.8,0.1);

\node at (5,0.25) {$i$};
\node at (5,-1.5) {$b_{i}$};

{\draw[line width=2pt][->>] (10,1) -- (9,0.5);}
{\draw[line width=2pt][-] (10,1) -- (8,0);}
{\draw[ line width=2pt][-<<] (10,-1)--(11,-0.5);}
{\draw[ line width=2pt][-] (12,0) -- (10,-1);}

{\draw[line width=2pt][-<] (10,1) -- (11,0.5);}
{\draw[line width=2pt][-] (12,0) -- (10,1);}
{\draw[line width=2pt][>-] (9,-0.5) -- (8,0);}
{\draw[line width=2pt][-] (10,-1) -- (8,0);}

{\draw (10,0) circle (0.2mm);}
{\draw (10.5,0) circle (0.2mm);}
{\draw (11,0) circle (0.2mm);}
{\draw (9.5,0) circle (0.2mm);}
{\draw (9,0) circle (0.2mm);}

\draw[->>] (10.5,0) .. controls (9,1.35) and (9,-1.32) .. (10.45,-0.09);
\node at (9.65,0.15) {$i$};
\node at (10.65,0.15) {$j$};
\node at (10,-1.5) {$C_{i,j}$};

{\draw[ line width=2pt][->>] (0,-2.5) -- (-1,-3);} 
{\draw[ line width=2pt][-] (0,-2.5) -- (-2,-3.5);}
{\draw[ line width=2pt][-] (2,-3.5) -- (0,-4.5);} 
{\draw[ line width=2pt][->>] (0,-4.5) -- (1,-4);} 

{\draw[line width=2pt][-<] (0,-2.5) -- (1,-3);}
{\draw[line width=2pt][-] (2,-3.5) -- (0,-2.5);}
{\draw[line width=2pt][>-] (-1,-4) -- (-2,-3.5);}
{\draw[line width=2pt][-] (0,-4.5) -- (-2,-3.5);}

{\draw (0,-3.5) circle (0.2mm);}
{\draw (0.3,-3.5) circle (0.2mm);}
{\draw (0.6,-3.5) circle (0.2mm);}
{\draw (-0.3,-3.5) circle (0.2mm);}
{\draw (-0.6,-3.5) circle (0.2mm);}

\draw (0,-3.5) -- (1,-3);
\draw[->>] (-1,-4) -- (-0.2,-3.6);

\node at (0,-3.25) {$i$};
\node at (0,-5) {$a_{i}$};
{\draw[ line width=2pt][-] (5,-2.5) -- (3,-3.5);}
{\draw[ line width=2pt][->>] (5,-2.5) -- (4,-3);}
{\draw[ line width=2pt][->>] (5,-4.5) --(6,-4);}

{\draw[ line width=2pt][-] (7,-3.5) -- (5,-4.5);}

{\draw[line width=2pt][-<] (5,-2.5) -- (6,-3);}
{\draw[line width=2pt][-] (7,-3.5) -- (5,-2.5);}
{\draw[line width=2pt][>-] (4,-4) -- (3,-3.5);}
{\draw[line width=2pt][-] (5,-4.5) -- (3,-3.5);}

%
{\draw (5,-3.5) circle (0.2mm);}
{\draw (5.3,-3.5) circle (0.2mm);}
{\draw (5.6,-3.5) circle (0.2mm);}
{\draw (4.7,-3.5) circle (0.2mm);}
{\draw (4.4,-3.5) circle (0.2mm);}

\draw (5,-3.5) -- (6,-4);
\draw[->>] (4,-3) -- (4.8,-3.4);

\node at (5,-3.25) {$i$};
\node at (5,-5) {$b_{i}$};

{\draw[ line width=2pt][-] (10,-2.5) -- (8,-3.5);}
{\draw[ line width=2pt][->>] (10,-2.5) -- (9,-3);}
{\draw[ line width=2pt][->>] (10,-4.5) -- (11,-4);}
{\draw[ line width=2pt][-] (12,-3.5) -- (10,-4.5);}

{\draw[line width=2pt][-<] (10,-2.5) -- (11,-3);}
{\draw[line width=2pt][-] (12,-3.5) -- (10,-2.5);}
{\draw[line width=2pt][>-] (9,-4) -- (8,-3.5);}
{\draw[line width=2pt][-] (10,-4.5) -- (8,-3.5);}

{\draw (10,-3.5) circle (0.2mm);}
{\draw (10.5,-3.5) circle (0.2mm);}
{\draw (11,-3.5) circle (0.2mm);}
{\draw (9.5,-3.5) circle (0.2mm);}
{\draw (9,-3.5) circle (0.2mm);}

\draw[->>] (10.5,-3.5) .. controls (9,-2.15) and (9,-4.82) .. (10.45,-3.59);
\node at (9.65,-3.35) {$i$};
\node at (10.65,-3.35) {$j$};
\node at (10,-5) {$C_{i,j}$};

\node at (0,1.5) {$M=\mathbb{T}$};
\node at (0,-2) {$M=\mathbb{K}$};

\foreach \j  in {1,6,11}
{\node at (\j,-0.85) {$\alpha$};};

\foreach \j  in {-1,4,9}
{\node at (\j,0.85) {$\alpha$};};

\foreach \j  in {1,6,11}
{\node at (\j,0.85) {$\beta$};};

\foreach \j  in {-1,4,9}
{\node at (\j,-0.85) {$\beta$};};

\foreach \j  in {1,6,11}
{\node at (\j,-4.35) {$\alpha$};};

\foreach \j  in {-1,4,9}
{\node at (\j,-2.65) {$\alpha$};};

\foreach \j  in {1,6,11}
{\node at (\j,-2.65) {$\beta$};};

\foreach \j  in {-1,4,9}
{\node at (\j,-4.35) {$\beta$};};
\end{tikzpicture}
\hspace*{\fill}
\caption{The generators of $P_{n}(\mathbb{T})$ and $P_{n}(\mathbb{K})$}\label{fig:geradores}
\end{figure}
The generators of $B_{n}(\mathbb{T})$ and $B_{n}(\mathbb{K})$ given in Theorem~\ref{total} may be taken to be the standard Artin generators $\sigma_{1},\ldots,\sigma_{n-1}$ of $B_{n}$ as shown in Figure~\ref{fig:artin}, and $a=a_{1}$ and $b=b_{1}$.
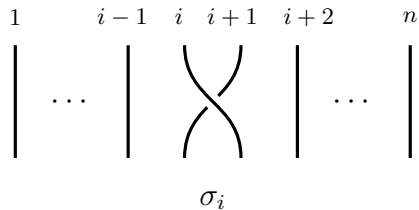
\begin{figure}[h!]
\hfill
\begin{tikzpicture}[scale=0.75, very thick]

\foreach \k in {5}
{\draw (\k,3) .. controls (\k,2) and (\k-1,2) .. (\k-1,1);};

\foreach \k in {4}
{\draw[white,line width=6pt] (\k,3) .. controls (\k,2) and (\k+1,2) .. (\k+1,1);
\draw (\k,3) .. controls (\k,2) and (\k+1,2) .. (\k+1,1);};

%

\foreach \k in {1,3,6,8}
{\draw (\k,1)--(\k,3);};

\foreach \k in {2,7}
{\node at (\k,2) {$\cdots$};};

\foreach \k in {1}
{\node at (\k,3.5) {{\scriptsize$1$}};
\node at (\k+1.9,3.5) {{\scriptsize$i-1$}};
\node at (\k+2.9,3.52) {{\scriptsize$i$}};
\node at (\k+3.85,3.5) {{\scriptsize$i+1$}};
\node at (\k+5.2,3.5) {{\scriptsize$i+2$}};
\node at (\k+7,3.5) {{\scriptsize$n$}};};

\node at (4.5,0.25) {$\sigma_{i}$};

\end{tikzpicture}
\hspace*{\fill}
\caption{The braid $\sigma_{i}$}\label{fig:artin}
\end{figure}
Various presentations of the braid and pure braid groups of the torus and the Klein bottle may be found in the literature~\cite{Be,B1,GM,S}, but we choose to work with those of Theorems~\ref{puras} and~\ref{total} because they highlight the similarities and differences between the braid groups of $\mathbb{T}$ and $\mathbb{K}$. For example, the word $C_{i,j}C^{-1}_{i+1,j}$ that appears in our presentation of $P_{n}(\mathbb{T})$ is often replaced by its inverse in $P_{n}(\mathbb{K})$. To prove Theorem~\ref{puras} (resp.\ Theorem~\ref{total}), one may use the Fadell-Neuwirth short exact sequence~(\ref{seqFN}) (resp.\ the short exact sequence~(\ref{seq1})), induction on $n$, and the following standard method for obtaining a presentation of a group extension~\cite[Proposition~1, p.~139]{J}. Given a short exact sequence $1\longrightarrow A\stackrel{i}{\longrightarrow}B\stackrel{p}{\longrightarrow}C\longrightarrow 1$ and presentations $C=\left\langle X\,|\,R\right\rangle$ and $A=\left\langle Y\,|\,S\right\rangle$, then $B=\bigl\langle \widetilde{X},\widetilde{Y}\,|\,\widetilde{S},\widetilde{R},\widetilde{T}\bigr\rangle$, where: 
\begin{align*}
\widetilde{X}&=\left\{\widetilde{x}\,:\, p(x)\in X\right\},\;\widetilde{Y}=\left\{\widetilde{y}=i(y)\,:\, y\in Y\right\}\\
\widetilde{S}&=\left\{i(s)\,:\, s\in S\right\},\;\widetilde{R}=\Bigl\{\widetilde{r}v^{-1}_{r}\,:\, r\in R, v_{r}\in\bigl\langle \widetilde{Y}\bigr\rangle\Bigr\}\\
\widetilde{T}&=\Bigl\{\widetilde{x}^{-1}\widetilde{y}\,\widetilde{x}w^{-1}_{x,y}\,:\,x\in X, y\in Y, w_{x,y}\in\bigl\langle \widetilde{Y}\bigr\rangle\Bigr\}.
\end{align*}
The details of the proofs of Theorems~\ref{puras} and~\ref{total} are left to the reader.

\begin{remark}\label{Ab} Using Theorem~\ref{total}, it is straightforward to check that:
\begin{align*}
B_{n}(\mathbb{T})^{\text{Ab}} &\cong\mathbb Z\oplus\mathbb Z\oplus\mathbb Z_{2} =\left\langle \overline{a},\overline{b},\sigma\,:\,[\overline{a},\overline{b}]=[\overline{a},\sigma]=[\overline{b},\sigma]=\sigma^{2}=1\right\rangle\\
B_{n}(\mathbb{K})^{\text{Ab}} &\cong\mathbb Z\oplus\mathbb Z_{2}\oplus\mathbb Z_{2} =\left\langle \overline{a},\overline{b},\sigma\,:\,[\overline{a},\overline{b}]=[\overline{a},\sigma]=[\overline{b},\sigma]=\sigma^{2}=\overline{a}^{2}=1\right\rangle,
\end{align*}
for all $n\geq 2$, where $\overline{a}$ (resp.\ $\overline{b}$, $\sigma$) represents the $\Gamma_{2}$-coset of $a$ (resp.\ of $b$, $\sigma_{1}$).
\end{remark}

For compact non-orientable surfaces of genus $g\geq 3$ without boundary, we shall make use of the following presentation of their braid groups due to Bellingeri.

\begin{teo}[\cite{Be}]\label{Bel} Let $N_{g}$ be a compact, connected non-orientable surface of genus $g\geq 3$ without boundary. The braid group $B_{n}(N_{g})$ admits the following presentation:

generators: $\sigma_{1},\ldots,\sigma_{n-1},a_{1},\ldots,a_{g}$.

relations:
\begin{enumerate}
	
	\item $\sigma_{i}\sigma_{i+1}\sigma_{i}=\sigma_{i+1}\sigma_{i}\sigma_{i+1}$;
	
\item $\sigma_{j}\sigma_{i}=\sigma_{i}\sigma_{j}$, if $|i-j|\geq 2$;
	
	\item $a_{r}\sigma_{i}=\sigma_{i}a_{r}\,(1\leq r \leq g;\, i\neq1)$;
	
	\item $\sigma^{-1}_{1}a_{r}\sigma^{-1}_{1}a_{r}=a_{r}\sigma^{-1}_{1}a_{r}\sigma_{1}\, (1\leq r\leq g)$;
	
	\item $\sigma^{-1}_{1}a_{s}\sigma_{1}a_{r}=a_{r}\sigma^{-1}_{1}a_{s}\sigma_{1}\, (1\leq s<r\leq g)$;
	
	\item $a^{2}_{1}\cdots a^{2}_{g}=\sigma_{1}\sigma_{2}\cdots\sigma^{2}_{n-1}\cdots\sigma_{2}\sigma_{1}$.
	
\end{enumerate}
\end{teo}

To prove some of our results, we will also require the following theorem of Gruenberg.
\begin{teo}[\cite{G}]\label{TG} Let $\mathcal P$ denote one of the following classes:
\begin{enumerate}
	\item the class of soluble groups;
	\item the class of finite groups;
	\item the class of $p$-finite groups for a given prime number $p$.
\end{enumerate}
Let $K$ and $H$ be groups, and suppose that $K$ is $\mathcal P$ and that $H$ is residually $\mathcal P$. Then, for any group extension \[1\longrightarrow H\longrightarrow G\longrightarrow K\longrightarrow1,\] the group $G$ is residually $\mathcal P$.
\end{teo}

\section{The lower central and derived series of semi-direct products}\label{sec:lcssemi}

The main aim of this section is to establish the general decomposition of the lower central series and an estimate of the derived series of an arbitrary semi-direct product given in the statement of Theorem~\ref{serie}, which will be used in later computations of the lower central and derived series of $P_{n}(\mathbb{K})$. We first prove two lemmas that will be used in what follows.
If $x_{1},\ldots,x_{n}$ are elements of a group $G$, we set:
\[[x_{1},x_{2},\ldots,x_{n-1},x_{n}]=\Bigl[x_{1},\bigl[x_{2},\ldots,[x_{n-1},x_{n}]\bigr]\Bigr],\]
and if $X$ is a subset of $G$ then we denote the normal closure of $X$ in $G$ by $\left\langle\! \left\langle X\right\rangle\! \right\rangle_{G}$.

\begin{lemma}\label{colchete} Let $G$ be a group, and let $x,y\in G$. For all $n\in \mathbb{N}$, we have:
\begin{equation}\label{itercomm}
[x^{2^{n}},y]=[x,x,x^{2},\ldots,x^{2^{n-1}},y].[x,x^{2},\ldots,x^{2^{n-1}},y]^{2}.[x^{2},\ldots,x^{2^{n-1}},y]^{2}\cdots[x^{2^{n-1}},y]^{2}.
\end{equation}
\end{lemma}


\begin{proof}
We prove the lemma by induction on $n$. Observe that:
\begin{equation}\label{itercomm1}
[x^{2},y]=x.x.y.x^{-1}.x^{-1}.y^{-1}=x[x,y]yx^{-1}y^{-1}=x[x,y]x^{-1}[x,y]=[x,x,y].[x,y]^{2},
\end{equation}
which proves~(\ref{itercomm}) in the case $n=1$.
Now let $n\geq 2$, and suppose that the result holds for all $1\leq i\leq n$. Applying~(\ref{itercomm1}) to the elements $x^{2^{n}}$ and $[x^{2^{n}},y]$, we have:
\[[x^{2^{n+1}},y]=[(x^{2^{n}})^{2},y]=[x^{2^{n}},x^{2^{n}},y][x^{2^{n}},y]^{2},\]
and applying~(\ref{itercomm1}) to the elements $x^{2^{n}}$ and $y$, we obtain:
\[\left[x^{2^{n}},[x^{2^{n}},y]\right]=\left[x,x,x^{2},\ldots,x^{2^{n-1}},[x^{2^{n}},y]\right]\left[x,x^{2},\ldots,x^{2^{n-1}},[x^{2^{n}},y]\right]^{2}\cdots\left[x^{2^{n-1}},[x^{2^{n}},y]\right]^{2}.\]
Thus:
\[\bigl[x^{2^{n+1}},y\bigr]=\bigl[x,x,x^{2}\!,\ldots,x^{2^{n-1}}\!\!,[x^{2^{n}}\!\!,y]\bigr]\bigl[x,x^{2},\ldots,x^{2^{n-1}},[x^{2^{n}}\!\!,y]\bigr]^{2}\cdots\bigl[x^{2^{n-1}},[x^{2^{n}}\!\!,y]\bigr]^{2}\bigl[x^{2^{n}}\!\!,y\bigr]^{2},\]
which completes the proof by induction.
\end{proof}

%
%

\begin{remark}\label{gerador}
With the notation of Theorem~\ref{serie},  In what follows, for the groups $K_{n}$, $H_{n}$ or $\widetilde{H}_{n}$, we will use the word \emph{generator} to mean a word of the form $\varphi(g)(h).h^{-1}$, where $g\in \Gamma_{n-1}(G)$ and $h\in H$, $g\in G$ and $h\in V_{n-1}$, or $g\in \Gamma_{n-1}(G)$ and $h\in H$ respectively. Similarly, a generator of the group $L_{n}$ (resp.\ $V_{n}$) will mean either a generator of $K_{n}$ or $H_{n}$, or an element of the form $[h,l]$, where $h\in H$ and $l\in L_{n-1}$ (resp.\ either a generator of $\widetilde{H}_{n}$, or an element of the form $[h,v]$, where $h\in H$ and $v\in V_{n-1}$) 
\end{remark}

\begin{lem}\label{Lnnormal}
Let $n\geq 2$. With the notation of Theorem~\ref{serie}, the subgroups $K_{n}$, $L_{n}$ and $V_{n}$ are normal in $H$ for all $n\geq 2$, and we have the inclusions $K_{n+1}\subset K_{n}$, $H_{n+1}\subset H_{n}$, $\widetilde{H}_{n+1}\subset \widetilde{H}_{n}$, $L_{n+1}\subset L_{n}$ and $V_{n+1}\subset V_{n}$.
\end{lem}

\begin{proof}
The proof is by induction on $n$. The proof 
in the case $n=2$ was given in~\cite[Proposition~3.3]{GG}. So suppose that $n\geq 2$, and assume that $L_{n}$ (resp.\ $V_{n}$) is a normal subgroup of $H$, let $x\in L_{n+1}$ (resp.\ $V_{n+1}$) and let $h\in H$. Since $L_{n+1}=\bigl\langle K_{n+1},H_{n+1},[H,L_{n}]\bigr\rangle$ (resp.\ $V_{n+1}=\bigl\langle \widetilde{H}_{n+1},[H,V_{n}]\bigr\rangle$), it suffices to show that $hxh^{-1}\in L_{n+1}$, where $x$ is a generator of  $K_{n+1}, H_{n+1}$ or $[H,L_{n}]$ (resp.\ of $\widetilde{H}_{n+1}$ or $[H,V_{n}]$), as in Remark \ref{gerador}.

\begin{itemize}
	\item Suppose that $x=\varphi(g)(y)y^{-1}\in K_{n+1}$, where $g\in\Gamma_{n}(G)$ and $y\in H$. Then $\varphi(g)\in \operatorname{\text{Aut}}(H)$, and there exists $h'\in H$ such that $\varphi(g)(h')=h$,  so: \[hxh^{-1}=h(\varphi(g)(y).y^{-1})h^{-1}=(\varphi(g)(h'y)y^{-1}h'^{-1})(\varphi(g)(h')h'^{-1})^{-1}\in K_{n+1}\subset L_{n+1}.\] This also implies that $K_{n}$ is a normal subgroup of $H$ for all $n\geq 2$.
	
	\item Suppose that $x=\varphi(g)(y)y^{-1}$ is an element of $H_{n+1}$ (resp.\ of $\widetilde{H}_{n+1}$), where $g\in G$ and $y\in L_{n}$ (resp.\ $y \in V_{n}$), and let $h'\in H$ be such that $\varphi(g)(h')=h$. Then: 
\begin{align*}
hxh^{-1} &=h(\varphi(g)(y).y^{-1})h^{-1}\\
& \text{$=(\varphi(g)(h'yh'^{-1}).(h'yh'^{-1})^{-1})[h',y][y,\varphi(g)(h')]\in L_{n+1}$ (resp.\ $V_{n+1}$),}
\end{align*}
because $h'yh'^{-1}\in L_{n}$ (resp.\ $V_{n}$) by the normality of $L_{n}$ (resp.\ $V_{n}$) in $H$ using the induction hypothesis. 
	
	\item Suppose that $x=[y,l]\in[H,L_{n}]$ (resp.\ $[H,V_{n}]$), where $y\in H$ and $l\in L_{n}$ (resp.\ $l\in V_{n}$). Then: \[hxh^{-1}=[hyh^{-1},hlh^{-1}]\in [H,L_{n}]\subset L_{n+1} (\mbox{resp.}\,\,[H,V_{n}]\subset V_{n+1}),\] because $hlh^{-1}\in L_{n}$ (resp.\ $V_{n}$) by the normality of $L_{n}$ (resp.\ $V_{n}$) in $H$.
\end{itemize}
This proves that $L_{n}$ (resp.\ $V_{n}$) is a normal subgroup of $H$ for all $n\geq 2$.

To prove the second part, notice that the inclusion $\Gamma_{n}(G)\subset \Gamma_{n-1}(G)$ implies that  $K_{n+1}\subset K_{n}$ for all $n\geq2$. It is easy to see that $H_{3}\subset H_{2}$ (resp. $\widetilde{H}_{3}\subset \widetilde{H}_{2}$) because $L_{2}\subset H$ (resp.\ $V_{2}\subset H$). By induction, suppose that $H_{n}\subset H_{n-1}$ (resp.\ $\widetilde{H}_{n}\subset \widetilde{H}_{n-1}$) for some $n\geq3$. Since $L_{n-1}$ (resp.\ $V_{n-1}$) is normal in $H$, we have $[H,L_{n-1}]\subset L_{n-1}$ (resp.\ $[H,V_{n-1}]\subset V_{n-1}$). Further, using the definitions and the induction hypothesis, we have the inclusions $K_{n}\subset K_{n-1}\subset L_{n-1}$ and $H_{n} \subset H_{n-1}\subset L_{n-1}$ (resp.\ $\widetilde{H}_{n} \subset \widetilde{H}_{n-1}\subset V_{n-1}$). It follows that $L_{n}\subset L_{n-1}$ (resp.\ $V_{n}\subset V_{n-1}$), and then that $H_{n+1}\subset H_{n}$ (resp.\ $\widetilde{H}_{n+1}\subset \widetilde{H}_{n}$). Consequently, $L_{n+1}\subset L_{n}$ and $V_{n+1}\subset V_{n}$ for all $n\geq 2$.
\end{proof}


\begin{proof}[Proof of Theorem~\ref{serie}]
The proof is by induction on $n$. The case $n=1$ is trivial. If $n=2$, part~(\ref{it:lcspart1}) was proved in~\cite[Proposition~3.3]{GG}, and part~(\ref{it:lcspart2}) follows from part~(\ref{it:lcspart1}) and the fact that $L_{2}=V_{2}$. Now suppose that parts~(\ref{it:lcspart1}) and~(\ref{it:lcspart2}) hold for some $n\geq 2$, and let us prove the result for $n+1$. Let $\varphi\colon\thinspace \Gamma_{n}(G)\longrightarrow \operatorname{\text{Aut}}(L_{n})$ be the action (also denoted by $\varphi$) induced by $\varphi$ such that $L_{n}\rtimes_{\varphi} \Gamma_{n}(G)=\Gamma_{n}(H\rtimes_{\varphi}G)$. We claim that $\varphi$ also induces an action $\varphi\colon\thinspace \Gamma_{n+1}(G)\longrightarrow \operatorname{\text{Aut}}(L_{n+1})$. To see this, let $g\in \Gamma_{n+1}(G)$. To prove that $\varphi(g)(L_{n+1})\subset L_{n+1}$, it suffices to take $x\in L_{n+1}$ to be of the form $x=\varphi(g')(h).h^{-1}$, where either $g'\in \Gamma_{n}(G)$ and $h\in H$, or $g'\in G$ and $h\in L_{n}$, or of the form $x=[h,l]\in[H,L_{n}]$, where $h\in H$ and $l\in L_{n}$. The result will then follow for all elements of $L_{n+1}$ because $\varphi(g)$ is a homomorphism.

\begin{itemize}
	\item If $x=\varphi(g')(h).h^{-1}\in K_{n+1}$, where $g'\in \Gamma_{n}(G)$ and $h\in H$, or $x=\varphi(g')(h).h^{-1}\in H_{n+1}$, where $g'\in G$ and $h\in L_{n}$, then: \[\varphi(g)(x)=\varphi(g)\left(\varphi(g')(h).h^{-1}\right)=\left(\varphi(gg')(h).h^{-1}\right)\left(\varphi(g)(h).h^{-1}\right)^{-1}.\]
If $g'\in \Gamma_{n}(G)$ and $h\in H$ then $\varphi(g)(x) \in K_{n+1}\subset L_{n+1}$ since $g$ and $gg'$ belong to $\Gamma_{n}(G)$. If $g'\in G$ and $h\in L_{n}$, then $\varphi(g)(x) \in H_{n+1}\subset L_{n+1}$ because $h\in L_{n+1}$.
		
	\item If $x=[h,l]\in[H,L_{n}]$, where $h\in H$ and $l\in L_{n}$ then:\[\varphi(g)(x)=[\varphi(g)(h),\varphi(g)(l)]\in [H,L_{n}],\] since $g\in \Gamma_{n+1}(G)\subset \Gamma_{n}(G)$, so $\varphi(g)(l) \in L_{n}$.
\end{itemize}

Since $\varphi(g)\colon\thinspace L_{n+1}\longrightarrow L_{n+1}$ is the restriction of an automorphism, it is injective, so to show that it is an automorphism, it suffices to prove surjectivity. We first consider the following two cases:
\begin{enumerate}[(a)]
	\item\label{it:surja} If $x=\varphi(g')(h).h^{-1}$, where either $g'\in \Gamma_{n}(G)$ and $h\in H$, or $g'\in G$ and $h\in L_{n}$, let: \[y=\left(\varphi(g^{-1}g')(h).h^{-1}\right)h\left(\varphi(g^{-1})(h^{-1}).h\right)h^{-1}\in L_{n+1},\] because $L_{n+1}\triangleleft H$, and one may check that $\varphi(g)(y)=x$.
		
	\item\label{it:surjb} If $x=[h,l]\in[H,L_{n}]$, where $h\in H$ and $l\in L_{n}$, there exist $l'\in L_{n}$ and $h'\in H$ such that $\varphi(l')=l$ and $\varphi(g)(h')=h$ by the induction hypothesis. Taking $y=[h',l']\in[H,L_{n}]\subset L_{n+1}$, we see that $\varphi(g)(y)=x$.
\end{enumerate}

This shows that if $x$ is a generator of $K_{n+1}, H_{n+1}$ or $[H,L_{n}]$, there exists $y\in L_{n+1}$ such that $\varphi(g)(y)=x$. Given an arbitrary element $x\in L_{n+1}$, there exist $x_{1},\ldots,x_{s}$, each of which satisfies one of the conditions of cases~(\ref{it:surja}) and~(\ref{it:surjb}) above, such that 
$x=x_{1}\cdots x_{s}$. So for $i=1,\ldots,s$, there exists $y_{i}\in L_{n+1}$ such that $\varphi(g)(y_{i})=x_{i}$, and we have $\varphi(g)(y_{1}\cdots y_{s})=x$, which proves the surjectivity of $\varphi(g)\colon\thinspace L_{n+1}\longrightarrow L_{n+1}$. Therefore the semi-direct product $L_{n+1}\rtimes_{\varphi} \Gamma_{n+1}(G)$ is well defined. Similar computations show that the same is true for the semi-direct product $V_{n+1}\rtimes_{\varphi} (G)^{(n)}$. 

To complete the proof of part~(\ref{it:lcspart1}) of Theorem~\ref{serie}, it remains to show that $L_{n+1}\rtimes_{\varphi} \Gamma_{n+1}(G)=\Gamma_{n+1}(H\rtimes_{\varphi}G)$. We first prove that $L_{n+1}\rtimes_{\varphi} \Gamma_{n+1}(G)\subset \Gamma_{n+1}(H\rtimes_{\varphi}G)$. Let $(x,g)\in L_{n+1}\rtimes_{\varphi} \Gamma_{n+1}(G)$, where $x\in L_{n+1}$ and $g\in \Gamma_{n+1}(G)$. Since $(x,g)=(x,1)(1,g)$, it suffices to show that $(x,1)$ and $(1,g)$ belong to $\Gamma_{n+1}(H\rtimes_{\varphi}G)$. Clearly, $(1,g)\in\Gamma_{n+1}(H\rtimes_{\varphi}G)$. Further, $(x,1)$ is a product of elements each of which is of one of the following forms:
\begin{itemize}
	\item $(\varphi(g)(h).h^{-1},1)=[(1,g),(h,1)]$, where $g\in \Gamma_{n}(G),\, h\in H$, and $(1,g)\in \Gamma_{n}(H\rtimes_{\varphi}G)$. Then $(\varphi(g)(h).h^{-1},1)\in\Gamma_{n+1}(H\rtimes_{\varphi}G)$.
	
	\item $(\varphi(g)(h).h^{-1},1)=[(1,g),(h,1)]$, where $g\in G$ and $h\in L_{n}$. Then $(h,1)\in L_{n}\rtimes_{\varphi}\Gamma_{n}(G)= \Gamma_{n}(H\rtimes_{\varphi}G)$ by the induction hypothesis, and $(\varphi(g)(h).h^{-1},1)\in\Gamma_{n+1}(H\rtimes_{\varphi}G)$.
	
	\item $([h,l],1)\in[H,L_{n}]$, where $h\in H$ and $l\in L_{n}$. Then $([h,l],1)=[(h,1),(l,1)]$, and $l\in L_{n}\rtimes_{\varphi}\Gamma_{n}(G)= \Gamma_{n}(H\rtimes_{\varphi}G)$ by the induction hypothesis, so $([h,l],1)\in\Gamma_{n+1}(H\rtimes_{\varphi}G)$.
\end{itemize}
Since all of these elements belong to $\Gamma_{n+1}(H\rtimes_{\varphi}G)$, it follows  that $(x,1)\in \Gamma_{n+1}(H\rtimes_{\varphi}G)$, whence $L_{n+1}\rtimes_{\varphi} \Gamma_{n+1}(G)\subset\Gamma_{n+1}(H\rtimes_{\varphi}G)$.

For the other inclusion, let $[(h,g),(x,y)]\in \Gamma_{n+1}(H\rtimes_{\varphi}G)$, where $(h,g)\in H\rtimes_{\varphi}G$ and $(x,y)\in\Gamma_{n}(H\rtimes_{\varphi}G)$. By the induction hypothesis, $\Gamma_{n}(H\rtimes_{\varphi}G)= L_{n}\rtimes_{\varphi}\Gamma_{n}(G)$, so $x\in L_{n}$ and $y\in \Gamma_{n}(G)$, and thus: 
\begin{equation}\label{eq:hgxy}
[(h,g),(x,y)]=(h.\varphi(g)(x).\varphi(gyg^{-1})(h^{-1}).\varphi([g,y])(x^{-1}),[g,y]).
\end{equation} The second factor $[g,y]$ on the right-hand side of~(\ref{eq:hgxy}) belongs to $\Gamma_{n+1}(H\rtimes_{\varphi}G)$, and the first factor, denoted by $\rho$, may be written in the following form: 
\begin{equation*}
\rho = [h,x].xhx^{-1}\!\!\left(\varphi(g)(x).x^{-1}\right)\!xh^{-1}x^{-1}.xh\!\left(\varphi(gyg^{-1})(h^{-1}).h\right)\!h^{-1}x^{-1}.
x\!\left(\varphi([g,y])(x^{-1}).x\right)\!x^{-1}.
\end{equation*}
Note that: 
\begin{itemize}
	\item $[h,x]\in [H,L_{n}]\subseteq L_{n+1}$, since $h\in H$ and $x\in L_{n}$.

\item $\varphi(g)(x).x^{-1} \in H_{n+1}\subseteq L_{n+1}$, since $x\in L_{n}$.

\item $\varphi(gyg^{-1})(h^{-1}).h \in K_{n+1}\subseteq L_{n+1}$, since $y\in \Gamma_{n}(G)$, so $gyg^{-1}\in \Gamma_{n}(G)$ because $\Gamma_{n}(G)$ is a normal subgroup of $G$.

\item $\varphi([g,y])(x^{-1}).x \in H_{n+1}\cap K_{n+1}\subseteq L_{n+1}$, since $x\in L_{n}$ and $[g,y]\in \Gamma_{n}(G)$.

\end{itemize}
By Lemma~\ref{Lnnormal}, the conjugates by elements of $H$ of the elements $[h,x]$, $\varphi(g)(x).x^{-1}$, $\varphi(gyg^{-1})(h^{-1}).h$ and $\varphi([g,y])(x^{-1}).x$ also belong to $L_{n+1}$, therefore $\rho\in L_{n+1}$ as required.
This proves part~(\ref{it:lcspart1}) of the statement.

To prove part~(\ref{it:lcspart2}), suppose by induction that 
$(H\rtimes_{\varphi}G)^{(n-1)}\subset V_{n}\rtimes_{\varphi}G^{(n-1)}$.
Then:
\[(H\rtimes_{\varphi}G)^{(n)}=[(H\rtimes_{\varphi}G)^{(n-1)},(H\rtimes_{\varphi}G)^{(n-1)}]\subset \left[V_{n}\rtimes_{\varphi}G^{(n-1)},V_{n}\rtimes_{\varphi}G^{(n-1)}\right].\]
To show that $\left[V_{n}\rtimes_{\varphi}G^{(n-1)},V_{n}\rtimes_{\varphi}G^{(n-1)}\right]\subset V_{n+1}\rtimes_{\varphi}G^{(n)}$, let $(h,g),(x,y)\in V_{n}\rtimes_{\varphi}G^{(n-1)}$. Then:
\begin{itemize}
\item $[h,x]\in [H,V_{n}]\subseteq V_{n+1}$ because $h,x \in V_{n}\subset H$.

\item the three elements $\varphi(g)(x).x^{-1}$, $\varphi(gyg^{-1})(h^{-1}).h$ and $\varphi([g,y])(x^{-1}).x$ belong to $\widetilde{H}_{n+1}$ because $h,x\in V_{n}$, so they belong to $V_{n+1}$.
\end{itemize}
Arguing in a manner similar to that for part~(\ref{it:lcspart1}) from~(\ref{eq:hgxy}) onwards, it follows that $[(h,g),(x,y)]\in V_{n+1}\rtimes_{\varphi}G^{(n)}$ as required.
\end{proof}

The following lemma will help us to simplify some of the calculations in the following sections.
 
\begin{lem}\label{geradoreslema} With the notation of Theorem~\ref{serie}, let $\widetilde{G}$ be a subgroup of $G$, let $\widetilde{H}$ be a subgroup of $H$, and let $X$ (resp.\ $Y$) be a generating set of $\widetilde{G}$ (resp.\ $\widetilde{H}$).

\begin{enumerate}
\item\label{geradoreslema1} The subgroup\[\bigl\langle \varphi(g)(h).h^{-1}\,:\,g\in \widetilde{G},\,h\in \widetilde{H}\bigr\rangle\] is contained in the normal closure $\bigl\langle \!\bigl\langle Z\bigr\rangle\! \bigr\rangle_{H}$ of  \[Z=\left\{\varphi(g)(h).h^{-1}\,:\,g\in X,\,h\in Y\right\}\] in $H$. In particular, if this subgroup is a normal subgroup of $G$, it is equal to $\bigl\langle \!\bigl\langle Z\bigr\rangle\! \bigr\rangle_{H}$. Consequently, if $X$ (resp.\ $Y$) is a generating set of $\Gamma_{n-1}(G)$ (resp.\ of $H$) then to calculate the subgroup $K_{n}$, it suffices to compute the elements $\varphi(g)(h).h^{-1}$, where $g\in X$ and $h\in Y$.
	
\item\label{geradoreslema2} Let $W$ be a subset of $H$ such that $L_{n}=\bigl\langle \!\bigl\langle  W\bigr\rangle\! \bigr\rangle_{H}$ (resp.\ $V_{n}=\bigl\langle\! \bigl\langle W\bigr\rangle\!\bigr\rangle_{H}$) is the normal closure of $W$ in $H$. Let $X$ (resp. $Y$) be a generating set of $G$ (resp.\ of $H$). Then up to conjugacy and elements of $[H,L_{n}]$ (resp.\ of $[H,V_{n}]$), $H_{n+1}$ (resp.\ $\widetilde{H}_{n+1}$) is generated by the set 
\[\left\{ \varphi(g)(w).w^{-1}\,:\,g\in X,\, w\in W\right\}.\] Therefore: 
\begin{align*}
L_{n+1}&=\bigl\langle\! \bigl\langle K_{n+1}, \varphi(g)(w).w^{-1}, [h,w]\,: \,g\in X,\,h\in Y,\,w\in W \bigr\rangle\! \bigr\rangle_{H},\\
V_{n+1}&=\bigl\langle \!\bigl\langle \varphi(g)(w).w^{-1}, [h,w]\,:\,g\in X,\,h\in Y,\,w\in W \bigr\rangle\! \bigr\rangle_{H}.
\end{align*}
\end{enumerate}
\end{lem}

\begin{remark} With the notation of Lemma~\ref{geradoreslema}(\ref{geradoreslema1}), we will say that the elements of $Z$ are generators of the subgroup $\bigl\langle \!\bigl\langle Z\bigr\rangle\! \bigr\rangle_{H}$.
\end{remark}

\begin{proof}[Proof of Lemma~\ref{geradoreslema}]\mbox{}
\begin{enumerate}
\item\label{it:geradoreslema1pf} To prove the first part of the statement, note that it suffices to prove the result for elements of the subgroup of the form $\varphi(g)(h).h^{-1}$, where $g\in \widetilde{G}$ and $h\in \widetilde{H}$. If $g\in G$, there exist $g_{1},\ldots, g_{p}\in \widetilde{G}$ and $\epsilon_{1},\ldots, \epsilon_{p},\in \{ 1,-1\}$ such that $g_{i}^{\epsilon_{i}}\in X$ for all $i=1,\ldots,p$ and $g=g_{1}^{\epsilon_{1}}\cdots g_{p}^{\epsilon_{p}}$. Now:
\begin{align}
\varphi(g)(h).h^{-1}&=\prod^{p}_{i=1} \varphi(g_{i}^{\epsilon_{i}})\Bigl(\varphi\bigl(g_{i+1}^{\epsilon_{i+1}}\cdots g_{p}^{\epsilon_{p}}\bigr)(h)\Bigr).\Bigl(\varphi\bigl(g_{i+1}^{\epsilon_{i+1}}\cdots g_{p}^{\epsilon_{p}}\bigr)(h)\Bigr)^{-1}\notag\\
&= \prod^{p}_{i=1} \varphi(g_{i}^{\epsilon_{i}})(h_{i}').h_{i}'^{-1},\label{eq:varphigh1}
\end{align}
where for all $i=1,\ldots,p$, $h_{i}'=\varphi\bigl(g_{i+1}^{\epsilon_{i+1}}\cdots g_{p}^{\epsilon_{p}}\bigr)(h)$. Further, for all $h'\in \widetilde{H}$, there exist $h_{1},\ldots, h_{q}\in \widetilde{H}$ and $\delta_{1},\ldots,\delta_{q} \in \{ 1,-1\}$ such that $h_{j}^{\delta_{j}}\in Y$ for all $j=1,\ldots,q$ and $h'=h_{1}^{\delta_{1}}\cdots h_{q}^{\delta_{q}}$. Since
\begin{equation}\label{eq:varphigh2}
\varphi(g')(h').h^{-1}=\prod^{q}_{j=1}\Bigl( h_{1}^{\delta_{1}}\cdots h_{j-1}^{\delta_{j-1}}\bigl(\varphi(g')(h_{j}^{\delta_{j}}).h^{-\delta_{j}}_{j}\bigr)h^{-\delta_{j-1}}_{j-1}\cdots h^{-\delta_{1}}_{1}\Bigr)
\end{equation}
for all $g'\in G$, the first part of the statement follows by combining~(\ref{eq:varphigh1}) and~(\ref{eq:varphigh2}). The second and third parts are consequences of the first part.


\item Let $\varphi(g)(h).h^{-1}\in H_{n+1}$ (resp.\ $\widetilde{H}_{n+1}$), where $g\in G$ and $h\in L_{n}$ (resp.\ $V_{n}$). As in~(\ref{it:geradoreslema1pf}) above,~(\ref{eq:varphigh1}) holds. For all $h'\in L_{n}$ (resp.\ $V_{n}$), there exist $x_{1}, \ldots, x_{q} \in W$, $\delta_{1}, \ldots, \delta_{q}\in \left\{1,-1\right\}$ and $\alpha_{1},\ldots, \alpha_{q} \in H$, such that $x^{\delta_{j}}_{j}\in W$ and $h'=\alpha_{1}x^{\delta_{1}}_{1}\alpha^{-1}_{1}\cdots \alpha_{q}x^{\delta_{q}}_{q}\alpha^{-1}_{q}$. Then we obtain an equation similar to~(\ref{eq:varphigh2}), where for all $j=1,\ldots, q$, $h^{\delta_{j}}_{j}$ is replaced by $\alpha_{j}x^{\delta_{j}}_{j}\alpha^{-1}_{j}$. Further, for all $j=1,\ldots,q$, $\varphi(g')(\alpha_{j} x^{\delta_{j}}_{j} \alpha^{-1}_{j}).(\alpha_{j} x^{\delta_{j}}_{j} \alpha^{-1}_{j})^{-1}$ is equal to:
\begin{equation}\label{eq:varphigh3}
\varphi(g')(\alpha_{j})\bigl(\varphi(g')(x^{\delta_{j}}_{j}).x^{-\delta_{j}}_{j}\bigr)\varphi(g')(\alpha^{-1}_{j})\cdot \alpha_{j}\underbrace{\bigl[\alpha^{-1}_{j}\varphi(g')(\alpha_{j}),x^{\delta_{j}}_{j}\bigr]}_{\in[H,L_{n}]}\alpha^{-1}_{j},
\end{equation}
Part~(\ref{geradoreslema1}) then follows from~(\ref{eq:varphigh1}),~(\ref{eq:varphigh2}) and~(\ref{eq:varphigh3}).\qedhere
\end{enumerate}
\end{proof}

\section{The case of the torus}\label{sec:torus}

In this section, we study the derived series of $B_{n}(\mathbb{T})$, the aim being to prove Theorem~\ref{toro}. We shall consider two cases, $n\leq 4$ and $n\geq 5$.

\begin{prop}\label{toro1}
If $n\leq4$ then $B_{n}(\mathbb{T})$ is residually soluble.
\end{prop}

\begin{proof} If $n\leq4$, the result follows by using the short exact sequence~(\ref{seq1}), Theorem~\ref{TG}, the solubility of $S_{n}$ if $n\leq4$, and the fact that $P_{n}(\mathbb{T})$ is residually soluble for all $n\geq1$~\cite[Theorem~4]{BGG}.\end{proof} 

To study the case $n\geq5$, we start by exhibiting a presentation of $(B_{n}(\mathbb{T}))^{(1)}$.
\begin{prop}\label{Bn(T)1} A presentation of $(B_{n}(\mathbb{T}))^{(1)}$ is given by:

generators: for $k,m\in \mathbb Z$ and $i=1,\ldots,n-1$:
\begin{itemize}
	\item $b_{k,m} = b^{k}a^mba^{-m}b^{-k-1}$;
	
	\item $d_{k,m} = b^ka^m\sigma_{1}b\sigma^{-1}_{1}a^{-m}b^{-1-k}$;
	
	\item $a_{k,m}=b^ka^m(\sigma_{1}a\sigma^{-1}_{1}a^{-1})a^{-m}b^{-k}$;
	
	\item $\theta_{i,k,m}=b^ka^m\sigma_{i}\sigma^{-1}_{1}a^{-m}b^{-k}$;
		
	\item $\rho_{i,k,m} = b^ka^m\sigma_{1}\sigma_{i}a^{-m}b^{-k}$;
\end{itemize}

relations:
\begin{enumerate}
	\item $\left\{\begin{array}{l} \theta_{i,k,m}\rho_{i+1,k,m}\theta_{i,k,m}=\theta_{i+1,k,m}\rho_{i,k,m}\theta_{i+1,k,m}\\
	\rho_{i,k,m}\theta_{i+1,k,m}\rho_{i,k,m}=\rho_{i+1,k,m}\theta_{i,k,m}\rho_{i+1,k,m}\end{array}\right.$

	\item $\left\{\begin{array}{l}\theta_{i,k,m}\rho_{j,k,m}=\theta_{j,k,m}\rho_{i,k,m}\\ \rho_{i,k,m}\theta_{j,k,m}=\rho_{j,k,m}\theta_{i,k,m}\end{array}\right.$ if $\left|i-j\right|\geq 2$.

  \item $\left\{\begin{array}{l} a_{k,m}=\theta^{-1}_{j,k,m}\theta_{j,k,m+1}\\ a_{k,m}=\rho_{j,k,m}\rho^{-1}_{j,k,m+1}\end{array}\right.$ for $j\geq 2$.
	
	\item $\left\{\begin{array}{l} b_{k,m}\theta_{j,k+1,m}=\theta_{j,k,m}d_{k,m}\\ d_{k,m}\rho_{j,k+1,m}=\rho_{j,k,m}b_{k,m}\end{array}\right.$ for $j\geq 2$.

	\item $\left\{\begin{array}{l} b^{-1}_{k-1,m}a_{k-1,m}b_{k-1,m+1}\rho^{-1}_{1,k,m+1}a^{-1}_{k,m}=1\\
	d^{-1}_{k-1,m}\rho_{1,k-1,m}\rho^{-1}_{1,k-1,m+1}d_{k-1,m+1}\rho^{-1}_{1,k,m}=1\end{array}\right.$

	\item $\left\{ \begin{array}{l} a_{k,m+1}\rho_{1,k,m+2}= a_{k,m}\rho_{1,k,m+1}\\
		\rho_{1,k,m}a_{k,m+1}=a_{k,m}\rho_{1,k,m+1}\end{array}\right.$
	
	\item $\left\{\begin{array}{l} b_{k,m}\rho^{-1}_{1,k+1,m}d_{k+1,m}=\rho^{-1}_{1,k,m}d_{k,m}b_{k+1,m}\\
	b_{k,m}\rho^{-1}_{1,k+1,m}d_{k+1,m}=d_{k,m}b_{k+1,m}\rho^{-1}_{1,k+2,m}\end{array}\right.$

	\item if $n$ is odd:	
	
\noindent $\left\{\begin{array}{l} \theta_{1,k,m}\rho_{2,k,m}\theta_{3,k,m}\cdots\rho_{n-1,k,m}\theta_{n-1,k,m}\cdots\rho_{3,k,m}\theta_{2,k,m}\rho_{1,k,m}=b_{k,m}b^{-1}_{k,m+1}\\
	\rho_{1,k,m}\theta_{2,k,m}\rho_{3,k,m}\cdots\theta_{n-1,k,m}\rho_{n-1,k,m}\cdots\theta_{3,k,m}\rho_{2,k,m}\theta_{1,k,m} = d_{k,m}a_{k+1,m}d^{-1}_{k,m+1}a^{-1}_{k,m} \end{array}\right.$
	
	\item if $n$ is even:
	
\noindent $\left\{\begin{array}{l} \theta_{1,k,m}\rho_{2,k,m}\theta_{3,k,m}\cdots\theta_{n-1,k,m}\rho_{n-1,k,m}\cdots\rho_{3,k,m}\theta_{2,k,m}\rho_{1,k,m}=b_{k,m}b^{-1}_{k,m+1}\\
	\rho_{1,k,m}\theta_{2,k,m}\rho_{3,k,m}\cdots\rho_{n-1,k,m}\theta_{n-1,k,m}\cdots\theta_{3,k,m}\rho_{2,k,m}\theta_{1,k,m} = d_{k,m}a_{k+1,m}d^{-1}_{k,m+1}a^{-1}_{k,m}.\end{array}\right.$
\end{enumerate}
\normalsize
\end{prop}

\begin{proof} One applies the Reidemeister-Schreier rewriting process~\cite[Appendix~1]{H} to the short exact sequence:
$$1\longrightarrow (B_{n}(\mathbb{T}))^{(1)}\longrightarrow B_{n}(\mathbb{T})\longrightarrow \underbrace{B_{n}(\mathbb{T})^{\text{Ab}}}_{\mathbb Z\oplus\mathbb Z\oplus\mathbb Z_{2}}\longrightarrow 1,$$
using the presentation of the group $B_{n}(\mathbb{T})$ given in Theorem~\ref{total}, and taking the Schreier transversal to be $\left\{b^{k}a^{m};\, b^{k}a^{m}\sigma_{1}\,:\, k,m \in \mathbb Z\right\}$. The details are left to the reader.\end{proof} 

\begin{prop}\label{toro2} If $n\geq5$, then $B_{n}(\mathbb{T})$ is not residually soluble. Moreover, $(B_{n}(\mathbb{T}))^{(2)}=(B_{n}(\mathbb{T}))^{(3)}$.
\end{prop}

Theorem~\ref{toro} then follows directly from Propositions~\ref{toro1} and~\ref{toro2}.

\medskip

\begin{proof}[Proof of Proposition~\ref{toro2}]
The first step is a standard procedure that may be found in~\cite[Theorem~1.4, p.~3389]{GG}, and uses just the Artin relations and some properties of the derived series. For future reference, we note that it may also be applied to the braid groups of non-orientable surfaces. If $M$ is a compact surface, consider the following short exact sequence: \[1\longrightarrow \frac{(B_{n}(M))^{(1)}}{(B_{n}(M))^{(2)}}\stackrel{i}{\longrightarrow} \frac{B_{n}(M)}{(B_{n}(M))^{(2)}}\stackrel{p}{\longrightarrow} B_{n}(M)^{\text{Ab}}\longrightarrow 1,\]
where $p$ is the canonical projection. Observe that for $i=1,\ldots,n-1$, the $(B_{n}(M))^{(2)}$-cosets of the $\sigma_{i}$ coincide in $B_{n}(M)/(B_{n}(M))^{(2)}$, and are equal to an element that we denote by $\sigma$. 

Now let $M=\mathbb{T}$. Using relations~(3) and~(4) of  Theorem~\ref{total}, the $(B_{n}(\mathbb{T}))^{(2)}$-cosets of $a$ and $b$ commute with $\sigma$ in $B_{n}(\mathbb{T})/(B_{n}(\mathbb{T}))^{(2)}$. Using this fact and relations~(5) and~(8) of Theorem~\ref{total}, it follows that $\sigma^{-2}=bab^{-1}a^{-1}$ and $\sigma^{2n}=1$, and so $\sigma$ has order at most $2n$. To show that the order of $\sigma$ in $B_{n}(\mathbb{T})/(B_{n}(\mathbb{T}))^{(2)}$ is exactly $2n$, using Proposition~\ref{Bn(T)1}, we note that: \begin{equation}\label{quociente}\frac{(B_{n}(\mathbb{T}))^{(1)}/(B_{n}(\mathbb{T}))^{(2)}}{\Theta}=\left\langle \rho_{1,0,0}\right\rangle\cong\mathbb Z_{n},\end{equation}
where $\Theta$ is the normal closure in $(B_{n}(\mathbb{T}))^{(1)}/(B_{n}(\mathbb{T}))^{(2)}$ of the $(B_{n}(\mathbb{T}))^{(2)}$-cosets of the elements of the set $\left\{\theta_{i,k,m},\, k,m\in\mathbb Z,\,i=1,\ldots,n-1\right\}$.
Let $q$ be the canonical projection of $(B_{n}(\mathbb{T}))^{(1)}/(B_{n}(\mathbb{T}))^{(2)}$ onto $\displaystyle \frac{((B_{n}(\mathbb{T}))^{(1)}/(B_{n}(\mathbb{T}))^{(2)})}{\Theta}$. The order of $\sigma$ in $B_{n}(\mathbb{T})/(B_{n}(\mathbb{T}))^{(2)}$ is even because $p(\sigma)$ is the generator of $\mathbb Z_{2}$. Suppose that the order of $\sigma$ is $2r$, where $r<n$. Then $i(\rho_{1,0,0})=\sigma^{2}$, and $i(\rho_{1,0,0}^{r})=\sigma^{2r}=1$. Since $i$ is injective, $\rho_{1,0,0}^{r}=1$, and it follows that $1=q(\rho_{1,0,0}^{r})=\rho_{1,0,0}^{r}$ in $\displaystyle \frac{({(B_{n}(\mathbb{T}))^{(1)}/(B_{n}(\mathbb{T}))^{(2)}})}{\Theta}$. Thus $\rho_{1,0,0}$ is of order $r<n$, which contradicts~(\ref{quociente}). Hence:
\[B_{n}(\mathbb{T})/(B_{n}(\mathbb{T}))^{(2)}=\left\langle \sigma,a,b\,:\,[a,\sigma]=[b,\sigma]=\sigma^{2n}=1,\,[b,a]=\sigma^{-2}\right\rangle.\]
To complete the proof, consider the short exact sequence:
\[1\longrightarrow \frac{(B_{n}(\mathbb{T}))^{(2)}}{(B_{n}(\mathbb{T}))^{(3)}} \longrightarrow \frac{B_{n}(\mathbb{T})}{(B_{n}(\mathbb{T}))^{(3)}}\stackrel{\widetilde{p}}{\longrightarrow} \frac{B_{n}(\mathbb{T})}{(B_{n}(\mathbb{T}))^{(2)}}\longrightarrow 1,\]
where $\widetilde{p}$ is the canonical projection. Then $\widetilde{p}(\sigma_{i})=\sigma$ for all $i=1,\ldots, n-1$, and as above, we see that for $i=1,\ldots,n-1$, the $(B_{n}(\mathbb{T}))^{(3)}$-cosets of the $\sigma_{i}$ coincide in $B_{n}(\mathbb{T})/(B_{n}(\mathbb{T}))^{(3)}$, and are equal to an element that we denote by $\widetilde{\sigma}$. Using the same relations of $B_{n}(\mathbb{T})$ as before, it follows that the $(B_{n}(\mathbb{T}))^{(3)}$-cosets of $a$ and $b$ commute with $\widetilde{\sigma}$ in $B_{n}(\mathbb{T})/(B_{n}(\mathbb{T}))^{(3)}$, $[b,a]=\sigma^{-2}$ and $\sigma^{2n}=1$, so $\widetilde{p}$ is an isomorphism, and hence $(B_{n}(\mathbb{T}))^{(2)}=(B_{n}(\mathbb{T}))^{(3)}$. We conclude that $B_{n}(\mathbb{T})$ is not residually soluble if $n\geq5$, because $(B_{n}(\mathbb{T}))^{(2)}$ is non trivial.
\end{proof}

\section{The case of the Klein bottle}\label{sec:klein}

In this section, we study the lower central and derived series of the (pure) braid groups of the Klein bottle, and we prove Theorem~\ref{klein}. In Section~\ref{sec:secklein}, we exhibit an algebraic section of the Fadell-Neuwirth short exact sequence~(\ref{seqFN}) for the Klein bottle, and we determine the centre of $B_{n}(\mathbb{K})$. These results will be used in the rest of the section, but are interesting in their own right. In Section~\ref{sec:lcsP2}, we focus on the case $n=2$, and in Theorems~\ref{GammaP2} and~\ref{gammaP2}, we describe the lower central and the lower $\mathbb{F}_{2}$-linear central filtration respectively of $P_{2}(\mathbb{K})$ in an explicit manner. In Section~\ref{sec:derivP2K}, we study the derived series of $P_{n}(\mathbb{K})$, and in Section~\ref{sec:seriesBnK}, we complete the proof of Theorem~\ref{klein} by extending our results to $B_{n}(\mathbb{K})$.


\subsection{A section in the case of the Klein bottle}\label{sec:secklein}

As we mentioned previously, the Fadell-Neuwirth short exact sequence~(\ref{seqFN}) admits a section for geometric reasons. We start by giving an explicit section.

\begin{prop}\label{sectionPn} In terms of the presentation of $P_{n}(\mathbb{K})$ given by Theorem~\ref{puras}, the map $s\colon\thinspace P_{n}(\mathbb{K})\longrightarrow P_{n+1}(\mathbb{K})$ defined on the generators of $P_{n}(\mathbb{K})$ by:
\begin{equation*}
\left\{
\begin{aligned}
a_i &\longmapsto a_i,\,\,i=1,\ldots,n-1\\
b_i &\longmapsto b_i,\,\,i=1,\ldots,n-1\\
C_{i,j} &\longmapsto C_{i,j},\,\, 1\leq i< j\leq n-1\\
C_{i,n} &\longmapsto C_{i,n}C_{i,n+1}C^{-1}_{n,n+1},\,\, 1\leq i \leq n-1 \\
a_n &\longmapsto a_{n}a_{n+1}\\
b_n &\longmapsto b_{n+1}b_{n}
\end{aligned}\right.
\end{equation*}
extends to an algebraic section for the Fadell-Neuwirth short exact sequence arising from equation~(\ref{seqFN}):
\begin{equation}\label{seqFNklein}
1\longrightarrow\pi_{1}(\mathbb{K}\setminus\left\{x_{1},\ldots,x_{n}\right\})\longrightarrow P_{n+1}(\mathbb{K})\stackrel{p_{\ast}}{\longrightarrow} P_{n}(\mathbb{K})\longrightarrow1.
\end{equation}
\end{prop}

The geometric idea behind this section is to make use of the non-vanishing vector field in $\mathbb{K}$ to duplicate the last string of the elements of $P_{n}(\mathbb{K})$.

\medskip

\begin{proof}[Proof of Proposition~\ref{sectionPn}] 
In what follows, the numbers~(\ref{it:puras1})--(\ref{it:puras8}) refer to the relations of $P_{n}(\mathbb{K})$ given by Theorem~\ref{puras}. The prove the proposition, it suffices to check that the images under $s$ of these 
remain valid in $P_{n+1}(\mathbb{K})$. We do this for relations~(\ref{it:puras2}) and~(\ref{it:puras6}). 
For relation~(\ref{it:puras2}), which is $a^{-1}_{i}b_{n}a_{i}=b_{n}a_{n}C^{-1}_{i,n}C_{i+1,n}a^{-1}_{n}$, we consider two cases.
\begin{itemize}
\item If $i+1=n$ then:
\begin{align*}
s(a^{-1}_{i}b_na_i)&=a^{-1}_{i}.b_nb_{n+1}\underbrace{C_{n,n+1}.a_i}_{(\ref{it:puras3})} =\underbrace{a^{-1}_{i}b_{n}a_i}_{(\ref{it:puras2})}\underbrace{a^{-1}_{i}b_{n+1}(a_i}_{(\ref{it:puras2})}C_{n,n+1})\\
&=(b_na_nC^{-1}_{i,n}\underbrace{a^{-1}_{n})(b_{n+1}}_{(\ref{it:puras2})}a_{n+1}C^{-1}_{i,n+1}C_{n,n+1}a^{-1}_{n+1})C_{n,n+1}\\
&=b_na_n\underbrace{C^{-1}_{i,n}(b_{n+1}a_{n+1}}_{\text{(\ref{it:puras3}) and (\ref{it:puras8})}} C^{-1}_{n,n+1}\underbrace{a^{-1}_{n+1}a^{-1}_{n})a_{n+1}}_{(\ref{it:puras1})}C^{-1}_{i,n+1}C_{n,n+1}a^{-1}_{n+1}C_{n,n+1}\\
&=b_n\underbrace{a_n(b_{n+1}}_{(\ast\ast)}a_{n+1}C^{-1}_{i,n})C^{-1}_{n,n+1}\underbrace{(a^{-1}_{n})C^{-1}_{i,n+1}C_{n,n+1}}_{(\ast)}a^{-1}_{n+1}C_{n,n+1}\\
&=b_n(b_{n+1}C_{n,n+1}a_{n})a_{n+1}C^{-1}_{i,n}\underbrace{C^{-1}_{n,n+1}(C_{n,n+1}}C^{-1}_{i,n+1}\underbrace{a^{-1}_{n})a^{-1}_{n+1}C_{n,n+1}}_{(\ref{it:puras3})}\\
&=b_nb_{n+1}C_{n,n+1}.a_{n}a_{n+1}.\underbrace{C^{-1}_{i,n}C^{-1}_{i,n+1}(C_{n,n+1}}_{(\ref{it:puras4})}a^{-1}_{n}a^{-1}_{n+1})\\
&=b_nb_{n+1}C_{n,n+1}.a_{n}a_{n+1}.(C_{n,n+1}C^{-1}_{i,n}C^{-1}_{i,n+1}).a^{-1}_{n}a^{-1}_{n+1}=s(b_na_nC^{-1}_{i,n}a^{-1}_{n}),
\end{align*}
%
where~$(\ast)$ (resp.~$(\ast\ast)$) is a consequence of relation~(\ref{it:puras3}) (resp.\ relations~(\ref{it:puras2}) and~(\ref{it:puras3})).

\item If $i+1<n$ then:
%
\begin{align*}
s(a^{-1}_{i}b_na_i)&=a^{-1}_{i}.b_nb_{n+1}\underbrace{C_{n,n+1}.a_i}_{(\ref{it:puras3})}=\underbrace{a^{-1}_{i}b_{n}a_{i}}_{(\ref{it:puras2})}\underbrace{a^{-1}_{i}b_{n+1}(a_i}_{(\ref{it:puras2})}C_{n,n+1})\\
&=(b_na_nC^{-1}_{i,n}C^{-1}_{i+1,n}\underbrace{a^{-1}_{n})(b_{n+1}}_{(\ref{it:puras2})}a_{n+1}C^{-1}_{i,n+1}C_{i+1,n+1}a^{-1}_{n+1})C_{n,n+1}\\
&=b_na_n\underbrace{C^{-1}_{i,n}C_{i+1,n}(b_{n+1}a_{n+1}}_{\text{(\ref{it:puras3}) and (\ref{it:puras8})}}C^{-1}_{n,n+1}\underbrace{a^{-1}_{n+1}a^{-1}_{n})a_{n+1}}_{(\ref{it:puras1})}C^{-1}_{i,n+1}C_{i+1,n+1}a^{-1}_{n+1}C_{n,n+1}\\
&=b_n\underbrace{a_n(b_{n+1}}_{(\ast\ast)}a_{n+1}C^{-1}_{i,n}C_{i+1,n})C^{-1}_{n,n+1}\underbrace{(a^{-1}_{n})C^{-1}_{i,n+1}C_{i+1,n+1}}_{(\ast)}a^{-1}_{n+1}C_{n,n+1}\\
&=b_n(b_{n+1}C_{n,n+1}a_{n})a_{n+1}C^{-1}_{i,n}C_{i+1,n}\underbrace{C^{-1}_{n,n+1}(C_{n,n+1}}C^{-1}_{i,n+1}C_{i+1,n+1}C^{-1}_{n,n+1}\underbrace{a^{-1}_{n})a^{-1}_{n+1}C_{n,n+1}}_{(\ref{it:puras3})}\\
&=b_nb_{n+1}C_{n,n+1}a_{n}a_{n+1}C^{-1}_{i,n}\underbrace{C_{i+1,n}C^{-1}_{i,n+1}}_{(\ref{it:puras4})}C_{i+1,n+1}\underbrace{C^{-1}_{n,n+1}(C_{n,n+1}}a^{-1}_{n}a^{-1}_{n+1})\\
&=b_nb_{n+1}C_{n,n+1}.a_{n}a_{n+1}.C^{-1}_{i,n}(C^{-1}_{i,n+1}C_{i+1,n})C_{i+1,n+1}.a^{-1}_{n}a^{-1}_{n+1}\\
&=s(b_na_nC^{-1}_{i,n}C_{i+1,n}a^{-1}_{n}),
\end{align*}
where~$(\ast)$ is a consequence of relation~(\ref{it:puras3}), similar to that in the case $i+1=n$,
and $(\ast\ast)$ is the same relation as in the case $i+1=n$.

\end{itemize}
Thus $s$ respects relation~(\ref{it:puras2}). For relation~(\ref{it:puras6}), which is $b_{n}b_{i}=b_{i}b_{n}C_{i,n}C^{-1}_{i+1,n}$, we also consider two cases.
\begin{itemize}
\item if $i+1=n$ then:
\begin{align*}
s(b_nb_i)&=b_nb_{n+1}\underbrace{C_{n,n+1}b_i}_{(\ref{it:puras8})}=b_n\underbrace{b_{n+1}(b_i}_{(\ref{it:puras6})}C_{n,n+1})=\underbrace{b_n(b_i}_{(\ref{it:puras6})}b_{n+1}C_{i,n+1}C^{-1}_{n,n+1})C_{n,n+1}\\
&=(b_{i}b_n\underbrace{C_{i,n})b_{n+1}}_{(\ref{it:puras8})}C_{i,n+1}=b_{i}b_n(b_{n+1}C_{i,n})C_{i,n+1}=b_{i}b_nb_{n+1}(C_{n,n+1}\underbrace{C^{-1}_{n,n+1})C_{i,n}C_{i,n+1}}_{(\ref{it:puras4})}\\
&=b_{i}b_nb_{n+1}C_{n,n+1}(C_{i,n}C_{i,n+1}C^{-1}_{n,n+1})=s(b_ib_nC_{i,n}).
\end{align*}


\item if $i+1< n$:
\begin{align*}
s(b_nb_i)&=b_nb_{n+1}\underbrace{C_{n,n+1}b_i}_{(\ref{it:puras8})}=b_n\underbrace{b_{n+1}(b_i}_{(\ref{it:puras6})}C_{n,n+1})=\underbrace{b_n(b_i}_{(\ref{it:puras6})}b_{n+1}C_{i,n+1}C^{-1}_{i+1,n+1})C_{n,n+1}\\
&=(b_{i}b_n\underbrace{C_{i,n}C^{-1}_{i+1,n})b_{n+1}}_{(\ref{it:puras8})}C_{i,n+1}C^{-1}_{i+1,n+1}C_{n,n+1}\\
&=b_{i}b_n(b_{n+1}C_{i,n}\underbrace{C^{-1}_{i+1,n})C_{i,n+1}}_{(\ref{it:puras4})}C^{-1}_{i+1,n+1}C_{n,n+1}\\
&=b_{i}b_nb_{n+1}(C_{n,n+1}\underbrace{C^{-1}_{n,n+1})C_{i,n}(C_{i,n+1}}_{(\ref{it:puras4})}\underbrace{C^{-1}_{i+1,n})C^{-1}_{i+1,n+1}C_{n,n+1}}_{(\ref{it:puras4})}\\
&=b_{i}b_{n}b_{n+1}C_{n,n+1}(C_{i,n}C_{i,n+1}C^{-1}_{n,n+1})(C_{n,n+1}C^{-1}_{i+1,n+1}C^{-1}_{i+1,n})=s(b_ib_nC_{i,n}C^{-1}_{i+1,n}).
\end{align*}

\end{itemize}
Thus $s$ respects relation~(\ref{it:puras6}). The computations for the other relations are similar, and are left to the reader.\end{proof} 

As we mentioned at the end of the introduction, for any compact surface $M$ and for all $n\in \mathbb{N}$, the centre $Z(B_{n}(M))$ of $B_{n}(M)$ is known, with the exception of the Klein bottle and the M\"{o}bius band. The section given by Proposition~\ref{sectionPn} allows us to determine $Z(B_{n}(M))$ if $M=\mathbb{K}$. For the sake of completeness, in Proposition~\ref{prop:centremobius} of the Appendix of this paper, we also compute the centre of the braid groups of the M\"{o}bius band.


\begin{prop}\label{centro}
For all $n\in \mathbb{N}$, the centre of $B_{n}(\mathbb{K})$ is equal to $\left\langle(b_{n}\cdots b_{1})^{2}\right\rangle$, and is isomorphic to $\mathbb{Z}$.
\end{prop}

\begin{proof}
The idea of the proof is similar to that of~\cite[Proposition 4.2]{PR}. Let $\beta_{n}=b_{n}\cdots b_{1}$, and let $Z_{n}=\left\langle \beta^{2}_{n}\right\rangle$. We will show by induction on $n$ that $Z_{n}=Z(B_{n}(\mathbb{K}))$. Arguing as in~\cite[Proposition~4.2, step~4]{PR}, we see that $Z(B_{n}(\mathbb K))\subset P_{n}(\mathbb{K})$, so $Z(B_{n}(\mathbb K))\subset Z(P_{n}(\mathbb K))$. Thus it suffices to show that $Z(P_{n}(\mathbb{K}))= Z_{n}$ and $Z_{n}\subset Z(B_{n}(\mathbb K))$. We prove that $Z(P_{n}(\mathbb{K}))= Z_{n}$ by induction on $n$. If $n=1$, by Theorem~\ref{puras}, we have:
\begin{equation}\label{presP1K}
\pi_{1}(\mathbb{K})=\bigl\langle a_{1}, b_{1}\,:\, a_{1}b_{1}=b_{1}a^{-1}_{1}\bigr\rangle,
\end{equation}
and it is well known that $Z(\pi_{1}(\mathbb{K}))=Z_{1}$. Now suppose by induction that $Z(P_{n-1}(\mathbb{K}))= Z_{n-1}$ for some $n\geq 2$. 

We first prove that $Z(P_{n}(\mathbb{K}))\subset Z_{n}$ in a manner similar to that of~\cite[Proposition~4.2, step~3]{PR}. Let $g\in Z(P_{n}(\mathbb{K}))$, and consider the Fadell-Neuwirth short exact sequence arising from~(\ref{seqFN}): \[1\longrightarrow\pi_{1}(\mathbb{K}\setminus\left\{x_{1},\ldots,x_{n-1}\right\})\longrightarrow P_{n}(\mathbb{K})\stackrel{p_{\ast}}{\longrightarrow} P_{n-1}(\mathbb{K})\longrightarrow1.\]
Since $p_{\ast}$ is surjective, $p_{\ast}(g)\in Z(P_{n-1}(\mathbb{K}))=Z_{n-1}$, and since $p_{\ast}(Z_{n})=Z_{n-1}$, there exists $h\in Z_{n}$ such that $p_{\ast}(h)=p_{\ast}(g)$. If $g'=gh^{-1}$ then $g'$ belongs to $Z(P_{n}(\mathbb{K}))$ and to the free group $\pi_{1}(\mathbb{K}\setminus\left\{x_{1},\ldots,x_{n-1}\right\})$ by exactness. Hence $g'\in Z(\pi_{1}(\mathbb{K}\setminus\left\{x_{1},\ldots,x_{n-1}\right\}))=\left\{1\right\}$, so $g'=1$, and thus $g=h\in Z_{n}$, which shows that $Z(P_{n}(\mathbb{K}))\subset Z_{n}$.

Still under the above induction hypothesis, we now prove that $Z_{n}\subset Z(P_{n}(\mathbb{K}))$. Using the section $s\colon\thinspace P_{n-1}(\mathbb{K})\longrightarrow P_{n}(\mathbb{K})$ given by Proposition~\ref{sectionPn}, we have $s(\beta^{2}_{n-1})=\beta^{2}_{n}$. By the induction hypothesis, 
$\beta^{2}_{n-1}a_{i}=a_{i}\beta^{2}_{n-1}$ and $\beta^{2}_{n-1}b_{i}=b_{i}\beta^{2}_{n-1}$ for all $1\leq i\leq n-1$, and $\beta^{2}_{n-1}C_{i,j}=C_{i,j}\beta^{2}_{n-1}$ for all $1\leq i<j\leq n-1$ in $P_{n-1}(\mathbb{K})$. Taking the image of both sides of these equations by $s$, we obtain the following relations:
\begin{equation}\label{eq:commbeta2}
\begin{cases}
\beta^{2}_{n}a_{i}=a_{i}\beta^{2}_{n} \;\text{and}\; \beta^{2}_{n}b_{i}=b_{i}\beta^{2}_{n}& \text{if $1\leq i< n-1$}\\ 
\beta^{2}_{n} \ldotp a_{n-1}a_{n}=a_{n-1}a_{n}\ldotp \beta^{2}_{n} \;\text{and}\; \beta^{2}_{n}\ldotp b_{n}b_{n-1}=b_{n}b_{n-1}\ldotp\beta^{2}_{n} &\text{if $i=n-1$,}
\end{cases}
\end{equation}
and
\begin{equation}\label{eq:commbeta2bis}
\begin{cases}
\beta^{2}_{n}C_{i,j}=C_{i,j}\beta^{2}_{n} & \text{if $1\leq i<j< n-1$}\\ 
\beta^{2}_{n}\ldotp C_{i,n-1}C_{i,n}C^{-1}_{n-1,n}=C_{i,n-1}C_{i,n}C^{-1}_{n-1,n}\ldotp \beta^{2}_{n} &\text{if $1\leq i<j= n-1$.}
\end{cases}
\end{equation}
Let us prove that $a_{n}$ and $b_{n}$ commute with $\beta^{2}_{n}$, from which it will follow from~(\ref{eq:commbeta2}) that $\beta^{2}_{n}$ commutes with $a_{i}$ and $b_{i}$ for all $1\leq i\leq n$. Using relations~(\ref{it:puras6})--(\ref{it:puras8}) of Theorem~\ref{puras}, we have: 
\begin{align}
C_{1,n}a^{-1}_{n}b_{i}&=C_{i+1,n}b_{i}C^{-1}_{i,n}C_{1,n}a^{-1}_{n}\label{eq:Cab}\\
b_{n} C_{i+1}b_{i} &= b_{i}b_{n} C_{i,n}\label{eq:bCb}
\end{align}
 for all $i=1,\ldots,n-1$. 
To prove relation~(\ref{eq:Cab}), notice that by relations~(\ref{it:puras7}) and~(\ref{it:puras8}) we have
\begin{equation*}
b^{-1}_{i}a_{n}b_{i}=a_{n}\underbrace{b_{n}(C_{i,n}C^{-1}_{i+1,n})^{-1}b^{-1}_{n}}_{(\ref{it:puras8})}=a_{n}(C^{-1}_{i+1,n}b^{-1}_{i}C_{i,n}b_{i}),
\end{equation*}
and taking the inverse of both sides, it follows that $a^{-1}_{n}b_{i}=C^{-1}_{i,n}b_{i}C_{i+1,n}a^{-1}_{n}$. We thus obtain:
\[\begin{array}{ll}C_{1,n}\underbrace{a^{-1}_{n}b_{i}}&=C_{1,n}(C^{-1}_{i,n}b_{i}C_{i+1,n}a^{-1}_{n})=(b_{i}\underbrace{b^{-1}_{i})C_{1,n}(b_{i}}_{(\ref{it:puras8})}\underbrace{b^{-1}_{i})C^{-1}_{i,n}b_{i}}_{(\ref{it:puras8})}C_{i+1,n}a^{-1}_{n}\\
&=\underbrace{b_{i}(C_{i+1}}_{(\ref{it:puras8})}C^{-1}_{i,n}C_{1,n}b_{n}C_{i+1,n}C^{-1}_{i,n}b^{-1}_{n})(b_{n}C_{i,n}C^{-1}_{i+1,n}b^{-1}_{n}C^{-1}_{i+1,n})C_{i+1,n}a^{-1}_{n}\\
&=(C_{i+1,n}b_{i})C^{-1}_{i,n}C_{1,n}a^{-1}_{n}.\end{array}\]
To prove relation~(\ref{eq:bCb}), one may use relation~(\ref{it:puras6}) and the fact that $b_{i}$ commutes with $C_{i+1,n}$ by relation~(\ref{it:puras8}).
We now claim that:
\begin{equation*}
\text{$a_{n}\beta_{n}=b_{n} \cdots b_{i+1} C_{i+1,n}^{-1} C_{1,n}a_{n}^{-1}b_{i} b_{i-1}\cdots b_{1}$ and $b_{n}\beta_{n}=b_{n}\cdots b_{i+1}b_{n} C_{i+1,n} b_{i}\cdots b_{1}$}
\end{equation*}
for all $i=0,\ldots,n-1$. We shall prove the claim by reverse induction on $i$. First, we have $a_{n}\beta_{n}=a_{n}(b_{n} b_{n-1}\cdots b_{1})\stackrel{(\ref{it:puras5})}{=}(b_{n}C_{1,n}a^{-1}_{n})b_{n-1}\cdots b_{1}$ and $b_{n}\beta_{n}=b_{n}b_{n}\cdots b_{1}$, so the claim is valid if $i=n-1$. Suppose that it holds for some $1\leq i\leq n-1$. Then:  
\begin{align*}
a_{n}\beta_{n}&= b_{n} \cdots b_{i+1} C_{i+1,n}^{-1} C_{1,n}a_{n}^{-1}b_{i} b_{i-1}\cdots b_{1}\stackrel{(\ref{eq:Cab})}{=} b_{n} \cdots b_{i+1} C_{i+1,n}^{-1} C_{i+1,n}b_{i}C^{-1}_{i,n}C_{1,n}a^{-1}_{n}b_{i-1}\cdots b_{1}\\
&= b_{n} \cdots b_{i+1} b_{i}C^{-1}_{i,n}C_{1,n}a^{-1}_{n}b_{i-1}\cdots b_{1},\; \text{and}\\
b_{n}\beta_{n}&= b_{n}\cdots b_{i+1}b_{n} C_{i+1,n} b_{i}\cdots b_{1}\stackrel{(\ref{eq:bCb})}{=} b_{n}\cdots b_{i+1}b_{i}b_{n} C_{i,n}b_{i-1}\cdots b_{1}
\end{align*}
so the claim holds. Taking $i=0$, we obtain $a_{n}\beta_{n}=\beta_{n}a^{-1}_{n}$ and $b_{n}\beta_{n}=\beta_{n}b_{n}C_{1,n}$. 
Hence $a_{n}\beta^{2}_{n}=\beta_{n}a^{-1}_{n}\beta_{n}=\beta^{2}_{n}a_{n}$, and applying~(\ref{it:puras5}) with $i=n$ twice, we obtain:
\begin{equation*}
b_{n}\beta^{2}_{n}=\beta_{n}b_{n}C_{1,n}\beta_{n}=\beta_{n} a_{n}b_{n}a_{n} \beta_{n}=\beta_{n}^{2} a_{n}^{-1}b_{n}C_{1,n} a_{n}^{-1}=\beta^{2}_{n}b_{n}.
\end{equation*}
Thus $\beta_{n}^{2}$ commutes with $a_{n}$ and $b_{n}$, and so with $a_{i}$ and $b_{i}$ for all $i\in \{1,\ldots,n\}$.
Finally, by relation~(\ref{it:puras5}), $C_{1,n}=b^{-1}_{n}a_{n}b_{n}a_{n}$ and $C_{i+1,n}=C_{i,n}a^{-1}_{n}b^{-1}_{n}a^{-1}_{i}b_{n}a_{i}a_{n}$ by relation~(\ref{it:puras2}). Since $\beta^{2}_{n}$ commutes with $a_{i}$ and $b_{i}$ for all $1\leq i \leq n$, it follows that $\beta^{2}_{n}$ commutes with $C_{1,n}$ by relation~(\ref{it:puras5}), and by induction that $\beta^{2}_{n}$ commutes with $C_{i,n}$ for all $1\leq i <n$.
So by~(\ref{eq:commbeta2bis}), $\beta^{2}_{n}$ commutes with $C_{i,j}$ for all $1\leq i<j\leq n$, and therefore commutes with all of the elements of a generating set of $P_{n}(\mathbb{K})$, hence $Z_{n}= Z(P_{n}(\mathbb K))$.

To complete the proof of the proposition, it remains to show that $Z_{n}\subset Z(B_{n}(\mathbb K))$. To do so, from above, it suffices to prove that $\sigma_{i}$ commutes with $\beta^{2}_{n}$ for all $i=1,\ldots,n-1$. 
One may see by Figure~\ref{fig:geradores} that $\sigma^{2}_{i}=C_{i,i+1}$. Further: 
\begin{equation}\label{sigmab}
\sigma^{-1}_{i}b_{j}\sigma_{i}=\begin{cases}
b_{i+1}\sigma^{2}_{i} & \text{if $j=i$}\\
\sigma^{-2}_{i}b_{i} & \text{if $j=i+1$}\\
b_{j} & \text{otherwise.}
\end{cases}
\end{equation}
The case $j=i$ is illustrated in Figure~\ref{fig:bj}, and if $j=i+1$ then 
$\sigma^{-1}_{i} b_{i+1}\sigma_{i}=\sigma^{-1}_{i}\ldotp \sigma^{-1}_{i}b_{i}\sigma_{i}\ldotp\sigma^{-2}_{i}\ldotp \sigma_{i}=\sigma^{-2}_{i}b_{i}$ using the case $j=i$.
\begin{figure}[h!]
\hfill
\begin{tikzpicture}[scale=1, very thick]

{\draw[line width=2pt][->>] (0,1)-- (-1,0.5);}
{\draw[line width=2pt][-] (0,1)-- (-2,0);}
{\draw[line width=2pt][->>] (0,-1) -- (1,-0.5);}
{\draw[line width=2pt][-] (2,0) -- (0,-1);}
{\draw[line width=2pt][-<] (0,1) -- (1,0.5);}
{\draw[line width=2pt][-] (2,0) -- (0,1);}
{\draw[line width=2pt][>-] (-1,-0.5) -- (-2,0);}
{\draw[line width=2pt][-] (0,-1) -- (-2,0);}

%
{\draw (0.75,0) circle (0.2mm);}
{\draw (-0.75,0) circle (0.2mm);}
{\draw (0.5,0) circle (0.2mm);}
{\draw (1,0) circle (0.2mm);}
{\draw (1.25,0) circle (0.2mm);}
{\draw (-1.25,0) circle (0.2mm);}
{\draw (-0.5,0) circle (0.2mm);}
{\draw (-1,0) circle (0.2mm);}

\draw[->>] (0.45,0.26) .. controls (0,0.5) .. (-0.51,0.11);
\draw[->>] (0.5,0) .. controls (-1,-0.29) and (-0.12,0.23) .. (0.75,-0.63);
\draw[white, line width=3pt] (-1.26,0.37).. controls (-0.27,0.48).. (0.4,0.05);
{\draw[line width=2pt][-] (0,1) -- (-2,0);}
\draw[->>] (-1.26,0.37).. controls (-0.27,0.48).. (0.4,0.05);
\draw[white, line width=3pt] (-0.5,0) .. controls (-0.12,0.23) and (0.68,0.11) .. (0.45,0.26);
\draw (-0.5,0) .. controls (-0.12,0.23) and (0.68,0.11) .. (0.45,0.26);
{\draw (-0.5,0) circle (0.2mm);}

\node at (0,-1.5) {$\sigma^{-1}_{i}b_{i}\sigma_{i}$};

{\draw[line width=2pt][->>] (5,1)-- (4,0.5);}
{\draw[line width=2pt][-] (5,1)-- (3,0);}
{\draw[line width=2pt][->>] (5,-1) -- (6,-0.5);}
{\draw[line width=2pt][-] (7,0) -- (5,-1);}
{\draw[line width=2pt][-<] (5,1) -- (6,0.5);}
{\draw[line width=2pt][-] (7,0) -- (5,1);}
{\draw[line width=2pt][>-] (4,-0.5) -- (3,0);}
{\draw[line width=2pt][-] (5,-1) -- (3,0);}
%
%
{\draw (5.75,0) circle (0.2mm);}
{\draw (4.25,0) circle (0.2mm);}
{\draw (5.5,0) circle (0.2mm);}
{\draw (6,0) circle (0.2mm);}
{\draw (6.25,0) circle (0.2mm);}
{\draw (3.75,0) circle (0.2mm);}
{\draw (4.5,0) circle (0.2mm);}
{\draw (4,0) circle (0.2mm);}
{\draw (4.25,0) circle (0.2mm);}

\draw[->>] (5.5,0)--(6.25,-0.38);
\draw[->>] (4.25,0.63) .. controls (4.3,-0.39) and (4.62,-0.26) .. (5.44,0.09);

\node at (5,-1.5) {$b_{i+1}\sigma^{2}_{i}$};

\end{tikzpicture}
\hspace*{\fill}
\caption{The relation $\sigma^{-1}_{i}b_{i}\sigma_{i}=b_{i+1}\sigma^{2}_{i}$.}\label{fig:bj}
\end{figure}
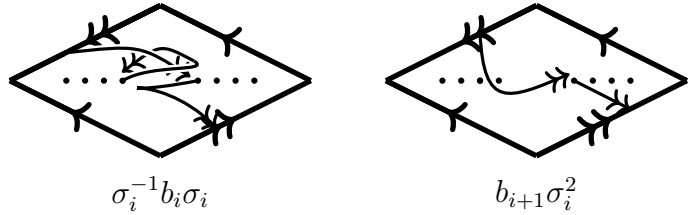
For $i=1,\ldots, n-1$, using (\ref{sigmab}) and relation~(\ref{it:puras6}) of Theorem~\ref{puras}, we have: 
\begin{align*}
\sigma_{i}^{-1} (b_{n}\cdots b_{1})\sigma_{i}= b_{n}\cdots b_{i+2} \cdot C_{i,i+1}^{-1} \underbrace{b_{i} \cdot b_{i+1} C_{i,i+1}}\cdot b_{i-1}\cdots b_{1}= C_{i,i+1}^{-1} b_{n}\cdots b_{1}=\sigma_{i}^{-2} b_{n}\cdots b_{1},
\end{align*}
from which it follows that $(b_{n}\cdots b_{1})^{2}\sigma_{i}=\sigma_{i}(b_{n}\cdots b_{1})^{2}$ as required.
\end{proof} 

\begin{remark}\label{P2} For $n=2$, we modify slightly the presentation of $P_{2}(\mathbb{K})$ given by Theorem~\ref{puras} by removing the generator $C_{1,2}$ using relation~(\ref{it:puras5}), so $C_{1,2}=b^{-1}_{2}a_{2}b_{2}a_{2}$. Hence $P_{2}(\mathbb{K})$ is generated by $a_{1},a_{2},b_{1}$ and $b_{2}$, subject to the relations:

\begin{enumerate}
	\item $a^{-1}_{1}a_{2}a_{1}=a_{2}$;
	
	\item $a^{-1}_{1}b_{2}a_{1}=a^{-1}_{2}b_{2}a^{-1}_{2}$;
	
	\item $b^{-1}_{1}a_{2}b_{1}=a_{2}b_{2}a^{-1}_{2}b^{-1}_{2}a^{-1}_{2}$;
	
	\item $b^{-1}_{1}b_{2}b_{1}=a_{2}b_{2}a_{2}$;
	
	\item $b^{-1}_{2}a_{2}b_{2}a_{2}=b_{1}a^{-1}_{1}b^{-1}_{1}a^{-1}_{1}$.
\end{enumerate}
\end{remark}

Using Proposition~\ref{sectionPn} and the Fadell-Neuwirth short exact sequence~(\ref{seqFNklein}), $P_{n+1}(\mathbb{K})$ may be written as a semi-direct product of the free group $\pi_{1}(\mathbb{K}\setminus\left\{x_{1},\ldots,x_{n}\right\})$ by $s(P_{n}(\mathbb{K}))$ for all $n\in \mathbb{N}$. In particular, if $n=1$ then:
\begin{equation}\label{eq:semiP2K}
P_{2}(\mathbb{K})\cong\pi_{1}(\mathbb{K} \setminus{\left\{x_{1}\right\}})\rtimes_{\varphi} s(P_{1}(\mathbb{K})),
\end{equation}
where $\pi_{1}(\mathbb{K} \setminus{\left\{x_{1}\right\}})=\left\langle a_{2}, b_{2}\right\rangle$ is a free group of rank $2$, \[s(P_{1}(\mathbb{K}))=\left\langle a_{1}a_{2}, b_{2}b_{1}\,:\,(b_{2}b_{1})(a_{1}a_{2})=(a_{1}a_{2})^{-1}(b_{2}b_{1})\right\rangle\]
by~(\ref{presP1K}), and where the action $\varphi\colon \thinspace s(P_{1}(\mathbb{K}))\longrightarrow \operatorname{\text{Aut}}(\pi_{1}(\mathbb{K} \setminus{\left\{x_{1}\right\}}))$ is given by:
\begin{equation}\label{acaoP2} 
\left\{\begin{aligned}
\varphi(a_{1}a_{2}):& \begin{cases}
a_{2}\longmapsto a_{2}\\ 
b_{2}\longmapsto a^{-2}_{2}b_{2}
\end{cases}\\
\varphi(b_{2}b_{1}):& 
\begin{cases}
 a_{2}\longmapsto a^{-1}_{2}\\ b_{2}\longmapsto a_{2}b_{2}a_{2},
\end{cases}
\end{aligned}\right.
\end{equation}
using Remark~\ref{P2}.

\subsection{The lower central series of $P_{2}(\mathbb{K})$}\label{sec:lcsP2} 

In this section, we use Theorem~\ref{serie} to calculate explicitly the lower central series of $P_{2}(\mathbb{K})$. This will enable us show that $P_{2}(\mathbb{K})$ is residually nilpotent.

\begin{teo} \label{GammaP2} Let $n\geq 2$. With the notation of Theorem~\ref{puras}, we have: 
\[\Gamma_{n}(P_{2}(\mathbb{K}))= \bigl\langle \!\bigl\langle a^{2^{n-1}}_{2}, x^{2^{n-i}}\,:\,x\in \Gamma_{i}(\pi_{1}(\mathbb{K}\setminus\left\{x_{1}\right\})),\,2\leq i \leq n\bigr\rangle\!\bigr\rangle_{\pi_{1}(\mathbb{K}\setminus\left\{x_{1}\right\})}\rtimes_{\varphi}\bigl\langle (a_{1}a_{2})^{2^{n-1}}\bigr\rangle,\] 
where
$\varphi$ is as defined in equation~(\ref{acaoP2}). Consequently, $P_{2}(\mathbb{K})$ is residually nilpotent.
\end{teo}

The idea of the proof is to apply Theorem~\ref{serie}(\ref{it:lcspart1}) to the semi-direct product structure of $P_{2}(\mathbb{K})$ given by~(\ref{eq:semiP2K}). From now on, we shall make use of the notation of that theorem, taking $H=\pi_{1}(\mathbb{K}\setminus\left\{x_{1}\right\})$, $G=s(P_{1}(\mathbb{K}))$ and $P_{2}(\mathbb{K})=H\rtimes_{\varphi} G$.  In the lemmas that follow, we first compute the subgroups $K_{n}$, $H_{n}$ and $L_{n}$ for these choices of $G$ and $H$, and in Proposition~\ref{lemma3}, we calculate the subgroup $L_{n}$, which will enable us to prove Theorem~\ref{GammaP2}.
We first we need to know the lower central series of $P_{1}(\mathbb{K})$. Applying Theorem~\ref{serie} to the semi-direct product structure $\langle a_{1}\rangle \rtimes \langle b_{1}\rangle$ of $\pi_{1}(\mathbb{K})$ given by~(\ref{presP1K}), it is straightforward to see that $K_{2}=H_{2}=L_{2}=\left\langle a^{2}_{1}\right\rangle$, and if $n\geq3$, $K_{n}=\left\{1\right\}$ and $H_{n}=L_{n}=\bigl\langle a^{2^{n-1}}_{1}\bigr\rangle$, therefore the lower central series of $s(P_{1}(\mathbb{K}))$ is given by:
\begin{equation}\label{eq:pi1}\Gamma_{n}(s(P_{1}(\mathbb{K})))=\bigl\langle (a_{1}a_{2})^{2^{n-1}}\bigr\rangle\end{equation} for all $n>1$. We now turn to the case of $P_{2}(\mathbb{K})$, and we first determine $K_{n}$.

\begin{lem}\label{lemma1} With the notation of Theorem~\ref{serie}, and taking $H=\pi_{1}(\mathbb{K}\setminus\left\{x_{1}\right\})$, $G=s(P_{1}(\mathbb{K}))$ and $P_{2}(\mathbb{K})=H\rtimes_{\varphi}G$, the subgroup $K_{n}$ is equal to 
$\left\langle \!\left\langle a^{2}_{2}, \Gamma_{2}(H)\right\rangle \! \right\rangle_{H}$ 
if $n=2$, and to 
$\bigl\langle \!\bigl\langle a^{2^{n-1}}_{2}\bigr\rangle\!\bigr\rangle_{H}$ 
if $n\geq 3$. 
\end{lem}

\begin{proof} First suppose that $n=2$. By~(\ref{acaoP2}), we have:
\begin{equation}\label{acaok2} 
\left\{\begin{aligned}\varphi(a_{1}a_{2})(a_{2}).a^{-1}_{2}&= 1\\
\varphi(a_{1}a_{2})(b_{2}).b^{-1}_{2}&= \varphi(b_{2}b_{1})(a_{2}).a^{-1}_{2}=a^{-2}_{2}\\
\varphi(b_{2}b_{1})(b_{2}).b^{-1}_{2}&= a_{2}b_{2}a_{2}b^{-1}_{2}=a^{2}_{2}[a^{-1}_{2},b_{2}],\\\end{aligned}\right.
\end{equation}
and so $\left\{a^{2}_{2},[a^{-1}_{2},b_{2}]\right\}\subset K_{2}$. Since $K_{2}$ is normal in $H$ by Lemma~\ref{Lnnormal}, $\left\langle \!\left\langle a^{2}_{2},[a^{-1}_{2},b_{2}]\right\rangle \! \right\rangle_{H}$
is a subgroup of $K_{2}$, and therefore $\left\langle \!\left\langle a^{2}_{2}, \Gamma_{2}(H)\right\rangle \! \right\rangle_{H} \subset K_{2}$ because $\Gamma_{2}(H)=\left\langle \!\left\langle [a^{-1}_{2},b_{2}]\right\rangle \! \right\rangle_{H}$. For the other inclusion, (\ref{acaok2}) implies that $\varphi(g)(h).h^{-1}\in\left\langle \!\left\langle a^{2}_{2},\Gamma_{2}(H)\right\rangle \! \right\rangle_{H}$ for all $h$ (resp.\ all $g$) belonging to the generating set $\left\{a_{2},b_{2}\right\}$ (resp.\ $\left\{a_{1}a_{2},b_{2}b_{1}\right\}$) of $H$ (resp.\ of $G$).
The inclusion $K_{2}\subset \left\langle \!\left\langle a^{2}_{2}, \Gamma_{2}(H)\right\rangle \! \right\rangle_{H}$ then follows from Lemma~\ref{geradoreslema}(\ref{geradoreslema1}). This proves the result for $n=2$.

Now suppose that $n\geq 3$. Then $\Gamma_{n-1}(G)=\bigl\langle (a_{1}a_{2})^{2^{n-2}}\bigr\rangle$ by~(\ref{eq:pi1}). Using (\ref{acaoP2}), we have $\varphi((a_{1}a_{2})^{2^{n-2}})(a_{2}).a^{-1}_{2}=1$ and: 
\[\varphi((a_{1}a_{2})^{2})(b_{2})=\varphi(a_{1}a_{2})(a^{-2}_{2}b_{2})= a^{-2}_{2}(a^{-2}_{2}b_{2})=a^{-4}_{2}b_{2}.\] By induction, suppose that $\varphi((a_{1}a_{2})^{j})(b_{2})=a^{-2j}_{2} b_{2}$, for some $j\geq 2$. Then:
\[\varphi((a_{1}a_{2})^{j+1})(b_{2})=\varphi(a_{1}a_{2})\varphi((a_{1}a_{2})^{j})(b_{2})=\varphi(a_{1}a_{2})(a^{-2j}_{2}b_{2})=a^{-2j}_{2}(a^{-2}_{2}b_{2})=a^{-2(j+1)}_{2}b_{2}.\]
In particular, if $j=2^{n-2}$, we have $\varphi((a_{1}a_{2})^{2^{n-2}})(b_{2}).b^{-1}_{2}=(a^{-2(2^{n-2})}_{2}b_{2}).b^{-1}_{2}=a^{-2^{(n-1)}}_{2}$, and hence $\bigl\langle \!\bigl\langle a^{2^{n-1}}_{2} \bigr\rangle\! \bigr\rangle_{H}\subset K_{n}$. 
Conversely, taking $\widetilde{G}=\{(a_{1}a_{2})^{2^{n-2}}\}$ and $\widetilde{H}=\{a_{2},b_{2}\}$ in Lemma~\ref{geradoreslema}(\ref{geradoreslema1}) and using~(\ref{eq:pi1}), we obtain $K_{n}\subset \left\langle \!\left\langle Z\right\rangle \! \right\rangle_{H}$, where:
\[Z=\Bigl\{\varphi(g)(h).h^{-1}\,:\,g\in\big\{(a_{1}a_{2})^{2^{n-2}}\big\},\,h\in\left\{a_{2},b_{2}\right\}\Bigr\}=\bigl\{a^{-2^{(n-1)}}_{2}\bigr\},\]
whence the inclusion $K_{n}\subset \bigl\langle \!\bigl\langle a^{2^{n-1}}_{2} \bigr\rangle\! \bigr\rangle_{H}$, and this proves the lemma.
\end{proof} 

If $\beta\in H$ then $\beta a^{2}_{2}\beta^{-1}=[\beta,a^{2}_{2}]a^{2}_{2}\in\left\langle a^{2}_{2}, \Gamma_{2}(H)\right\rangle$, and since $\Gamma_{2}(H)$ is normal in $H$, it follows that the subgroup $\left\langle a^{2}_{2}, \Gamma_{2}(H)\right\rangle$ is also normal in $H$, so $\left\langle a^{2}_{2}, \Gamma_{2}(H)\right\rangle= \left\langle\!\left\langle a^{2}_{2},\Gamma_{2}(H)\right\rangle\! \right\rangle_{H}$. Using the relations $K_{2}=H_{2}$ and $[H,L_{1}]=\Gamma_{2}(H)$, it follows from Lemma~\ref{lemma1} that: 
\begin{equation}\label{L2} L_{2}=\left\langle a^{2}_{2}, \Gamma_{2}(H)\right\rangle=\left\langle\!\left\langle a^{2}_{2},[a_{2},b_{2}]\right\rangle\! \right\rangle_{H}.
\end{equation}
Let $W_{2}=\widetilde{W}_{2}=L_{2}$, and for $n\geq 3$, define:
\begin{align}
W_{n}&=\left\langle  \Gamma_{n}(H),\,x^{2}\,:\,x\in W_{n-1}\right\rangle\label{eq:defwn}\\
\widetilde{W}_{n}&=\bigl\langle \!\bigl\langle a^{2^{n-1}}_{2},\,x^{2^{n-i}}\,:\,x\in\Gamma_{i}(H),\,2\leq i\leq n\bigr\rangle \! \bigr\rangle_{H}.\label{eq:defwntilde}
\end{align}
Note that $W_{n}$ is normal in $H$ for all $n\geq2$. This follows from the fact that $\Gamma_{n}(H)$ is normal in $H$ for all $n\geq 2$ and arguing by induction on $n$ as follows. If $n=2$ then $W_{2}=L_{2}$ by Lemma~\ref{Lnnormal}, so suppose that $n\geq 3$, and that $W_{n-1}$ is normal in $H$. Then $hxh^{-1}\in W_{n-1}$
for all $x\in W_{n-1}$ and $h\in H$, thus
$hx^{2}h^{-1}=(hxh^{-1})^{2}\in W_{n}$, and $W_{n}$ is normal in $H$ as claimed. If $x\in \Gamma_{i}(H)$ and $2\leq i\leq n$, we will refer to the elements $a^{2^{n-1}}_{2}$ and $x^{2^{n-i}}$ as generators of $\widetilde{W}_{n}$. In order to prove Theorem~\ref{GammaP2}, we will show in Proposition~\ref{lemma3} that $L_{n}=W_{n}=\widetilde{W}_{n}$ for all $n\geq 2$. Supposing this to be the case, the following lemma implies that to determine the lower central series of $P_{2}(\mathbb{K})$ using Theorem~\ref{serie}, it is not necessary to calculate the subgroups $H_{n}$.

\begin{lem}\label{lemma2} Suppose that $L_{i}=W_{i}=\widetilde{W}_{i}$ for all $2\leq i\leq n$. Then $H_{n+1}\subset \left\langle K_{n+1},[H,L_{n}]\right\rangle$. In particular, $L_{n+1}=\left\langle K_{n+1},[H,L_{n}]\right\rangle$.
\end{lem}

\begin{proof} We prove the given inclusion by induction on $n$. If $n=2$, using (\ref{acaoP2}) notice that:
\begin{itemize}
	\item $\varphi(a_{1}a_{2})(a^{2}_{2})a^{-2}_{2}=1$.
	\item $\varphi(b_{2}b_{1})(a^{2}_{2})a^{-2}_{2}=a^{-4}_{2}\in K_{3}$ by Lemma~\ref{lemma1}.
	\item $\varphi(a_{1}a_{2})([a_{2},b_{2}])[b_{2},a_{2}]=[a^{-2}_{2},a_{2}b_{2}a^{-1}_{2}b^{-1}_{2}a^{-1}_{2}]\in[L_{2},H]$ by~(\ref{L2}).
	\item $\varphi(b_{2}b_{1})([a_{2},b_{2}])[b_{2},a_{2}]=\underbrace{[b_{2},a^{2}_{2}]}_{\in [H,L_{2}]}\underbrace{[a_{2},[a_{2},b_{2}]]}_{\in \Gamma_{3}(H)}\in [L_{2},H]$ by~(\ref{L2}).
\end{itemize}
Using these calculations and the description of $L_{2}$ given in~(\ref{L2}), it follows from the first part of Lemma~\ref{geradoreslema}(\ref{geradoreslema2}) and Lemma~\ref{lemma1} that: 
\[H_{3}= \left\langle \!\left\langle [H,L_{2}],\,\varphi(g)(w).w^{-1}\,:\,g\in\left\{a_{1}a_{2},b_{2}b_{1}\right\}, w\in\left\{a^{2}_{2},[a_{2},b_{2}]\right\} \right\rangle\!\right\rangle_{H}\subset \left\langle K_{3}, [H,L_{2}]\right\rangle,\] which proves the result if $n=2$.

Suppose now that the given inclusion holds for $n-1$ for some $n\geq 3$, and assume that $L_{i}=W_{i}=\widetilde{W}_{i}$ for all $2\leq i\leq n$. Let $\varphi(g)(x).x^{-1}$ be an element of $H_{n+1}$, where $g\in G$ and $x\in L_{n}=\left\langle K_{n},[H,L_{n-1}]\right\rangle$ by the induction hypothesis. We wish to show that $\varphi(g)(x).x^{-1}\in\left\langle K_{n+1},[H,L_{n}]\right\rangle$. By Lemma~\ref{geradoreslema}(\ref{geradoreslema1}), we only need to check the following two possibilities:
\begin{itemize}
\item $x=a^{2^{n-1}}_{2}$. Then $x\in K_{n}$ by Lemma~\ref{lemma1}, $\varphi(a_{1}a_{2})(x).x^{-1}=1$ and $\varphi(b_{2}b_{1})(x)x^{-1}=a^{-2^{n}}\in K_{n+1}$ by Lemma~\ref{lemma1} and (\ref{acaok2}).
	
\item $x=[h,l]\in[H,L_{n-1}]$, where $h\in H$ and $l\in L_{n-1}$. Then $\varphi(g)([h,l])[h,l]^{-1}$ may be written in the following form:
\begin{equation}\label{eq:comutador}
\Bigl(\varphi(g)(h).l\underbrace{[(\varphi(g)(l^{-1}).l)^{-1},\varphi(g)(h^{-1})]}_{\in [L_{n},H]}l^{-1}\varphi(g)(h^{-1})\Bigr) h\underbrace{\left[(\varphi(g)(h^{-1}).h)^{-1},l\right]}_{\in[L_{2},L_{n-1}]}h^{-1}.
\end{equation}
\end{itemize}
To complete the proof, it suffices to show that the subgroup $[L_{2},L_{n-1}]$ is contained in $[H,L_{n}]$. To do so, first note that $[L_{2},L_{n-1}]$ is normal in $H$ because $L_{j}$ is normal in $H$ for all $j\geq 2$ by Lemma~\ref{Lnnormal}. Using the fact that $L_{2}=\left\langle a^{2}_{2},\Gamma_{2}(H)\right\rangle$ by~(\ref{L2}), it suffices to show that the following elements belong to $[H,L_{n}]$:
\begin{itemize}
	\item $\left[[y,z],l\right]\in \left[\Gamma_{2}(H),L_{n-1}\right]$, where $y,z\in H$ and $l\in L_{n-1}$. Then: 
\begin{equation*}
[[y,z],l]=y\underbrace{\left[z,[y^{-1},l]\right]}_{\in [H,L_{n}]}y^{-1}.lyz\underbrace{\left[[z^{-1},l^{-1}],y^{-1}\right]}_{\in[H,L_{n}]}z^{-1}y^{-1}l^{-1},
\end{equation*}
because $[H,L_{n-1}]\subset L_{n}$. Thus $\left[[y,z],l\right]\in[H,L_{n}]$ because $[H,L_{n}]$ is a normal subgroup of $H$, and therefore $[\Gamma_{2}(H),L_{n-1}]\subset [H,L_{n}]$.
		
\item $[a^{2}_{2},l]\in\left[\left\langle a^{2}_{2}\right\rangle, L_{n-1}\right]$,  where $l\in L_{n-1}$.  Then using~(\ref{itercomm1}), we have:
\begin{equation}\label{eq:a22l}
[a^{2}_{2},l]=\underbrace{[a_{2},[a_{2},l]]}_{\in [H,L_{n}]}[a_{2},l]^{2},
\end{equation}
since $[H,L_{n-1}]\subset L_{n}$. Further $[a_{2},l]^{2} \in [H,L_{n}]$ because $l\in L_{n-1}=W_{n-1}$ by hypothesis, so $l^{2}\in W_{n}=L_{n}$, and \[\underbrace{[l^{2},a_{2}]}_{\in[L_{n},H]}=\underbrace{[l,[l,a_{2}]]}_{\in[H,L_{n}]}[l,a_{2}]^{2}\]	
\end{itemize}
by~(\ref{itercomm1}). So $[l,a_{2}]^{2}\in [H,L_{n}]$, and thus $[a^{2}_{2},l]\in [H,L_{n}]$ by~(\ref{eq:a22l}). This shows that $[L_{2},L_{n-1}]\subset [H,L_{n}]$, hence $H_{n+1}\subset \left\langle K_{n+1},[H,L_{n}]\right\rangle$ as desired, which concludes the proof of the first part of the statement. The second part follows from the first part and the definition of $L_{n+1}$.
\end{proof} 

In order to prove Proposition~\ref{lemma3}, we shall require a couple of intermediate results. Let $A$  be either the empty set or a normal subgroup of $H$, and for $m\geq1$, let: 
\[B_{m}=\left\langle [x_{1},\ldots,x_{i}]^{2^{m-i-k}}\in\Gamma_{i}(H)\,:\begin{array}{r}\exists 1\leq j_{1}<\cdots<j_{k}\leq i\,:\,x_{j_{1}},\ldots,x_{j_{k}}\in A,\\ \text{$0\leq k \leq m-i$ and $i=1,\ldots,m$}\end{array}\right\rangle.\]
For $1\leq l \leq m$, let:
\begin{equation}\label{eq:defekn}
\mathcal{E}_{l,m}=\left\langle\!  [x_{1},\ldots,x_{i}]^{2^{m-i-k}}\in\Gamma_{i}(H)\,:\!\!\begin{array}{r}\exists 1\leq j_{1}<\cdots<j_{k}\leq i\,:\,x_{j_{1}},\ldots,x_{j_{k}}\in A,\\ \text{$0\leq k \leq m-i$ and $l\leq i\leq m$}\end{array}\!\right\rangle.
\end{equation}
In the case that $A=\varnothing$, we will denote the subgroup $\mathcal{E}_{l,m}$ by $\widetilde{\mathcal{E}}_{l,m}$. This corresponds to taking $k=0$ in~(\ref{eq:defekn}), so:
\begin{equation}\label{eq:defekn1}
\widetilde{\mathcal{E}}_{l,m}=\bigl\langle  [x_{1},\ldots,x_{i}]^{2^{m-i}}\in\Gamma_{i}(H)\,:\, l\leq i\leq m\bigr\rangle.
\end{equation}
If $x=[x_{1},\ldots, x_{i}]\in \Gamma_{i}(H)$, the elements $x^{2^{m-i-k}}$ (resp.\ $x^{2^{m-i}}$) of $\mathcal{E}_{l,m}$ (resp.\ of $\widetilde{\mathcal{E}}_{l,m}$) given in~(\ref{eq:defekn}) (resp.\ in~(\ref{eq:defekn1})) will be termed \emph{generators} of $\mathcal{E}_{l,m}$ (resp.\ of $\widetilde{\mathcal{E}}_{l,m}$).
Note that $\Gamma_{m}(H) = \mathcal{E}_{m,m} \subset \mathcal{E}_{l+1,m}\subset \mathcal{E}_{l,m} \subset \mathcal{E}_{1,m}=B_{m}$
by~\cite[Problem~3, Section~5.3, p.~297]{MKS} for all $l=1,\ldots,m-1$, and that for all $1\leq l \leq m$, $\mathcal{E}_{l,m}$ is normal in $H$, since if $h\in H$ and $x^{2^{m-i-k}}$ is a generator of $\mathcal{E}_{l,m}$, where $x=[x_{1},\ldots,x_{i}]\in \Gamma_{i}(H)$, then $h[x_{1},\ldots,x_{i}]^{2^{m-i-k}}h^{-1}=[hx_{1}h^{-1},\ldots,hx_{i}h^{-1}]^{2^{m-i-k}}\in \mathcal{E}_{l,m}$ because $A$ is normal in $H$ or is empty. In particular, taking $A=\varnothing$, we have:
\begin{equation}\label{eq:inclElm}
\Gamma_{m}(H) = \widetilde{\mathcal{E}}_{m,m} \subset \widetilde{\mathcal{E}}_{l+1,m}\subset \widetilde{\mathcal{E}}_{l,m} \subset \widetilde{\mathcal{E}}_{1,m}=B_{m},
\end{equation}
and that $\widetilde{\mathcal{E}}_{l,m}$ is normal in $H$ for all $1\leq l \leq m$.

\begin{lem}\label{lemaprinc} Let $y\in H$, let $m\geq 1$, and let $1\leq i \leq m$. If $x=[x_{1},\ldots, x_{i}]$ is an element of $\Gamma_{i}(H)$ for which $x^{2^{m-i-k}}$ is a generator of $\mathcal{E}_{i,m}$, where $0\leq k \leq m-i$ (resp.\ $x^{2^{m-i}}$ is a generator of $\widetilde{\mathcal{E}}_{i,m}$), then:
\begin{align}\label{cong}[x^{2^{m-i-k}},y] &\equiv[x,y]^{2^{m-i-k}} \bmod{\mathcal{E}_{i+1,m+1}}\; \text{and}\\
\label{congWn}[x^{2^{m-i}},y]&\equiv[x,y]^{2^{m-i}} \bmod{\widetilde{\mathcal{E}}_{i+1,m+1}}
\end{align}
respectively. 
\end{lem}

\begin{remark}\label{rem:inclWn}
Let $m\geq 1$, and let $1\leq i \leq m$. Since $\mathcal{E}_{i+1,m+1}\subset B_{m+1}$ (resp.\ $\widetilde{\mathcal{E}}_{i+1,m+1}\subset \widetilde{\mathcal{E}}_{2,m+1} \subset \widetilde{W}_{m+1}$ by~(\ref{eq:defwntilde})), the congruence~(\ref{cong}) (resp.~(\ref{congWn})) is also valid modulo $B_{m+1}$ (resp.\ modulo $\widetilde{W}_{m+1}$).
\end{remark}

\begin{proof} 
It suffices to prove~(\ref{cong}), since then the congruence~(\ref{congWn}) follows by taking $A=\varnothing$.
We will do so by induction on $m$. If $m=1$ then $i=1$ and $k=0$, and the congruence is in fact an equality. If $m=2$, we consider two cases:
\begin{itemize}
\item if $i=2$, or if $i=k=1$, then $x\in \Gamma_{2}(H)$, and the two sides of~(\ref{cong}) are equal.
	
\item if $i=1$ and $k=0$ then $x\in \Gamma_{1}(H)=H$. Thus $[x^{2},y]=[x,x,y][x,y]^{2}$ by~(\ref{itercomm1}), and $[x,x,y]\in \Gamma_{3}(H)$, so $[x,x,y]\in \mathcal{E}_{2,3}$, and we obtain~(\ref{cong}). 
\end{itemize}
This proves the result if $m=2$. We now consider the general case. 

\smallskip

\noindent
\textbf{Induction hypothesis 1:} suppose that the congruence~(\ref{cong}) holds for some $m\geq 1$. 

Let us show by induction that the result holds for $m+1$ \emph{i.e.}\ if $1\leq i \leq m+1$,
$x\in \Gamma_{i}(H)$, $x^{2^{m+1-i-k}}$ is a generator of $\mathcal E_{i,m+1}$, $0\leq k \leq m+1-i$ and $y\in H$, then: 
\begin{equation}\label{cong1}
[x^{2^{m+1-i-k}},y]\equiv [x,y]^{2^{m+1-i-k}} \bmod{\mathcal{E}_{i+1,m+2}}.
\end{equation}
%
%
This will be achieved by making a second induction hypothesis as follows.

\smallskip

\noindent
\textbf{Induction hypothesis 2:} let $1\leq i \leq m+1$ be such that: 
\begin{equation}\label{cong2}
[x^{2^{m+1-j-k}},y]\equiv [x,y]^{2^{m+1-j-k}} \bmod{\mathcal{E}_{j+1,m+2}}
\end{equation}
for all $i\leq j\leq m+1$, where $x\in \Gamma_{j}(H)$, $x^{2^{m+1-j-k}}$ is a generator of $\mathcal E_{j,m+1}$, $0\leq k \leq m+1-j$ and $y\in H$. If $i=m+1$ then~(\ref{cong2}) holds trivially. So suppose that~(\ref{cong2}) is valid for some $2\leq i \leq m+1$, and let us prove by reverse induction on $i$ that it also holds for $i-1$. By induction hypothesis~2, it suffices to prove~(\ref{cong2}) for $j=i-1$. Let $y\in H$, and let $x= [x_{1},\ldots, x_{i-1}]\in \Gamma_{i-1}(H)$, where $x^{2^{m+1-(i-1)-k}}$ is a generator of $\mathcal E_{i-1,m+1}$ and $0\leq k \leq m+1-(i-1)$. Then $[x^{2^{(m+1)-(i-1)-k}},y]=[x^{2^{m+2-i-k}},y]$. If $k=m+2-i$ 
then~(\ref{cong2}) is an equality. So assume that $0\leq k\leq m+1-i$. By Proposition~\ref{colchete}, $[x^{2^{(m+1)-(i-1)-k}},y]$ is equal to: 
%
\begin{equation}\label{[x,y]}
\underbrace{[x,x,x^{2},\ldots,x^{2^{m+1-i-k}},y][x,x^{2},\ldots,x^{2^{m+1-i-k}},y]^{2}\cdots[x^{2^{m-i-k}},x^{2^{m+1-i-k}},y]^{2}}_{(\ast)}\underbrace{[x^{2^{m+1-i-k}},y]^{2}}_{(\ast\ast)}.
\end{equation}
Using induction hypotheses~1 and~2, we will first show that the expression~$(\ast\ast)$ is congruent to $[x,y]^{2^{m+2-i-k}}$ modulo $\mathcal{E}_{i,m+2}$, and then that the expression~$(\ast)$ belongs to $\mathcal{E}_{i,m+2}$, from which we will conclude by induction that $[x^{2^{n+2-i-k}},y]\equiv [x,y]^{2^{n+2-i-k}}$ modulo $\mathcal{E}_{i,m+2}$ for all $i=1,\ldots,m+1$.

To show that $[x^{2^{m+1-i-k}},y]^{2}\equiv[x,y]^{2^{m+2-i-k}}$ modulo $\mathcal{E}_{i,m+2}$, since $x=[x_{1},\ldots,x_{i-1}]\in \Gamma_{i-1}(H)$ and $x^{2^{m+1-(i-1)-k}}$ is a generator of $\mathcal E_{i-1,m+1}$, there exist $1\leq j_{1}<\cdots<j_{k}\leq i-1$ such that $x_{j_{1}},\ldots, x_{j_{k}}\in A$, and therefore $[x_{1},\ldots,x_{i-1}]^{2^{m-(i-1)-k}}$ is a generator of $\mathcal E_{i-1,m}$.
Applying induction hypothesis~1, we have:
\[[x^{2^{m+1-i-k}},y]\equiv[x,y]^{2^{m+1-i-k}}\bmod{\mathcal{E}_{i,m+1}}.\]
It follows that there exist generators $\alpha_{1},\ldots,\alpha_{t}$ 
of $\mathcal{E}_{i,m+1}$ and $\delta_{1},\ldots, \delta_{t}\in\left\{1,-1\right\}$ for which $[x^{2^{m+1-i-k}},y]=[x,y]^{2^{m+1-i-k}}\alpha^{\delta_{1}}_{1}\cdots\alpha^{\delta_{t}}_{t}$, and so:
\begin{align}
[x^{2^{m+1-i-k}},y]^{2}=&([x,y]^{2^{m+1-i-k}}\alpha^{\delta_{1}}_{1}\cdots\alpha^{\delta_{t}}_{t})^{2} 
=[x,y]^{2^{m+2-i-k}}\bigl[[x,y]^{-2^{n+1-i-k}},\alpha^{\delta_{1}}_{1}\cdots\alpha^{\delta_{t}}_{t}\bigr]\cdot\notag\\
&\alpha^{2{\delta_{1}}}_{1} [\alpha^{-{\delta_{1}}}_{1},\alpha^{\delta_{2}}_{2}\cdots\alpha^{\delta_{t}}_{t}]\alpha^{2{\delta_{2}}}_{2}\cdots [\alpha^{-{\delta_{t-1}}}_{t-1},\alpha^{\delta_{t}}_{t}]\alpha^{2{\delta_{t}}}_{t}.\label{congg}
\end{align}
We claim that:
\begin{enumerate}
\item\label{it:claim1} $\alpha^{2{\delta_{1}}}_{1},\ldots,\alpha^{2{\delta_{t}}}_{t}\in \mathcal{E}_{i,m+2}$.
	
\item\label{it:claim2} $[\alpha_{s},\theta],[\alpha^{-1}_{s},\theta]\in \mathcal{E}_{i,m+2}$ for all $s=1,\ldots,t$ and all $\theta\in H$.


\item\label{it:claim3} $\bigl[[x,y]^{-2^{m+1-i-k}},\theta\bigr]\in \mathcal{E}_{i,m+2}$ for all $\theta\in H$. 
\end{enumerate}
Claim~(\ref{it:claim1}) follows from~(\ref{eq:defekn}). To prove~(\ref{it:claim2}), let $s\in\{1,\ldots, t\}$. Since $\alpha_{s}$ is a generator of $\mathcal{E}_{i,m+1}$, $\alpha_{s}=\beta^{2^{m+1-l-\widetilde{k}}}$, where $\beta=[\beta_1,\ldots,\beta_l]\in \Gamma_{l}(H)$, $\beta_{j_{1}},\ldots,\beta_{j_{\widetilde{k}}}\in A$, $i\leq l\leq m+1$ and $0\leq \widetilde{k}\leq m+1-l$.
By induction hypothesis~2,
\begin{equation}\label{eq:alsth}
[\alpha_{s},\theta]=[\beta^{2^{m+1-l-\widetilde{k}}},\theta]\equiv[\beta,\theta]^{2^{m+1-l-\widetilde{k}}}\bmod{\mathcal{E}_{l+1,m+2}}.
\end{equation}
Moreover, $[\theta,\beta]=[\theta,\beta_{1},\ldots,\beta_{l}]\in\Gamma_{l+1}(H)$, where $\beta_{j_{1}},\ldots,\beta_{j_{\widetilde{k}}}\in A$, so $[\theta,\beta]^{2^{(m+2)-(l+1)-\widetilde{k}}}\in \mathcal{E}_{l+1,m+2}$ since $i\leq l+1\leq m+2$ and $0\leq \widetilde{k}\leq (m+2)-(l+1)$. Hence:
\begin{equation*}
[\beta,\theta]^{2^{m+1-l-\widetilde{k}}}= \bigl([\theta,\beta]^{2^{m+1-l-\widetilde{k}}}\bigr)^{-1}\in\mathcal{E}_{l+1,m+2}.
\end{equation*}
The fact that $\mathcal{E}_{l+1,m+2}\subset \mathcal{E}_{i,m+2}$ implies that the congruence~(\ref{eq:alsth}) is also valid modulo $\mathcal{E}_{i,m+2}$, from which it follows using~(\ref{eq:alsth}) that $[\alpha_{s},\theta]\in \mathcal{E}_{l+1,m+2} \subset \mathcal{E}_{i,m+2}$.
Further, $[\alpha^{-1}_{s},\theta]=\alpha^{-2}_{s}[\alpha_{s},\theta](\theta \alpha^{2}_{s}\theta^{-1}) \in \mathcal{E}_{i,m+2}$, using also~(\ref{it:claim1}) and the fact that $\mathcal{E}_{i,m+2}$ is normal in $H$, which proves~(\ref{it:claim2}).
To prove~(\ref{it:claim3}), since $x^{2^{m+1-(i-1)-k}}\in \mathcal{E}_{i-1,m+1}$, where $x=[x_{1},\ldots,x_{i-1}]\in \Gamma_{i-1}(H)$ and $0\leq k\leq m+1-i$, and there exist $1\leq j_{1}<\cdots<j_{k}\leq i-1$ such that $x_{j_{1}},\ldots, x_{j_{k}}\in A$, it follows that $[y,x]=[y,x_{1},\ldots,x_{i-1}]\in \Gamma_{i}(H)$ and $[y,x]^{2^{(m+1)-i-k}}\in \mathcal{E}_{i,m+1}$. 
So applying induction hypothesis~2, we have:
\begin{equation}\label{eq:ei1n2}
\bigl[[x,y]^{-2^{m+1-i-k}},\theta\bigr]=\bigl[[y,x]^{2^{m+1-i-k}},\theta\bigr]\equiv\bigl[[y,x],\theta\bigr]^{2^{m+1-i-k}}\bmod{\mathcal{E}_{i+1,m+2}}.
\end{equation} 
Now
\begin{equation*}
\left[[y,x],\theta\right]=[\theta, y,x]^{-1}= \left[[y,x],\theta\right]=[\theta, y,x_{1},\ldots,x_{i-1}]^{-1},
\end{equation*}
and since $[\theta, y,x_{1},\ldots,x_{i-1}]^{2^{(m+2)-(i+1)-k}}\in \mathcal{E}_{i+1,m+2}$, we conclude that $\bigl[[x,y]^{-2^{m+1-i-k}},\theta\bigr]\in \mathcal{E}_{i+1,m+2} \subset \mathcal{E}_{i,m+2}$, and this proves claim~(\ref{it:claim3}).
Thus it follows from~(\ref{congg}) and claims~(\ref{it:claim1}),~(\ref{it:claim2}) and~(\ref{it:claim3}) that: \[[x^{2^{m+1-i-k}},y]^{2}\equiv[x,y]^{2^{m+2-i-k}}\bmod{\mathcal{E}_{i,m+2}},\] which proves that the expression~$(\ast\ast)$ is congruent to $[x,y]^{2^{m+2-i-k}}$ modulo $\mathcal{E}_{i,m+2}$.

To see that the expression~$(\ast)$ belongs to $\mathcal{E}_{i,m+2}$, notice that each of its terms is a commutator, so can be written as a product of conjugates of the element $\bigl[x^{2^{m-i-k}},[x^{2^{m+1-i-k}},y]\bigr]$ or its inverse. Since $\mathcal{E}_{i,m+2}$ is normal in $H$, it thus suffices to show that $\bigl[\theta,[x^{2^{m+1-i-k}},y]\bigr]^{-1}= \bigl[[x^{2^{m+1-i-k}},y],\theta\bigr]\in \mathcal{E}_{i,m+2}$ 
for all $y,\theta\in H$ and for all $x=[x_{1},\ldots, x_{i-1}]\in \Gamma_{i-1}(H)$ for which $x^{2^{m+1-(i-1)-k}}\in \mathcal{E}_{i-1,m+1}$, where $i\geq 3$, $0\leq k \leq m+1-i$, $x_{j_{1}},\ldots,x_{j_{k}}\in A$ and $1\leq j_{i}<\cdots< j_{k}\leq i-1$. To do so, note that $x^{2^{m+1-i-k}} =x^{2^{m-(i-1)-k}}\in \mathcal{E}_{i-1,m}$, so $[x^{2^{m+1-i-k}},y]\equiv[x,y]^{2^{m+1-i-k}}\bmod{\mathcal{E}_{i,m+1}}$ by induction hypothesis~1.
Thus there exist generators $\alpha_{1},\ldots,\alpha_{t}$ of $\mathcal{E}_{i,m+1}$, and $\delta_{1},\ldots, \delta_{t}\in \left\{1,-1\right\}$ such that $[x^{2^{m+1-i-k}},y]=[x,y]^{2^{m+1-i-k}}\alpha^{\delta_{1}}_{1}\cdots\alpha^{\delta_{t}}_{t}$, and hence:
%
\begin{align}
\bigl[[x^{2^{m+1-i-k}},y], \theta\bigr]&=\bigl[[x,y]^{2^{m+1-i-k}}\alpha^{\delta_{1}}_{1}\cdots\alpha^{\delta_{t}}_{t},\theta\bigr]\notag\\
&=\bigl([x,y]^{2^{m+1-i-k}}[\alpha^{\delta_{1}}_{1}\cdots\alpha^{\delta_{t}}_{t},\theta][x,y]^{-2^{m+1-i-k}}\bigr)\bigl[[x,y]^{2^{m+1-i-k}},\theta\bigr].\label{eq:xytheta}
\end{align}
Further, $[\alpha^{\delta_{1}}_{1}\cdots\alpha^{\delta_{t}}_{t},\theta]$ belongs to the normal closure of $\bigl\{ \left[\alpha^{\delta_{1}}_{1},\theta\right],\ldots,\left[\alpha^{\delta_{t}}_{t},\theta\right]\bigr\}$ in $H$. This may be seen by applying reverse induction on $1\leq s \leq t$, and by noting that for $s\geq 2$, 
\begin{equation}\label{eq:conjalphat}
[\alpha^{\delta_{s-1}}_{s-1}\alpha^{\delta_{s}}_{s}\cdots\alpha^{\delta_{t}}_{t},\theta]=(\alpha^{\delta_{s-1}}_{s-1}[\alpha^{\delta_{s}}_{s}\cdots\alpha^{\delta_{t}}_{t},\theta]\alpha^{-\delta_{s-1}}_{s-1})[\alpha^{\delta_{s-1}}_{s-1},\theta].
\end{equation}
Then $\bigl[[x^{2^{n+1-i-k}},y], \theta\bigr]\in \bigl\langle\!\bigl\langle \bigl[[x,y]^{2^{n+1-i-k}},\theta\bigr],\left[\alpha^{\delta_{1}}_{1},\theta\right],\ldots,\left[\alpha^{\delta_{t}}_{t},\theta\right] \bigr\rangle\!\bigr\rangle_{H}$ by~(\ref{eq:xytheta}) and~(\ref{eq:conjalphat}). 
Now by claims~(\ref{it:claim2}) and~(\ref{it:claim3}), the elements
$[\alpha^{\delta_{1}}_{1},\theta],\ldots,[\alpha^{\delta_{t}}_{t},\theta]$ and $\bigl[[x,y]^{2^{m+1-i-k}},\theta\bigr]$ belong to $\mathcal{E}_{i,m+2}$, and
since $\mathcal{E}_{i,m+2}$ is normal in $H$, we conclude that the expression~$(\ast)$ belongs to $\mathcal{E}_{i,m+2}$. This completes the proof of~(\ref{cong2}) for $i-1$, and so by induction,~(\ref{cong2}) holds for all $1\leq i \leq m+1$, which is exactly~(\ref{cong1}), By induction, we conclude that~(\ref{cong}) holds, and this completes the proof of the lemma.
\end{proof}

Lemma~\ref{lemaprinc} has the following consequences. 

\begin{coro}\label{corolario} For all $1\leq l \leq m$, $[\mathcal{E}_{l,m}, H]\subset \mathcal{E}_{l+1,m+1}$ and $[\widetilde{\mathcal{E}}_{l,m},H]\subset \widetilde{\mathcal{E}}_{l+1,m+1}$.
\end{coro}

\begin{proof}
Let $l\leq i \leq m$. First assume that $x=[x_1,\ldots,x_{i}]\in\Gamma_{i}(H)$ is such that $x^{2^{m-i-k}}$ is a generator of $\mathcal{E}_{l,m}$, and let $y\in H$. Then by~(\ref{cong}), $[x^{2^{m-i-k}},y]\equiv [x,y]^{2^{m-i-k}}$ mod $\mathcal{E}_{l+1,m+1}$. Since $[y,x]=[y,x_{1},\ldots,x_{i}]\in\Gamma_{i+1}(H)$, it follows that $[y,x]^{2^{(m+1)-(i+1)-k}}\in \mathcal{E}_{l+1,m+1}$, so $[x^{2^{m-i-k}},y]\in\mathcal{E}_{l+1,m+1}$. 

Now suppose that $x=\alpha^{\delta_{1}}_{1}\cdots\alpha^{\delta_{t}}_{t}$, where for all $i=1,\ldots, t$, $\alpha_{i}$ is a generator of $\mathcal{E}_{l,m}$, and $\delta_{i}\in\left\{1,-1\right\}$. Then $[x,y]$ belongs to the normal closure of $\left\{[\alpha^{\delta_{i}},y], i=1,\ldots,t\right\}$ in $H$ by (\ref{eq:conjalphat}), and so $[x,y]\in\mathcal{E}_{l+1,m+1}$ for all $x\in\mathcal{E}_{l,m}$, $y\in H$ by the first paragraph of the proof and the fact that $\mathcal{E}_{l+1,m+1}$ is normal in $H$. Once more, the result for $\widetilde{\mathcal{E}}_{l,m}$ is obtained from that for $\mathcal{E}_{l,m}$ by taking $A=\varnothing$.
\end{proof}

\begin{coro}\label{corolario1} If $m\geq 2$, then $\widetilde{W}_{m}=\bigl\langle \!\bigl\langle a^{2^{m-1}}_{2},\,\widetilde{\mathcal{E}}_{2,m}\bigr\rangle \! \bigr\rangle_{H}$.
\end{coro}

\begin{proof} It suffices to prove that 
\begin{equation}\label{eq:E2m}
\widetilde{\mathcal{E}}_{2,m}=\bigl\langle x^{2^{m-i}}\,:\,x\in\Gamma_{i}(H),\, 2\leq i \leq m\bigr\rangle
\end{equation}
for all $m\geq 2$. If $m=2$,~(\ref{eq:E2m}) follows from~(\ref{L2}) and~(\ref{eq:defekn1}). Suppose by induction on $m$ that~(\ref{eq:E2m}) holds for some $m\geq 2$. It is clear from~(\ref{eq:defekn1}) that $\widetilde{\mathcal{E}}_{2,m+1}\subset \bigl\langle x^{2^{m-i}}\,:\,x\in\Gamma_{i}(H),\, 2\leq i \leq m\bigr\rangle$. To prove the converse, let $y=x^{2^{m+1-i}}$, where $x\in\Gamma_{i}(H)$ and $2\leq i \leq m+1$. If $i=m+1$ then $y\in\Gamma_{m+1}(H)\subset \widetilde{\mathcal{E}}_{2,m+1}$ by~(\ref{eq:inclElm}). So suppose that $2\leq i \leq m$. Then $x^{2^{m-i}}\in\widetilde{\mathcal{E}}_{2,m}$ by the induction hypothesis, so $x^{2^{m-i}}=\alpha^{\delta_{1}}_{1}\cdots\alpha^{\delta_{t}}_{t}$, where for all $i=1,\ldots, t$, $\alpha_{i}$ is a generator of $\widetilde{\mathcal{E}}_{2,m}$ and $\delta_{i}\in\left\{1,-1\right\}$. Hence:
\begin{equation*}
y\!=\!(x^{2^{m-i}})^{2}\!=\!(\alpha^{\delta_{1}}_{1}\cdots\alpha^{\delta_{t}}_{t})^{2}=\alpha^{2\delta_{1}}_{1}[\alpha^{-\delta_{1}}_{1},\alpha^{\delta_{2}}_{2}\cdots\alpha^{\delta_{t}}_{t}]\alpha^{2\delta_{2}}_{2}[\alpha^{-\delta_{2}}_{2},\alpha^{\delta_{3}}_{3}\cdots\alpha^{\delta_{t}}_{t}]\cdots[\alpha^{-\delta_{t-1}}_{t-1},\alpha^{\delta_{t}}_{t}]\alpha^{2\delta_{t}}_{t}.
\end{equation*}
So $y\in\widetilde{\mathcal{E}}_{2,m+1}$ because for all $i=1,\ldots,t$, $\alpha^{2\delta_{i}}_{i}\in\widetilde{\mathcal{E}}_{2,m+1}$ by~(\ref{eq:defekn1}), and $[\alpha^{-\delta_{i}}_{i},\theta]\in[\widetilde{\mathcal{E}}_{2,m},H]\subset \widetilde{\mathcal{E}}_{3,m+1}\subset \widetilde{\mathcal{E}}_{2,m+1}$ by Corollary~\ref{corolario} and~(\ref{eq:inclElm}). The inclusion $\bigl\langle x^{2^{m-i}}\,:\,x\in\Gamma_{i}(H),\, 2\leq i \leq m\bigr\rangle \subset \widetilde{\mathcal{E}}_{2,m+1}$ then follows.
\end{proof}

\begin{coro}\label{coro}
For all $m\geq 2$, $[H,\widetilde{W}_{m}]\subset \widetilde{W}_{m+1}$.
\end{coro}

\begin{proof}
Since $[\widetilde{\mathcal{E}}_{2,m},H]\subset \widetilde{\mathcal{E}}_{2,m+1}$ using Corollary~\ref{corolario}, by Corollary~\ref{corolario1}, it suffices to prove that $[h,a^{2^{m-1}}_{2}]\in \widetilde{W}_{m+1}$ for all $h\in H$. To see this, observe that $[h,a^{2^{m-1}}_{2}]\equiv [h,a_{2}]^{2^{m-1}}\bmod{\widetilde{W}_{m+1}}$ by taking $i=1$ and $x=a_{2}$ in~(\ref{congWn}) and using Remark~\ref{rem:inclWn}. So there exists $\widetilde{w}\in \widetilde{W}_{m+1}$ such that $[h,a^{2^{m-1}}_{2}]=[h,a_{2}]^{2^{m-1}}\ldotp \widetilde{w}$. Now $[h,a_{2}]\in\Gamma_{2}(H)$, hence $[h,a_{2}]^{2^{m-1}}\in\widetilde{W}_{m+1}$ by~(\ref{eq:defwntilde}), and we conclude that $[h,a^{2^{m-1}}_{2}]\in\widetilde{W}_{m+1}$ as required.
\end{proof}

The following result will enable us to obtain the explicit characterisation of $\Gamma_{n}(P_{2}(\mathbb{K}))$ given in Theorem~\ref{GammaP2}.

\begin{prop}\label{lemma3}
For all $n\geq 2$, $L_{n}=W_{n}=\widetilde{W}_{n}$.
\end{prop}

\begin{proof} If $n=2$, the statement is true by definition. So suppose by induction that $L_{n}=W_{n}=\widetilde{W}_{n}$ for some $n\geq 2$. Then we have the following inclusions:

\begin{itemize}
\item $\widetilde{W}_{n+1}\subset W_{n+1}$. To see this, let $y\in \widetilde{W}_{n+1}$. If $y$ is a generator of $\widetilde{W}_{n+1}$, then by~(\ref{eq:defwntilde}), $y$ either belongs to $\Gamma_{n+1}(H)$, and so belongs to $W_{n+1}$ by~(\ref{eq:defwn}), or is of the form $x^{2^{n+1-i}}$, where $1\leq i\leq n$, $x\in \Gamma_{i}(H)$, and $x=a_{2}$ if $i=1$. Hence $x^{2^{n-i}}\in \widetilde{W}_{n}=W_{n}$ by induction, and thus $y=x^{2^{n+1-i}}=(x^{2^{n-i}})^{2}\in W_{n+1}$ by~(\ref{eq:defwn}). If $y$ is an arbitrary element of $\widetilde{W}_{n+1}$, it may be written as a product of conjugates of generators and their inverses, so it belongs to $W_{n+1}$ because $W_{n+1}$ is normal in $H$.

	
\item $L_{n+1}\subset \widetilde{W}_{n+1}$, since $K_{n+1}\subset \widetilde{W}_{n+1}$ by Lemma~\ref{lemma1} and~(\ref{eq:defwntilde}), and $[H,L_{n}]=[H,\widetilde{W}_{n}]\subset \widetilde{W}_{n+1}$ by induction and Corollary~\ref{coro}, so $L_{n+1}=\left\langle K_{n+1}, [H,L_{n}]\right\rangle\subset \widetilde{W}_{n+1}$
using Lemma~\ref{lemma2}.
	
\item $W_{n+1}\subset L_{n+1}$. First, $\Gamma_{n+1}(H)=[H,\Gamma_{n}(H)] \subset [H,W_{n}]=[H,L_{n}]\subset L_{n+1}$ by~(\ref{eq:defwn}) and induction. Secondly, let $x^{2}\in W_{n+1}$, where $x\in W_{n}=L_{n}$ is a generator. If $n=2$, by (\ref{L2}), $x=a^{2}_{2}$ or $x=[b_{2},a_{2}]$, then $x^{2}=a^{4}_{2}\in K_{3}\subset L_{3}$ by Lemma \ref{lemma1}, or $x^{2}=[b_{2},\underbrace{a^{2}_{2}}_{\in L_{2}}][a_{2},\underbrace{[a_{2},b_{2}]}_{\in L_{2}}]\in [H,L_{2}]\subset L_{3}$. So assume that $n\geq 3$. By Lemma~\ref{lemma2}, there are two possibilities for $x$:
\begin{itemize}
\item if $x\in K_{n}$, then  $x=a_{2}^{2^{n-1}}$ by Lemma~\ref{lemma1}, and $x^{2}=a_{2}^{2^{n}}\in K_{n+1}\subset L_{n+1}$.
\item if $x=[h,l]\in [H,L_{n-1}]$, where $h\in H$ and $l\in L_{n-1}=W_{n-1}$, then $x^{2}=[h,l]^{2}=[h,l^{2}][l,[l,h]]$. Now $l^{2}\in W_{n}=L_{n}$ by induction and $[l,h]\in L_{n}$, so $[h,l^{2}]$ and $[l,[l,h]]$ belong to $[H,L_{n}]$, which is contained in $L_{n+1}$.
\end{itemize}
Finally, let $x^{2}\in W_{n+1}$, where $x$ is an arbitrary element of $W_{n}$. Then there exists $q\in \mathbb{N}$ such that $x=x_{1}\cdots x_{q}$, where for all $i=1,\ldots, q$, $x_{i}$ is a generator of $W_{n}=L_{n}$. Then as in~(\ref{congg}), we have:
\begin{equation}\label{produto}
x^{2}=x^{2}_{1} [x^{-1}_{1},x_{2}\cdots x_{q}] x^{2}_{2} [x^{-1}_{2},x_{3}\cdots x_{q}]\cdots [x^{-1}_{q-1},x_{q}] x^{2}_{q}.
\end{equation}
From the second case above, for all $i=1,\ldots, q$, $x_{i}^{2}\in L_{n+1}$. Further, for all $i=1,\ldots, q-1$, $[x_{i}^{-1},x_{i+1}\cdots x_{q}]\in [L_{n},H]\subset L_{n+1}$. It then follows from~(\ref{produto}) that  $x^{2}\in L_{n+1}$.
\end{itemize}
It follows from these three inclusions that $L_{n+1}=W_{n+1}=\widetilde{W}_{n+1}$.
\end{proof} 

\begin{proof}[Proof of Theorem~\ref{GammaP2}]
The result is a consequence of Theorem~\ref{serie}(\ref{it:lcspart1}), Proposition~\ref{lemma3}, and equations~(\ref{eq:semiP2K}),~(\ref{eq:pi1}) and~(\ref{eq:defwntilde}).\end{proof} 

\subsection{The lower $\mathbb{F}_{2}$-linear central filtration of $P_{2}(\mathbb{K})$}\label{sec:lcfP2}

Now that we have a good description of $\Gamma_{n}(P_{2}(\mathbb{K}))$, we may obtain the following decomposition of $\gamma^{2}_{n}(P_{2}(\mathbb{K}))$.

\begin{teo}\label{gammaP2} Let $n\geq 2$. With the notation of Theorem~\ref{puras}, $\gamma^{2}_{n}P_{2}(\mathbb{K})$ is equal to: \[ \bigl\langle \!\bigl\langle a^{2^{n-1}}_{2}, b^{2^{n-1}}_{2}, x^{2^{n-i}}\,:\,x\in \Gamma_{i}(\pi_{1}(\mathbb{K}\setminus\left\{x_{1}\right\})),\,2\leq i \leq n\bigr\rangle \! \bigr\rangle_{\pi_{1}(\mathbb K\setminus\left\{x_{1}\right\})}\rtimes\bigl\langle (a_{1}a_{2})^{2^{n-1}}, (b_{2}b_{1})^{2^{n-1}}\bigr\rangle.\] Consequently, $P_{2}(\mathbb{K})$ is residually $2$-finite.
\end{teo}

\begin{remark}\label{Un} 
Using the action given by~(\ref{acaoP2}) and the description of $\Gamma_{n}(P_{2}(\mathbb K))$ given by Theorem~\ref{GammaP2}, it is straightforward to see that: \[ \bigl\langle \!\bigl\langle a^{2^{n-1}}_{2}, b^{2^{n-1}}_{2}, x^{2^{n-i}}\,:\,x\in \Gamma_{i}(\pi_{1}(\mathbb{K}\setminus\left\{x_{1}\right\})),\,2\leq i \leq n\bigr\rangle \! \bigr\rangle_{\pi_{1}(\mathbb K\setminus\left\{x_{1}\right\})}\rtimes\bigl\langle (a_{1}a_{2})^{2^{n-1}}, (b_{2}b_{1})^{2^{n-1}}\bigr\rangle.\] is equal to $U_{n}$, where: 
\begin{equation}\label{eq:defUn}
U_{n}=\bigl\langle \! \bigl\langle \Gamma_{n}(P_{2}(\mathbb{K})),\, b^{2^{n-1}}_{2},(b_{2}b_{1})^{2^{n-1}}\bigr\rangle \!\bigr\rangle_{P_{2}(\mathbb K)}.
\end{equation}
\end{remark}

\begin{proof}[Proof of Theorem~\ref{gammaP2}] 
Let $n\geq 2$. By Remark~\ref{Un}, it suffices to prove by induction on $n$ that $U_{n}=\gamma^{2}_{n}P_{2}(\mathbb{K})$, where $U_{n}$ is defined by~(\ref{eq:defUn}).
If $n=2$, $U_{2}\subset\gamma^{2}_{2}P_{2}(\mathbb{K})$ because $b^{2}_{2}$ and $(b_{2}b_{1})^{2}$ belong to $\left\{x^{2}\,:\,x\in P_{2}(\mathbb{K})\right\}$, which is contained in $ \gamma^{2}_{2}P_{2}(\mathbb{K})$, $\Gamma_{2}(P_{2}(\mathbb{K}))$ is contained in $\gamma^{2}_{2}P_{2}(\mathbb{K})$, and $\gamma^{2}_{2}P_{2}(\mathbb K)$ is normal in $P_{2}(\mathbb K)$. 
For the converse inclusion, we know that $\Gamma_{2}(P_{2}(\mathbb{K}))\subset U_{2}$, that $P_{2}(\mathbb{K})$ is generated by $\left\{a_{2},a_{2}a_{1}, b_{2},b_{2}b_{1}\right\}$ by Remark~(\ref{P2}), and that the square of each element of this set belongs to $U_{2}$, since by Theorem~\ref{GammaP2}, $a^{2}_{2}\in \Gamma_{2}(P_{2}(\mathbb{K}))$ and $(a_{2}a_{1})^{2}\in\Gamma_{2}(s(\pi_{1}(\mathbb{K})))\subset\Gamma_{2}(P_{2}(\mathbb{K}))$. Further, if $x=x_{1}\cdots x_{q}$, where for $i=1,\ldots, q$, $x_{i}\in \left\{a^{\pm1}_{2},(a_{2}a_{1})^{\pm1}, b^{\pm1}_{2},(b_{2}b_{1})^{\pm1}\right\}$, then using the decomposition given in~(\ref{produto}), we have $x^{2}\in U_{2}$ because $x_{i}^{2}\in U_{2}$ for all $1\leq i\leq q$ and $[x^{-1}_{i},x_{i+1}\cdots x_{q}]\in \Gamma_{2}(P_{2}(\mathbb{K}))\subset U_{2}$ for all $1\leq i\leq q-1$.
Thus $\gamma^{2}_{2}P_{2}(\mathbb{K})\subset U_{2}$, and therefore $\gamma^{2}_{2}P_{2}(\mathbb{K})=U_{2}$. 

Now assume that $n\geq 2$, and suppose by induction that $U_{i}=\gamma^{2}_{i}P_{2}(\mathbb{K})$ for all $2\leq i \leq n$. Then $U_{n+1}\subset \gamma^{2}_{n+1}P_{2}(\mathbb{K})$ since $\Gamma_{n+1}(P_{2}(\mathbb{K})) \subset \gamma^{2}_{n+1}P_{2}(\mathbb{K})$, $\Gamma_{n+1}(P_{2}(\mathbb{K})) $ and $\gamma^{2}_{n+1}P_{2}(\mathbb{K})$ are normal in $P_{2}(\mathbb{K})$, and if $x\in \{b_{2}, b_{2}b_{1}\}$ then $x^{2^{n}}=(x^{2^{n-1}})^{2}$, where $x^{2^{n-1}}\in U_{n}=\gamma^{2}_{n}P_{2}(\mathbb{K})$ by the induction hypothesis, so $x^{2^{n}}\in \gamma^{2}_{n+1}P_{2}(\mathbb{K})$. 

To prove that $\gamma^{2}_{n+1} P_{2}(\mathbb K) \subset U_{n+1}$, using the induction hypothesis and the fact that $\gamma^{2}_{n+1} P_{2}(\mathbb K)$ is generated by $[P_{2}(\mathbb K), \gamma^{2}_{n} P_{2}(\mathbb K)]\cup \left\{x^{2}\,:\,x\in \gamma^{2}_{n} P_{2}(\mathbb K)\right\}$, it suffices to show that $[P_{2}(\mathbb K), U_{n}]\subset U_{n+1}$, and that $x^{2}\in U_{n+1}$ for all $x\in U_{n}$. We first show that $[P_{2}(\mathbb{K}),U_{n}]\subset U_{n+1}$.  Let $x\in P_{2}(\mathbb{K})$, and let $u\in U_{n}$. 
\begin{enumerate}
	\item If $u\in\Gamma_{n}(P_{2}(\mathbb{K}))$ then $[x,u]\in U_{n+1}$. 
	
	\item If $u=(b_{2}b_{1})^{2^{n-1}}$ then $[x,u]=1$ because $(b_{2}b_{1})^{2^{n-1}}\in Z(P_{2}(\mathbb{K}))$ by Proposition~\ref{centro}.
	
	\item\label{it:ub2} If $u=b^{2^{n-1}}_{2}$, we claim that $[b^{2^{n-1}}_{2},x]\in L_{n+1}$ for all $x\in\left\{a_{2},(a_{2}a_{1})^{-1}, b_{2},(b_{2}b_{1})^{-1}\right\}$. The result is clear if $x=b_{2}$, so we consider the three other cases. We proceed by induction on $n$. Suppose first that $n=2$. 
If $x=a_{2}$ then by~(\ref{itercomm1}),~(\ref{eq:defwn}) and Proposition~\ref{lemma3}, we have: \[[u,x]=[b^{2}_{2},a_{2}]=\underbrace{[b_{2},b_{2},a_{2}]}_{\in W_{3}=L_{3}}\underbrace{[b_{2},a_{2}]^{2}}_{\in W_{3}=L_{3}}\in L_{3}.\] 
In the remaining two cases, by~(\ref{acaoP2}),~(\ref{L2}),~(\ref{eq:defwntilde}) and Proposition~\ref{lemma3}, we have:
\begin{align*}
[(a_{2}a_{1})^{-1},u]&=[(a_{2}a_{1})^{-1},b^{2}_{2}]=(a^{-2}_{2}b_{2})(a^{-2}_{2}b_{2})b^{-2}_{2}=\underbrace{a^{-4}_{2}}_{\in \widetilde{W}_{3} = L_{3}}\underbrace{[a^{2}_{2},b_{2}]}_{\in [L_{2},H]\subset L_{3}}\in L_{3}\\
[(b_{2}b_{1})^{-1},u] &=[(b_{2}b_{1})^{-1},b^{2}_{2}]=(a_{2}b_{2}a_{2})(a_{2}b_{2}a_{2})b^{-2}_{2}\\
&=\underbrace{[a_{2}b_{2},a^{2}_{2}]}_{\in [H,L_{2}]\subset L_{3}}\underbrace{a^{4}_{2}}_{\in \widetilde{W}_{3}= L_{3}}\underbrace{[a^{-1}_{2},b_{2}]^{2}}_{\in \widetilde{W}_{3}=L_{3}}\underbrace{\left[[b_{2},a^{-1}_{2}],b_{2}\right]}_{\in \Gamma_{3}(H)\subset L_{3}} \in L_{3},
\end{align*}
which proves the claim in the case $n=2$. Now suppose that $[b^{2^{j-1}}_{2},x]\in L_{j+1}$ for all $2\leq j \leq n$ and $x\in\left\{a_{2},(a_{2}a_{1})^{-1}, b_{2},(b_{2}b_{1})^{-1}\right\}$. Then by Proposition~\ref{lemma3} and~(\ref{eq:defwn}), $[b^{2^{n-1}}_{2},x]\in L_{n+1}=W_{n+1}$, and hence $[b^{2^{n-1}}_{2},x]^{2}\in W_{n+2}=L_{n+2}$. So by~(\ref{itercomm1}), we have:
\[[b^{2^{n}}_{2},x]=\underbrace{\Bigl[b^{2^{n-1}}_{2},\underbrace{[b^{2^{n-1}}_{2},x]}_{\in L_{n+1}}\Bigr]}_{\in[H,L_{n+1}]\subset L_{n+2}}\underbrace{[b^{2^{n-1}}_{2},x]^{2}}_{\in L_{n+2}} \in L_{n+2},\] which proves the claim for all $n\geq 2$. Now let $x$ be an arbitrary element of $P_{2}(\mathbb{K})$. Since the set $\left\{a_{2},(a_{2}a_{1})^{-1}, b_{2},(b_{2}b_{1})^{-1}\right\}$ generates $P_{2}(\mathbb{K})$, for some $t\geq 0$, there exist $x_{1}, \ldots, x_{t}\in \left\{a_{2},(a_{2}a_{1})^{-1}, b_{2},(b_{2}b_{1})^{-1}\right\}$ and $\varepsilon_{1},\ldots, \varepsilon_{t}\in\left\{1,-1\right\}$ for which $x=x^{\varepsilon_{1}}_{1}x^{\varepsilon_{2}}_{2}\cdots x^{\varepsilon_{t}}_{t}$. 
As in~(\ref{eq:conjalphat}), we have the following relation:
\begin{equation}\label{x1xt}
[u,x]=[u,x^{\varepsilon_{1}}_{1}]\left(x^{\varepsilon_{1}}_{1}[u,x^{\varepsilon_{2}}_{2}]x^{-\varepsilon_{1}}_{1}\right)\cdots \left(x^{\varepsilon_{1}}_{1}x^{\varepsilon_{2}}_{2}\cdots x^{\varepsilon_{t-1}}_{t-1}[u,x^{\varepsilon_{t}}_{t}]x^{-\varepsilon_{t-1}}_{t-1}\cdots x^{-\varepsilon_{2}}_{2}x^{-\varepsilon_{1}}_{1}\right).
\end{equation}
Using the fact that $L_{n+1}$ is normal in $H$ by Lemma~\ref{Lnnormal}, it follows from~(\ref{x1xt}) that $[b^{2^{n-1}}_{2},x]\in L_{n+1}$, and since $L_{n+1}\subset \Gamma_{n+1}(P_{2}(\mathbb{K}))$ by Theorem~\ref{serie}, we deduce that $[b^{2^{n-1}}_{2},x]\in U_{n+1}$.
\end{enumerate} 
This concludes the proof of the inclusion $[P_{2}(\mathbb{K}), U_{n}]\subset U_{n+1}$. It remains to prove that $\left\{x^{2}\,:\,x\in U_{n}\right\}\subset U_{n+1}$. If $x=b^{2^{n-1}}_{2}$ or $x=(b_{2}b_{1})^{2^{n-1}}$ then clearly $x^{2}\in U_{n+1}$. Using Theorems~\ref{serie}(\ref{it:lcspart1}) and~\ref{GammaP2}, if $x\in \Gamma_{n}(P_{2}(\mathbb{K}))=L_{n}\rtimes \bigl\langle (a_{1}a_{2})^{2^{n-1}}\bigr\rangle$, then in terms of this semi-direct product, $x=(x_{1},x_{2})$, where $x_{1}\in L_{n}$ and $x_{2}\in \bigl\langle (a_{1}a_{2})^{2^{n-1}}\bigr\rangle$, and $x^{2}=(x_{1},x_{2})(x_{1},x_{2})=(x_{1}.\varphi(x_{2})(x_{1}),x^{2}_{2})$.
Now $x^{2}_{2}\in \bigl\langle (a_{1}a_{2})^{2^{n}}\bigr\rangle$, and since $x_{1}\in L_{n}=W_{n}$, we have $x^{2}_{1}\in W_{n+1}=L_{n+1}$ by Proposition~\ref{lemma3}, and therefore $x_{1}.\varphi(x_{2})(x_{1})=x^{2}_{1}(\varphi(x_{2})(x^{-1}_{1}).x_{1})^{-1}\in L_{n+1}$ and $x^{2}\in L_{n+1}\rtimes \left\langle (a_{1}a_{2})^{2^{n}}\right\rangle = \Gamma_{n+1}(P_{2}(\mathbb{K}))\subset U_{n+1}$.
If $x$ is a product of conjugates of generators of $U_{n}$ then $x^{2}\in U_{n+1}$ using~(\ref{produto}). This shows that $\left\{x^{2}\,:\,x\in U_{n}\right\}\subset U_{n+1}$. It follows that $\gamma^{2}_{n+1}(P_{2}(\mathbb{K}))\subset U_{n+1}$, so $\gamma^{2}_{n+1}(P_{2}(\mathbb{K}) = U_{n+1}$. Using the explicit characterisation of $\Gamma_{n}(P_{2}(\mathbb{K}))$ given by Theorem~\ref{GammaP2} and Remark \ref{Un}, the first part of the statement follows. The second part follows from the fact that $\bigcap_{n\geq 2}U_{n}=\left\{1\right\}$ using the first part.
\end{proof}

\subsection{The derived series of $P_{n}(\mathbb{K})$}\label{sec:derivP2K}

For $n\geq 3$, the situation is more complicated, and we did not succeed in computing the lower central series of $P_{n}(\mathbb{K})$ as we did in Section~\ref{sec:lcsP2} in the case $n=2$. Due to the complexity of the lower central series of $P_{2}(\mathbb{K})$, we were not able to find a suitable formula for $\varphi(g)(h).h^{-1}$ for $g\in\Gamma_{m}(P_{n}(\mathbb{K}))$ in order to apply Theorem~\ref{serie}(\ref{it:lcspart1}), the idea being to use the Fadell-Neuwirth short exact sequence~(\ref{seqFNklein}) and to calculate recursively $\Gamma_{m}(P_{n}(\mathbb K)$ for $m\geq1$ and $n\geq3$. Nevertheless, in this section, we will use Theorem~\ref{serie}(\ref{it:lcspart2}) to estimate the derived series of $P_{n}(\mathbb{K})$, and to show in Theorem~\ref{soluvel} that this group is residually soluble for all $n\geq 1$. 

With the notation of  Theorem~\ref{serie} and equation~(\ref{seqFNklein}), we may write $P_{n+1}(\mathbb{K})=H\rtimes_{\varphi} G$, where $H=\pi_{1}(\mathbb{K}\setminus\left\{x_{1},\ldots, x_{n}\right\})$, $G=s(P_{n}(\mathbb{K}))$, and $s$ is the section for $p_{\ast}$ given by Proposition~\ref{sectionPn}. In what follows, we will take $G$ to be equipped with the generating set: 
\begin{align}
X &=\bigl\{ s(z) \;:\; z\in \{ a_{i},b_{i}, C_{j,k} \;:\; \text{$1\leq i\leq n$ and $1\leq j<k\leq n$}\} \bigr\}\notag\\
&= \bigl\{a_{i},b_{i},C_{i,k}, a_{n}a_{n+1}, b_{n}b_{n+1}C_{n,n+1}, C_{i,n}C_{i,n+1}C^{-1}_{n,n+1} \;:\; 1\leq i\leq k\leq n-1\bigr\}.\label{eq:defX}
\end{align}
In $P_{n+1}(\mathbb{K})$, for $1\leq j \leq n$, let $D_{j}=C^{-1}_{j,n+1}C_{j+1,n+1}$. Since $C_{j,n+1}=D^{-1}_{n}\cdots D^{-1}_{j+1}D^{-1}_{j}$ for all $j=1,\ldots,n$, the set:
\begin{equation}\label{eq:defY}
Y=\left\{a_{n+1}, b_{n+1}, D_{j}\,:\,j=1,\ldots, n\right\}
\end{equation}
generates $H$,
and using Theorem~\ref{puras}, the action $\varphi\colon\thinspace s(P_{n}(\mathbb{K}))\longrightarrow \operatorname{\text{Aut}}(\pi_{1}(\mathbb{K}\setminus \left\{x_{1},\ldots,x_{n}\right\}))$ is given by:
\begin{equation}\label{acaoPn+1}
\left\{\begin{aligned}
\varphi(a_{i})(z) & \!=\!\begin{cases}
a_{n+1} & \!\!\!\text{if $z=a_{n+1}$}\\
b_{n+1}a_{n+1}D_{i}a^{-1}_{n+1} & \!\!\!\text{if $z\!=\!b_{n+1}$}\\
\alpha_{i,j} D_{j}\alpha^{-1}_{i,j} & \!\!\!\text{if $z\!=\!D_{j}$}
\end{cases}\\
\varphi(b_{i})(z) & \!=\!\begin{cases} 
a_{n+1}b_{n+1}C_{i,n+1}D_{i}C^{-1}_{i,n+1}b^{-1}_{n+1} & \!\!\!\text{if $z\!=\!a_{n+1}$}\\ 
b_{n+1}C_{i,n+1}D^{-1}_{i}C^{-1}_{i,n+1} & \!\!\!\text{if $z\!=\!b_{n+1}$}\\
\beta_{i,i} D^{-1}_{i} \beta^{-1}_{i,i} & \!\!\!\text{if $z\!=\!D_{i}$}\\
\beta_{i,j} D_{j} \beta^{-1}_{i,j} & \!\!\!\text{if $z\!=\!D_{j}$, $j\neq i$}
\end{cases}\\
\varphi(C_{i,k})(z) & \!=\!\begin{cases}
a_{n+1} & \!\!\!\text{if $z\!=\!a_{n+1}$}\\
b_{n+1} & \!\!\!\text{if $z\!=\!b_{n+1}$}\\
\delta_{i,j,k} D_{j}\delta^{-1}_{i,j,k} & \!\!\!\text{if $z\!=\!D_{j}$}
\end{cases}\\
\varphi(a_{n}a_{n+1})(z) & \!=\!\begin{cases}
a_{n+1} & \!\!\!\text{if $z\!=\!a_{n+1}$}\\
a^{-1}_{n+1}b_{n+1}a_{n+1}D_{n} & \!\!\!\text{if $z\!=\!b_{n+1}$}\\
\widetilde{\alpha}_{j} D_{j}\widetilde{\alpha}^{-1}_{j} & \!\!\!\text{if $z\!=\!D_{j}$}
\end{cases}\\
\varphi(b_{n}b_{n+1}C_{n,n+1})(z) & \!=\!\begin{cases}
D_{n}b^{-1}_{n+1}a_{n+1}b_{n+1} & \!\!\!\text{if $z\!=\!a_{n+1}$}\\ 
b_{n+1}D^{-1}_{n} & \!\!\!\text{if $z\!=\!b_{n+1}$}\\
D^{-1}_{n} & \!\!\!\text{if $z\!=\!D_{n}$}\\
b^{-1}_{n+1}D_{j}b_{n+1} & \!\!\!\text{if $z\!=\!D_{j}$, $j\neq n$}
\end{cases}\\
\varphi(C_{i,n}C_{i,n+1}C_{n,n+1}^{-1})(z) & \!=\!\begin{cases}
C_{n,n+1}C^{-1}_{i,n+1}a_{n+1}C_{i,n+1}C^{-1}_{n,n+1} & \!\!\!\!\text{if $z\!=\!a_{n+1}$}\\
C_{n,n+1}C^{-1}_{i,n+1}b_{n+1}C_{i,n+1}C^{-1}_{n,n+1} & \!\!\!\!\text{if $z\!=\!b_{n+1}$}\\
\widetilde{\delta}_{i,j} D_{j}\widetilde{\delta}^{-1}_{i,j} & \!\!\!\!\text{if $z\!=\!D_{j}$,}
\end{cases}\\
\end{aligned}\right.
\end{equation}
where:
\begin{align*}
\alpha_{i,j}&=\begin{cases}
1 & \text{if $i<j$}\\ 
C^{-1}_{j+1,n+1}a_{n+1}& \text{if $i=j$}\\
C^{-1}_{i+1,n+1}C_{i,n+1} & \text{if $i>j$}
\end{cases} &
\beta_{i,j}&=\begin{cases}
1 & \text{if $i<j$}\\ 
b_{n+1}C_{i,n+1} & \text{if $i=j$}\\
b_{n+1}C_{i,n+1}C^{-1}_{i+1,n+1}b^{-1}_{n+1} & \text{if $i>j$}
\end{cases}\\
\delta_{i,j,k}&=\begin{cases}
1 & \text{if $k<j$ or $i>j$}\\ 
C^{-1}_{j+1,n+1}C_{i,n+1} & \text{if $k=j$}\\
C^{-1}_{k+1,n+1}C_{k,n+1} & \text{if $k>j\geq i$}
\end{cases} & &
\end{align*}
%
%
for all $1\leq i\leq k \leq n-1$ and  $1\leq j \leq n$, and where $\widetilde{\alpha}_{j}=a^{-1}_{n+1}\alpha_{n,j}$ and $\widetilde{\delta}_{i,j}=C_{n,n+1}C^{-1}_{i,n+1}\delta_{i,j,n}$. 

Our aim is to calculate the subgroups $V_{m}$ that were defined in the statement of Theorem~\ref{serie} for all $m\geq 2$. For $i\geq 0$, let
\begin{equation}\label{eq:defA2in}
A^{2^i}_{n}=\bigl\langle \!\bigl\langle D^{2^i}_{j}\,:\,j=1,\ldots,n\bigr\rangle\!\bigr\rangle_{H}.
\end{equation}
If $i=0$, we write $A_{n}=A^{1}_{n}$. 
To compute $\widetilde{H}_{2}$, first notice that by definition, $\widetilde{H}_{2}=K_{2}$, so $\widetilde{H}_{2}$ is normal in $H$ by Lemma \ref{Lnnormal}.
Using~(\ref{acaoPn+1}), observe that $\varphi(a_{i})(b_{n+1})b^{-1}_{n+1}=b_{n+1}a_{n+1}D_{i}a^{-1}_{n+1}b^{-1}_{n+1}$ for all $1\leq i \leq n-1$, and $\varphi(b_{n}b_{n+1}C_{n,n+1})(b_{n+1})b^{-1}_{n+1}=b_{n+1}D^{-1}_{n}b^{-1}_{n+1}$. So by normality of $\widetilde{H}_{2}$, $D_{j}\in\widetilde{H}_{2}$ for all $1\leq j\leq n$, and
%
therefore $A_{n}\subset \widetilde{H}_{2}$ by~(\ref{eq:defA2in}). Moreover, $\varphi(b_{n}b_{n+1}C_{n,n+1})(a_{n+1})a^{-1}_{n+1}=D_{n}b^{-1}_{n+1}a_{n+1}b_{n+1}a^{-1}_{n+1}$, so $[b^{-1}_{n+1},a_{n+1}]\in \widetilde{H}_{2}$. Applying Lemma~\ref{geradoreslema}(\ref{geradoreslema1}) with $X$ and $Y$ as defined in~(\ref{eq:defX}) and~(\ref{eq:defY}), we see that $\widetilde{H}_{2}=\bigl\langle \!\bigl\langle A_{n},[a_{n+1},b_{n+1}]\bigr\rangle\!\bigr\rangle_{H}$, and therefore:
\begin{equation}\label{eq:V2}
V_{2}=\left\langle \Gamma_{2}(H), A_{n}\right\rangle
\end{equation}

Let $Y_{1}=V_{1}=H$, and for $m\geq 2$, let:
\begin{equation}\label{eq:defym}
Y_{m}= \bigl\langle A^{2^{m-2}}_{n},\, [Y_{i},Y_{k}] \;:\; \text{$1\leq i \leq k<m$, $i+k=m$}\bigr\rangle.
\end{equation}
In what follows, we will refer to the elements of the set
\begin{equation*}
\bigl\{  D^{2^{m-2}}_{j},\, [y_{i},y_{k}] \;:\; j=1,\ldots,n,\,  \text{$y_{i}\in Y_{i}$, $y_k\in Y_{k}$ and $i+k=m$}\bigr\}
\end{equation*}
as \emph{generators} of $Y_{m}$.
Since $(P_{n+1}(\mathbb{K}))^{(m)}\subset V_{m+1}\rtimes (P_{n}(\mathbb{K}))^{(m)}$ by Theorem~\ref{serie}(\ref{it:lcspart2}), to prove that $P_{n+1}(\mathbb{K})$ is residually soluble, it suffices to show that $V_{m}\subset Y_{m}$ for all $m\geq 2$, and then that $\bigcap_{m\geq1}Y_{m}=\left\{1\right\}$.

\begin{remark}\label{ym}
For all $m\geq 1$, $Y_{m}$ is a normal subgroup of $H$. Further, we claim that $\varphi(g)(Y_{i})\subset Y_{i}$ for all $i\geq 1$ and $g\in G$. To see this, note that $\varphi(g)(D_j)$ is a conjugate of $D_{j}$ or $D_{j}^{-1}$ for all $j=1,\ldots,n$ by~(\ref{acaoPn+1}), so $\varphi(g)(A^{2^{i}}_{n})\subset A^{2^{i}}_{n}$ for all $i\geq 1$. Also, if $y_{k}\in Y_{k}$ and $y_{l}\in Y_{l}$ then $\varphi(g)([y_{k},y_{l}])=[\varphi(g)(y_{k}),\varphi(g)(y_{l})]$, and the claim follows by induction on $i$.
\end{remark}

\begin{lem}\label{VmYm}
For all $m\geq 2$, $V_{m}\subset Y_{m}$.  
\end{lem}

\begin{proof} We argue by induction on $m$. If $m=2$, we saw above that $V_{2}=\left\langle \Gamma_{2}(H), A_{n}\right\rangle$, and it then follows from~(\ref{eq:defym}) that $V_{2}=Y_{2}$.
So assume that $m\geq 2$, and suppose that $V_{i}\subset Y_{i}$ for all $i=2,\ldots,m$. By definition, $V_{m+1}=\bigl\langle \widetilde{H}_{m+1},[H,V_{m}]\bigr\rangle$. From the induction hypothesis, it follows that $[H,V_{m}]\subset [H,Y_{m}]= [Y_{1},Y_{m}]\subset Y_{m+1}$. Thus it suffices to show that $\widetilde{H}_{m+1}\subset Y_{m+1}$. To do so, suppose that $x=\varphi(g)(h).h^{-1}\in \widetilde H_{m+1}$, where $g\in G$ and $h\in V_{m}$, so $h\in Y_{m}$. By Lemma~\ref{geradoreslema}(\ref{geradoreslema2}), we need only study the following cases where $g\in X$ and $h$ is a generator of $Y_{m}$:
\begin{itemize}
\item $h=D^{2^{m-2}}_{j}$, where $j\in \{1,\ldots,n\}$. By~(\ref{acaoPn+1}), if $j\in \{1,\ldots,n-1\}$ then: \[\varphi(b_{j})(D^{2^{m-2}}_{j}).D^{-2^{m-2}}_{j}=[\beta_{j,j},D^{-2^{m-2}}_{j}]D^{-2^{m-1}}_{j}\in Y_{m+1},\]
because $D^{2^{m-2}}_{j}\in Y_{m}$, and if $j=n$ then: \[\varphi(b_{n}b_{n+1}C_{n,n+1})(D^{2^{m-2}}_{n}).D^{-2^{m-2}}_{n}=D^{-2^{m-1}}_{n} \in Y_{m+1}.\]
Similarly, by~(\ref{acaoPn+1}), one may check that if $g$ is any other element of $X$ then $\varphi(g)(D^{2^{m-2}}_{j}).D^{-2^{m-2}}_{j}$ is a commutator of $D^{2^{m-2}}_{n}$ with an element of $H$, and so belongs to $[H, Y_{m}]$, which is contained in $Y_{m+1}$. 
		
\item $h=[y_{i},y_{k}]$, where $1\leq  i \leq k< m$, $i+k=m$, $y_{i}\in Y_{i}$ and $y_{k}\in Y_{k}$. If $1\leq j < m$, $x\in V_{j}$ and $g\in X$, then $\varphi(g)(x).x^{-1}\in \widetilde{H}_{j+1}\subset V_{j+1}\subset Y_{j+1}$ using the induction hypothesis, 
and $\varphi(g)(Y_{l})\subset Y_{l}$ for all $l\geq 1$ by Remark \ref{ym}. So $\left[(\varphi(g)(y^{-1}_{k}).y_{k})^{-1},\varphi(g)(y^{-1}_{i})\right]\in[Y_{k+1},Y_{i}]\subset Y_{m+1}$ and $[(\varphi(g)(y^{-1}_{i}).y_{i})^{-1},y_{k}]\in[Y_{i+1},Y_{k}]\subset Y_{m+1}$ by~(\ref{eq:defym}), and writing $\varphi(g)([y_{i},y_{k}]).[y_{k},y_{i}]$ in the form of~(\ref{eq:comutador}), where we replace $l$ by $y_{k}$ and $h$ by $y_{i}$, and using the fact that $Y_{m+1}$ is normal in $H$, it follows that $\varphi(g)([y_{i},y_{k}]).[y_{k},y_{i}]\in Y_{m+1}$.


\end{itemize}
Thus $\widetilde{H}_{m+1}\subset Y_{m+1}$, and hence $V_{m+1}\subset Y_{m+1}$ as required.\end{proof} 

It remains to prove that $\bigcap_{m\geq1}Y_{m}=\left\{1\right\}$. To do so, we define two families $(Z_{m})_{m\geq 1}$ and $(\widetilde{Z}_{m})_{m\geq 1}$ of subgroups of $H$ as follows. Let $Z_{1}=\widetilde{Z}_{1}=H$, $Z_{2}=V_{2}=Y_{2}$, and if $m\geq 3$, let:
\begin{equation}\label{eq:defZm}
\text{$Z_{m}= \bigl\langle \!\bigl\langle \left\{x^{2}\,:\,x\in Z_{m-1}\right\}\cup X_{m}\bigr\rangle\!\bigr\rangle_{H}$ and 
$\widetilde{Z}_{m-1}= \bigl\langle \!\bigl\langle A^{2^{m-3}}_{n}\cup\widetilde{X}_{m-1}\bigr\rangle\!\bigr\rangle_{H}$, where:}
\end{equation}
\begin{align*}
X_{m}& \!=\!\left\{\hspace*{-1.5mm}\begin{array}{l}[x_{1},\ldots,x_{i}]\in\Gamma_{i}(H)\!:\exists\, 1\leq j_{1}<\ldots<j_{m-i}\leq i,\,x_{j_{1}},\ldots,x_{j_{m-i}}\in A_{n}\\
\hfill \text{for all $i=2,\ldots, m$}\end{array}\hspace*{-1.5mm}\right\} \,\text{for all $m\geq 3$}\\
\widetilde{X}_{m}&\!=\!\left\{\hspace*{-1.7mm}\begin{array}{l}[x_{1},\ldots,x_{i}]^{2^{m-i-k}}\!\!\in\Gamma_{i}(H)\!:\exists\, 1\leq j_{1}<\ldots<j_{k}\leq i,\,x_{j_{1}},\ldots,x_{j_{k}}\in A_{n}\\
\hfill  \text{for all $0\leq k\leq m-i$, and $i=2,\ldots, m$}\end{array}\hspace*{-2mm}\right\} \text{for all $m\geq 2$.}
\end{align*}
Note that if $i=m$ (resp.\ $k=0$), the elements $x_{1},\ldots,x_{i}$ of $H$ that appear in the definition of $X_{m}$ (resp.\ of $\widetilde{X}_{m}$) are arbitrary.
If $m\geq 3$ (resp.\ $m\geq 2$), we will refer to the elements of $X_{m}\cup\left\{x^{2}\,:\,x\in Z_{m-1}\right\}$ (resp.\ of $\widetilde{X}_{m}\cup\bigl\{D^{2^{m-2}}_{j}\,:\,j=1,\ldots, n\bigr\}$) as \emph{generators} of $Z_{m}$ (resp.\ of $\widetilde{Z}_{m}$).

\begin{prop}\label{acima} 
Let $m\geq 1$. Then $\widetilde{Z}_{m}\subset \gamma^{2}_{\lceil{m/2}\rceil}(H)$, where $\lceil x\rceil$ denotes the least integer greater than or equal to $x$. In particular, $\bigcap_{m\geq 1} \widetilde{Z}_{m}=\left\{1\right\}$.
\end{prop}

\begin{proof}
If $m\in \{1,2\}$ then $\lceil{m/2}\rceil=1$, $\gamma^{2}_{1}(H)=H$ and thus $\widetilde{Z}_{m}\subset \gamma^{2}_{1}(H)$. So suppose by induction on $m$ that $\widetilde{Z}_{i}\subset \gamma^{2}_{\lceil{i/2}\rceil}(H)$ for some $m\geq 2$ and all $1\leq i \leq m$. Since $\gamma^{2}_{\lceil{(m+1)/2}\rceil}(H)$ is normal in $H$, by~(\ref{eq:defA2in}) and~(\ref{eq:defZm}), it suffices to show that $\{ D^{2^{m-1}}_{j}\,:\,j=1,\ldots,n \} \cup \widetilde{X}_{m+1}\subset \gamma^{2}_{\lceil{(m+1)/2}\rceil}(H)$. If $j=1,\ldots,n$ then $D^{2^{m-1}}_{j}=(D^{2^{m-2}}_{j})^{2}$, and since $D^{2^{m-2}}_{j}\in A^{2^{m-2}}_{n}$ and $A^{2^{m-2}}_{n}\subset \widetilde{Z}_{m}$, it follows by the induction hypothesis that $D^{2^{m-2}}_{j}\in \gamma^{2}_{\lceil{m/2}\rceil}(H)$, and hence $D^{2^{m-1}}_{j}\in \gamma^{2}_{\lceil{m/2}\rceil+1}(H)$. The fact that $\lceil{m/2}\rceil+1\geq \lceil{(m+1)/2}\rceil$ implies that $\gamma^{2}_{\lceil{m/2}\rceil+1}(H)\subset \gamma^{2}_{\lceil{(m+1)/2}\rceil}(H)$, whence $D^{2^{m-1}}_{j}\in \gamma^{2}_{\lceil{(m+1)/2}\rceil}(H)$. Now let $x\in\widetilde{X}_{m+1}$, and let $x=[x_{1},\ldots, x_{i}]^{2^{m+1-i-k}}\in \Gamma_{i}(H)$, where $2\leq i \leq m+1$ and there exist $1\leq j_{1}<\cdots < j_{k}\leq i$, such that $x_{j_{1}},\ldots, x_{j_{k}}\in A_{n}$, for all $0\leq k\leq m+1-i$. If $i=m+1$ then $x\in \gamma^{2}_{m+1}(H)\subset \gamma^{2}_{\lceil{(m+1)/2}\rceil}(H)$ because $m+1\geq \lceil{(m+1)/2}\rceil$. So suppose that $2\leq i \leq m$. If $0\leq k \leq m-i$ then $x=\bigl([x_{1},\ldots, x_{i}]^{2^{m-i-k}}\bigr)^{2}$, where $[x_{1},\ldots, x_{i}]^{2^{m-i-k}} \in \widetilde{X}_{m}\subset \gamma^{2}_{\lceil{m/2}\rceil}(H)$ using the induction hypothesis, and thus $x\in \gamma^{2}_{\lceil{m/2}\rceil+1}(H) \subset \gamma^{2}_{\lceil{(m+1)/2}\rceil}(H)$. Finally, if $k=m+1-i$ then $x=[x_{1},\ldots,x_{i}]\in \Gamma_{i}(H)$. Since $k\leq i$, we have $i\geq (m+1)/2\geq \lceil{(m+1)/2}\rceil$, and hence $x\in \Gamma_{i}(H)\subset \Gamma_{\lceil{(m+1)/2}\rceil}(H)\subset \gamma^{2}_{\lceil{(m+1)/2}\rceil}(H)$ as required, and this completes the proof of the inclusion $\{ D^{2^{m-1}}_{j}\,:\,j=1,\ldots,n \} \cup \widetilde{X}_{m+1}\subset \gamma^{2}_{\lceil{(m+1)/2}\rceil}(H)$.

To finish, since $H$ is finitely generated free group, $H$ is residually $2$-finite and $\bigcap_{m\geq1} \gamma^{2}_{m}(H)=\left\{1\right\}$, then $\bigcap_{m\geq 1}\widetilde{Z}_{m}=\left\{1\right\}$.
\end{proof}

The aim now is to prove that $Y_{m}=Z_{m}=\widetilde{Z}_{m}$ for all $m\geq 2$, from which we will conclude that $\bigcap_{m\geq1}Y_{m}=\left\{1\right\}$ and that $P_{n+1}(\mathbb{K})$ is residually soluble.

\begin{remark}\label{zm} Let $m\geq 2$. Taking $A=A_{n}$ in~(\ref{eq:defekn}), we have $\bigl\langle \widetilde{X}_{m}\bigr\rangle=\mathcal{E}_{2,m}$. Further, since $D_{j}\in A_{n}$ for all $j=1,\ldots,n$, we see that $D_{j}^{2^{m-2}}\in \mathcal{E}_{1,m}$, where we take $i=k=1$, so $A^{2^{m-2}}_{n}\subset\mathcal{E}_{1,m}$. Therefore $\widetilde{Z}_{m}\subset \mathcal{E}_{1,m}$, and Corollary~\ref{corolario} then implies that \[[H,\widetilde{Z}_{m}]\subset[H,\mathcal{E}_{1,m}]\subset \mathcal{E}_{2,m+1}=\bigl\langle \widetilde{X}_{m+1}\bigr\rangle \subset \widetilde{Z}_{m+1}.\]
Note also that if $m=1$ then $[H,\widetilde{Z}_{1}]=\Gamma_{2}(H)=\bigl\langle \widetilde{X}_{2}\bigr\rangle\subset \widetilde{Z}_{2}$.
%
\end{remark}

\begin{lem}\label{ymzm}
For all $m\geq 1$, $Y_{m}=Z_{m}=\widetilde{Z}_{m}$.
\end{lem}

\begin{proof}
If $m=1$, the given equality holds by definition. If $m=2$ then $\bigl\langle \widetilde{X}_{2}\bigr\rangle=\Gamma_{2}(H)$, and $\widetilde{Z}_{2}=\bigl\langle \!\bigl\langle A_{n}\cup\Gamma_{2}(H)\bigr\rangle\!\bigr\rangle_{H}=V_{2}$ using~(\ref{eq:V2}) and the fact that $V_{2}$ is normal in $H$ by Lemma~\ref{Lnnormal}.
So suppose by induction that: 
\begin{equation}\label{eq:IH1}
\text{$Y_{i}=Z_{i}=\widetilde{Z}_{i}$ for some $m\geq 2$ and all  $1\leq i \leq m$.}
\end{equation}
To prove that $\widetilde{Z}_{m+1}\subset Z_{m+1}$, let us show that $A^{2^{m-1}}_{n}$ and $\widetilde{X}_{m+1}$ are contained in $Z_{m+1}$. For all $j=1,\ldots,n$, we have $D^{2^{m+1-2}}_{j}=(D^{2^{m-2}}_{j})^{2}$ and $D^{2^{m-2}}_{j}\in \widetilde{Z}_{m}=Z_{m}$ by induction, so $A^{2^{m-1}}_{n}\subset Z_{m+1}$ by~(\ref{eq:defZm}). Now suppose that $[x_{1},\ldots, x_{i}]^{2^{m+1-i-k}}\in \widetilde{X}_{m+1}$, so $k$ elements of $\{ x_{1},\ldots, x_{i}\}$ belong to $A_{n}$, where $0\leq k\leq m+1-i$. If $m+1-i>k$ then $[x_{1},\ldots, x_{i}]^{2^{m-i-k}}$ belongs to $\widetilde{Z}_{m}=Z_{m}$ by induction, so $[x_{1},\ldots, x_{i}]^{2^{m+1-i-k}}=([x_{1},\ldots, x_{i}]^{2^{m-i-k}})^{2}$ belongs to $Z_{m+1}$. If $m+1-i=k$ then $[x_{1},\ldots, x_{i}]^{2^{m+1-i-k}}=[x_{1},\ldots, x_{i}]\in \Gamma_{i}(H)$, where $k=(m+1)-i$ elements of this commutator belong to $A_{n}$, so $[x_{1},\ldots, x_{i}] \in Z_{m+1}$. Hence $\widetilde{X}_{m+1}\subset Z_{m+1}$, and thus $\widetilde{Z}_{m+1}\subset Z_{m+1}$.


We now show that $Z_{m+1}\subset \widetilde{Z}_{m+1}$. First note that $X_{m+1}\subset \widetilde{X}_{m+1}$, for if the commutator $[x_{1},\ldots,x_{i}]$ belongs to $X_{m+1}$ then $m+1-i$ of its elements belong to $A_{n}$, and so $[x_{1},\ldots,x_{i}]=[x_{1},\ldots,x_{i}]^{2^{m+1-i-k}}\in\widetilde{X}_{m+1}$, where $k=m+1-i$.
Now let $y\in Z_{m+1}$ be of the form $y=x^{2}$, where $x\in Z_{m}=\widetilde{Z}_{m}$ by induction. If $x$ is a generator of $\widetilde{Z}_{m}$ then $x^{2}\in \widetilde{Z}_{m+1}$ from the definition of $\widetilde{Z}_{m+1}$. If $x=x_{1}\cdots x_{l}$, where $x_{i}$ is a generator of $\widetilde{Z}_{m}$ for all $1\leq i \leq l$, then making use of a decomposition of $x$ analogous to that of~(\ref{produto}),
the previous sentence and the fact that $[H,\widetilde{Z}_{m}]\subset \widetilde{Z}_{m+1}$ by Remark~\ref{zm}, it follows that $y\in \widetilde{Z}_{m+1}$. We conclude that $Z_{m+1}\subset \widetilde{Z}_{m+1}$, and hence $Z_{m+1}= \widetilde{Z}_{m+1}$.


To show that $Z_{m+1}\subset Y_{m+1}$, let $y\in Z_{m+1}$. We first consider the following two possibilities:
\begin{enumerate}[(i)]
\item\label{it:zym1i} $y=x^{2}$, where $x\in Z_{m}=Y_{m}$ by induction. If $x=D^{2^{m-2}}_{j}$ for some $j=1,\ldots,n$, it follows from the definition of $Y_{m+1}$ that $y=x^{2}=D^{2^{m-1}}_{j}\in Y_{m+1}$. If $x=[x_{i},x_{j}]$, where $1\leq i\leq j <m$, $i+j=m$, $x_{i}\in Y_{i}$ and $x_{j}\in Y_{j}$, then:
\[ y=[x_{i},x_{j}]^{2}=\underbrace{\Bigl[x_{i},\underbrace{[x_{j},x^{-1}_{i}]}_{\in[Y_{i},Y_{j}]\subset Y_{m}}\Bigr]}_{\in[H,Y_{m}]\subset Y_{m+1}}x_{j}x^{-1}_{i}\underbrace{\Bigl[x^{-1}_{j},\underbrace{x^{2}_{i}}_{\in Z_{i+1}=Y_{i+1}}\Bigr]}_{\in [Y_{j},Y_{i+1}]\subset Y_{m+1}}x_{i}x^{-1}_{j}.\]
Note that to obtain $x_{i}^{2}\in Y_{i+1}$, we have $i<m$, so $Y_{i}=Z_{i}$ and $Y_{i+1}=Z_{i+1}$ by the induction hypothesis.
Therefore $y\in Y_{m+1}$ since $Y_{m+1}$ is normal in $H$ by Remark~\ref{ym}.
	
\item\label{it:zym1ii} $y=[x_{1},\ldots,x_{i}]\in\Gamma_{i}(H)$, where $2\leq i \leq m+1$, $x_{1},\ldots,x_{i}\in H$, and there exist $1\leq j_{1}<\ldots<j_{m+1-i}\leq i$ such that $x_{j_{1}},\ldots,x_{j_{m+1-i}}\in A_{n}$. If $j_{1}=1$, \emph{i.e.}\ $x_{1}\in A_{n}$, then $x_{1}\in Y_{2}$, $[x_{2},\ldots,x_{i}]\in \Gamma_{i-1}(H)$, and $m-i=(m-1)-(i-1)$ elements of $\left\{x_{2},\ldots,x_{i}\right\}$ belong to $A_{n}$, so $[x_{2},\ldots,x_{i}]\in Z_{m-1}=Y_{m-1}$ by induction. Hence: \[y=\left[x_{1},[x_{2},\ldots,x_{i}]\right]\in[Y_{2},Y_{m-1}]\subset Y_{m+1}.\] If $j_{1}>1$ then $[x_{2},\ldots,x_{i}]\in \Gamma_{i-1}(H)$ and $m+1-i=m-(i-1)$ elements of $\left\{x_{2},\ldots,x_{i}\right\}$ belong to $A_{n}$. Therefore $[x_{2},\ldots,x_{i}]\in Z_{m}=Y_{m}$ by induction, and thus: \[y=\left[x_{1},[x_{2},\ldots,x_{i}]\right]\in[H,Y_{m}]\subset Y_{m+1}.\]
\end{enumerate}
For the general case, if $y$ is a product of conjugates of the two types of elements described in~(\ref{it:zym1i}) and~(\ref{it:zym1ii}) above, then $y\in Y_{m+1}$ because $Y_{m+1}$ is normal in $H$ by Remark~\ref{ym}, and we conclude that $Z_{m+1}\subset Y_{m+1}$.

To complete the proof, it remains to see that $Y_{m+1}\subset \widetilde{Z}_{m+1}$. From the definition of $\widetilde{Z}_{m+1}$ given in~(\ref{eq:defZm}), it is clear that $A^{2^{m-1}}_{n} \subset \widetilde{Z}_{m+1}$. Applying the induction hypothesis, we have $[Y_{i},Y_{j}]= [\widetilde{Z}_{i},\widetilde{Z}_{j}]$ for all $1\leq i,j\leq m$ for which $i+j=m+1$, so it suffices to show that $[\widetilde{Z}_{i},\widetilde{Z}_{j}]\subset \widetilde{Z}_{i+j}$. We shall prove by induction on $i$ that $[\widetilde{Z}_{i},\widetilde{Z}_{j}]\subset \widetilde{Z}_{i+j}$ for all $1\leq i \leq m$ and $j\geq 1$. The inclusion that we require then follows as a special case.
So let $j\geq 1$. If $i=1$ then $[\widetilde{Z}_{1},\widetilde{Z}_{j}] =[H,\widetilde{Z}_{j}]\subset \widetilde{Z}_{j+1}$ by Remark~\ref{zm}.
So assume that $1<i\leq m$, and suppose by induction that:
\begin{equation}\label{eq:IH2}
\text{$[\widetilde{Z}_{k},\widetilde{Z}_{j}]\subset \widetilde{Z}_{k+j}$ for all $j\geq 1$ and $1\leq k<i\leq m$.}
\end{equation}
The aim is to prove that $[\widetilde{Z}_{i},\widetilde{Z}_{j}]\subset \widetilde{Z}_{i+j}$. Let $x\in \widetilde{Z}_{i}$ and $y\in \widetilde{Z}_{j}$. Then $x\in Y_{i}$ by the induction hypothesis~(\ref{eq:IH1}) and the fact that $1<i\leq m$. Assume first that $x$ is a generator of $Y_{i}$, so that one of the following conditions holds: 
\begin{enumerate}
\item\label{it:xyyzz1} $x\in  [Y_{l},Y_{k}]$, where $1\leq l\leq k< i\leq m$ and $l+k=i$. Then $x\in [\widetilde{Z}_{l},\widetilde{Z}_{k}]$ by the induction hypothesis~(\ref{eq:IH1}).
\item\label{it:xyyzz2} $x\in A^{2^{i-2}}_{n}$,
\end{enumerate}
and let us prove that $[x,y]\in \widetilde{Z}_{i+j}$.
In case~(\ref{it:xyyzz1}),~\cite[Theorem~5.2]{MKS} implies that:
\begin{equation}\label{eq:MKS}
[[\widetilde{Z}_{l},\widetilde{Z}_{k}],\widetilde{Z}_{j}]= [\widetilde{Z}_{j},[\widetilde{Z}_{l},\widetilde{Z}_{k}]] \subset[\widetilde{Z}_{l},[\widetilde{Z}_{k},\widetilde{Z}_{j}]] \ldotp [\widetilde{Z}_{k},[\widetilde{Z}_{j},\widetilde{Z}_{l}]].
\end{equation}
By the induction hypothesis~(\ref{eq:IH2}) and the fact that $k,l<m$, we have: 
\begin{equation}\label{eq:ZLkj}
\left\{\begin{aligned}
{[\widetilde{Z}_{l},[\widetilde{Z}_{k},\widetilde{Z}_{j}]]} &\subset [\widetilde{Z}_{l},\widetilde{Z}_{k+j}]\subset \widetilde{Z}_{l+k+j}=\widetilde{Z}_{i+j}\\
[\widetilde{Z}_{k},[\widetilde{Z}_{j},\widetilde{Z}_{l}]]&=[\widetilde{Z}_{k},[\widetilde{Z}_{l},\widetilde{Z}_{j}]] \subset[\widetilde{Z}_{k},\widetilde{Z}_{l+j}]\subset \widetilde{Z}_{k+l+j}=\widetilde{Z}_{i+j},
\end{aligned}\right.
\end{equation}
and thus $[x,y]\in\widetilde{Z}_{i+j}$ for all $y\in \widetilde{Z}_{j}$ using~(\ref{eq:MKS}) and~(\ref{eq:ZLkj}).
In case~(\ref{it:xyyzz2}), let $x=D^{2^{i-2}}_{k}$, where $k\in \{1,\ldots, n\}$. We consider the following cases:
\begin{itemize}
\item $i=2$. Then $x\in A_{n}$. Let us show by induction on $j$ that $[A_{n},Z_{j}]\subset Z_{j+2}$, for all $j\geq 1$. Suppose first that $j=1$, let $x\in A_{n}$, and let $y\in Z_{1}=H$. Then $[x,y]\in \Gamma_{2}(H)$, and the set $\{ x,y \}$ has one element, $x$, in $A_{n}$, therefore $[x,y]\in X_{3}\subset Z_{3}$, and thus $[A_{n},Z_{1}]\subset Z_{3}$. Now suppose that $j>1$ and that $[A_{n},Z_{j-1}]\subset Z_{j+1}$ by induction. Let $x\in A_{n}$, and let $y\in Z_{j}$. First assume that $y$ is a generator of $Z_{j}$. If $y=[y_{1},\ldots,y_{l}]\in\Gamma_{l}(H)$ belongs to $X_{j}$ then $j-l$ elements of $\left\{y_{1},\ldots,y_{l}\right\}$ belong to $A_{n}$, therefore $[x,y]\in \Gamma_{l+1}(H)$, where $j-l+1=(j+2)-(l+1)$ elements of $\left\{x,y_{1},\ldots ,y_{l}\right\}$ belong to $A_{n}$. So $[x,y]\in X_{j+2} \subset Z_{j+2}$. If $y=z^{2}$,
where $z\in Z_{j-1}$, then: 
\[[x,y]=[x,z^{2}]=[x,z]^{2}\left[[z,x],z\right],\] 
by~(\ref{itercomm1}). Since $[x,z]\in[A_{n},Z_{j-1}]\subset Z_{j+1}$ by induction, it follows that $[x,z]^{2}\in Z_{j+2}$ from the definition of $Z_{j+2}$ in~(\ref{eq:defZm}), and $\left[[z,x],z\right]\in[Z_{j+1},H]\subset Z_{j+2}$ by Remark~\ref{zm} and the fact that $\widetilde{Z}_{l}=Z_{l}$ for all $l\geq 1$. Therefore $[x,y]\in Z_{j+2}$. Now if $y=(\alpha_{1}y^{\delta_{1}}_{1}\alpha^{-1}_{1})\cdots (\alpha_{s}y^{\delta_{s}}_{s}\alpha^{-1}_{s})$, where for all $k=1,\ldots,s$, $y_{k}$ is a generator of $Z_{j}$, $\alpha_{k}\in H$ and $\delta_{k}\in\left\{1,-1\right\}$, then applying~(\ref{eq:conjalphat}) and induction on $s$, we see that $[x,y]$ may be written as a product of conjugates of the $[x,y^{\delta_{k}}_{k}]$. Then $[x,y]\in \widetilde{Z}_{j+2}$ using the above computations, the normality of $Z_{j+2}$ in $H$, and the fact $\widetilde{Z}_{l}=Z_{l}$ for all $l\geq 1$ from the first part of the proof. 

	
\item $i\geq 3$. Then:
\[ [x,y]=[D^{2^{i-2}}_{k},y]=\bigl[D^{2^{i-3}}_{k},[D^{2^{i-3}}_{k},y]\bigr] [D^{2^{i-3}}_{k},y]^{2}\] by~(\ref{itercomm1}). Now $[D^{2^{i-3}}_{k},y] \in[\widetilde{Z}_{i-1},\widetilde{Z}_{j}]\subset\widetilde{Z}_{i-1+j}$ using~(\ref{eq:IH2}), and thus $[D^{2^{i-3}}_{k}, [D^{2^{i-3}}_{k},y]]\in [H,\widetilde{Z}_{i-1+j}]\subset\widetilde{Z}_{i+j}$ by Remark~\ref{zm}. Since $[D^{2^{i-3}}_{k},y]\in\widetilde{Z}_{i-1+j}=Z_{i-1+j}$, we have $[D^{2^{i-3}}_{k},y]^{2}\in Z_{i+j}=\widetilde{Z}_{i+j}$, so $[x,y]\in \widetilde{Z}_{i+j}$. 
\end{itemize}

If now $x\in Y_{i}$ is a product of conjugates of generators of $Y_{i}$, then $[x,y]\in \widetilde{Z}_{i+j}$ using the above computations,~(\ref{eq:conjalphat}), and the normality of $\widetilde{Z}_{i+j}$. This shows that $[\widetilde{Z}_{i},\widetilde{Z}_{j}]= [Y_{i},\widetilde{Z}_{j}]\subset \widetilde{Z}_{i+j}$ as claimed, and so $Y_{m+1}\subset \widetilde{Z}_{m+1}$ as required.
%
\end{proof} 

\begin{teo}\label{soluvel} If $n\geq 3$, then $(P_{n+1}(\mathbb{K}))^{(m)}\subset Z_{m+1}\rtimes (P_{n}(\mathbb{K}))^{(m)}$ for all $m\geq1$. In particular, $P_{n}(\mathbb{K})$ is residually soluble for all $n\geq1$.
\end{teo}

\begin{proof} If $n=1,2$ then by~(\ref{eq:pi1}) and Theorem~\ref{GammaP2}, $P_{n}(\mathbb K)$ is residually nilpotent, and so is residually soluble. Suppose by induction on $n$ that $P_{n}(\mathbb{K})$ is residually soluble for some $n\geq 2$, and consider the Fadell-Neuwirth split short exact sequence~(\ref{seqFNklein}). The result follows from Theorem~\ref{serie}, Lemmas~\ref{VmYm} and~\ref{ymzm}, and the fact that $\bigcap_{m\geq 2}Z_{m}=\{ 1\}$. \end{proof} 

\subsection{The case of the braid group $B_{n}(\mathbb{K})$}\label{sec:seriesBnK}

In this section, we prove Theorem~\ref{klein}. We start by giving two propositions in the cases where $B_{n}(\mathbb{K})$ is not residually nilpotent or is not residually soluble.

\begin{prop}\label{klein.nilp} If $n\geq 3$ then $B_{n}(\mathbb{K})$ is not residually nilpotent. Further, for all $i\geq 3$, $\Gamma_{2}(B_{n}(\mathbb{K}))=\Gamma_{i}(B_{n}(\mathbb{K}))=\left\langle \!\left\langle \sigma^{-1}_{2}\sigma_{1}\right\rangle \! \right\rangle_{B_{n}(\mathbb{K})}$.
\end{prop}

\begin{proof} In this proof, the relation numbers will refer to those of Theorem~\ref{total}. Let $n\geq 3$, and consider the following short exact sequence:
\begin{equation}\label{eq:bgamma3K}
1\longrightarrow \frac{\Gamma_{2}(B_{n}(\mathbb{K}))}{\Gamma_{3}(B_{n}(\mathbb{K}))}\longrightarrow \frac{B_{n}(\mathbb{K})}{\Gamma_{3}(B_{n}(\mathbb{K}))}\stackrel{p_{\ast}}{\longrightarrow} \frac{B_{n}(\mathbb{K})}{\Gamma_{2}(B_{n}(\mathbb{K}))}\longrightarrow 1,
\end{equation}
where $p_{\ast}$ is the canonical projection. Using relations~(\ref{it:full1}) and~(\ref{it:full2}), and arguing as in~\cite[p.~680]{GGjktr} or~\cite[Proposition~3]{BGG}, it follows that the $\Gamma_{3}(B_{n}(\mathbb{K}))$-cosets of $\sigma_{1},\ldots,\sigma_{n-1}$ in $B_{n}(\mathbb{K})/\Gamma_{3}(B_{n}(\mathbb{K}))$ are all identified to a single element,  which we denote by $\sigma$.
Since $n\geq 3$, it follows from relations~(\ref{it:full3}) and~(\ref{it:full4}) that the $\Gamma_{3}(B_{n}(\mathbb{K}))$-cosets of $a$ and $b$ commute with $\sigma$. By relation~(\ref{it:full7}) and~(\ref{eq:bgamma3K}), $\sigma$ is of order $2$, by relation~(\ref{it:full5}), the $\Gamma_{3}(B_{n}(\mathbb{K}))$-coset of $a$ commutes with that of $b$, and from relation~(\ref{it:full8}), the $\Gamma_{3}(B_{n}(\mathbb{K}))$-coset of $a$ is of order $2$ in $B_{n}(\mathbb{K})/ \Gamma_{3}(B_{n}(\mathbb{K}))$. By Remark~\ref{Ab}, $\overline{\sigma}$ and the $\Gamma_{2}(B_{n}(\mathbb{K}))$-cosets of $a$ and $b$ are non trivial in $B_{n}(\mathbb{K})/ \Gamma_{2}(B_{n}(\mathbb{K}))$, therefore $\sigma$ and the $\Gamma_{3}(B_{n}(\mathbb{K}))$-cosets of $a$ and $b$ are also non trivial in $B_{n}(\mathbb{K})/ \Gamma_{2}(B_{n}(\mathbb{K}))$ and  satisfy the same relations in $B_{n}(\mathbb{K})/ \Gamma_{3}(B_{n}(\mathbb{K}))$ as their images in $B_{n}(\mathbb{K})/ \Gamma_{2}(B_{n}(\mathbb{K}))$ under $p_{\ast}$. Hence $p_{\ast}$ is a isomorphism, and $\Gamma_{2}(B_{n}(\mathbb{K}))=\Gamma_{3}(B_{n}(\mathbb{K}))$, so $\Gamma_{2}(B_{n}(\mathbb{K}))=\Gamma_{i}(B_{n}(\mathbb{K}))$ for all $i\geq 3$. Since $\Gamma_{2}(B_{n}(\mathbb{K}))$ is non trivial, we see that $B_{n}(\mathbb{K})$ is not residually nilpotent.

It remains to show that $\Gamma_{2}(B_{n}(\mathbb{K}))=\left\langle\! \left\langle \sigma^{-1}_{2}\sigma_{1}\right\rangle \!\right\rangle_{B_{n}(\mathbb{K})}$. From relations~(\ref{it:full1}) and~(\ref{it:full2}), for all $i=1,\ldots, n-2$, we have $[\sigma_{i},\sigma_{i+1}]=\sigma_{i}\sigma_{i+1}\sigma^{-1}_{i}\sigma^{-1}_{i+1}=\sigma^{-1}_{i+1}\sigma_{i}$, in particular,  $\left\langle \!\left\langle \sigma^{-1}_{2}\sigma_{1}\right\rangle\!\right\rangle_{B_{n}(\mathbb{K})} \subset \Gamma_{2}(B_{n}(\mathbb{K}))$. To prove the other inclusion, for all $i=2,\ldots, n-2$, note that:
\[
\begin{array}{ll}[\sigma_{i},\sigma_{i+1}]& =\sigma^{-1}_{i+1}\sigma_{i}=\sigma^{-1}_{i+1}\underbrace{\sigma_{i}(\sigma_{i-1}\sigma_{i}}_{(\ref{it:full1})}\sigma^{-1}_{i}\sigma^{-1}_{i-1})=\underbrace{\sigma^{-1}_{i+1}(\sigma_{i-1}}_{(\ref{it:full2})}\sigma_{i}\sigma_{i-1})\sigma^{-1}_{i}\sigma^{-1}_{i-1}\\&=(\sigma_{i-1}\underbrace{\sigma^{-1}_{i+1})\sigma_{i}}_{(\ref{it:full1})}\sigma_{i-1}\sigma^{-1}_{i}\sigma^{-1}_{i-1}=\sigma_{i-1}(\sigma_{i}\sigma_{i+1}\sigma^{-1}_{i}\underbrace{\sigma^{-1}_{i+1})\sigma_{i-1}}_{(\ref{it:full2})}\sigma^{-1}_{i}\sigma^{-1}_{i-1} \\&= \sigma_{i-1}\sigma_{i}\sigma_{i+1}(\sigma^{-1}_{i}\sigma_{i-1})\sigma^{-1}_{i+1}\sigma^{-1}_{i}\sigma^{-1}_{i-1}. \end{array}\] 
It follows by induction on $i$ that $[\sigma_{i},\sigma_{i+1}]\in\left\langle \!\left\langle \sigma^{-1}_{2}\sigma_{1}\right\rangle\!\right\rangle_{B_{n}(\mathbb{K})}$ for all $i=1,\ldots, n-2$. Further, 
\begin{equation}\label{eq:comasig}
[a,\sigma_{1}]=a\sigma_{1}(\sigma^{-1}_{2}\underbrace{\sigma_{2})a^{-1}}_{(\ref{it:full3})}\sigma^{-1}_{1}=a(\sigma_{2}\sigma^{-1}_{1})^{-1}a^{-1}\ldotp(\sigma_{2}\sigma^{-1}_{1})\in\left\langle \!\left\langle \sigma^{-1}_{2}\sigma_{1}\right\rangle \! \right\rangle_{B_{n}(\mathbb{K})},
\end{equation}
and similarly, $[b,\sigma_{1}]\in\left\langle \!\left\langle \sigma^{-1}_{2}\sigma_{1}\right\rangle \! \right\rangle_{B_{n}(\mathbb{K})}$ using relation~(\ref{it:full4}). To see that $[b,a]\in \left\langle \!\left\langle \sigma^{-1}_{2}\sigma_{1}\right\rangle \! \right\rangle_{B_{n}(\mathbb{K})}$, first note that:
\begin{equation}\label{eq:comsig2}
\sigma^{2}_{1}=\sigma_{1}b^{-1}\sigma_{1}b^{-1}\sigma^{-1}_{1}b\sigma^{-1}_{1}b= (\sigma_{1}b^{-1}[\sigma_{1},b^{-1}]b\sigma^{-1}_{1})[\sigma_{1},b^{-1}]\in\left\langle \!\left\langle \sigma^{-1}_{2}\sigma_{1}\right\rangle \! \right\rangle_{B_{n}(\mathbb{K})}
\end{equation}
using relation~(\ref{it:full7}), and:
\[\begin{array}{ll}bab^{-1}a^{-1}&=ba(\sigma^{-1}_{1}\underbrace{\sigma_{1})b^{-1}(\sigma_{1}}_{(\ref{it:full5})}\sigma^{-1}_{1})a^{-1}=ba\sigma^{-1}_{1}(a^{-1}\sigma^{-1}_{1}b^{-1}\sigma_{1}a)\sigma^{-1}_{1}a^{-1}\\
&=b[a,\sigma^{-1}_{1}]\sigma^{-2}_{1}b^{-1}[\sigma_{1},a] \in \left\langle \!\left\langle \sigma^{-1}_{2}\sigma_{1}\right\rangle \! \right\rangle_{B_{n}(\mathbb{K})},\end{array}\]
by~(\ref{eq:comasig}) and~(\ref{eq:comsig2}). Since the result is valid for the generators of $B_{n}(\mathbb K)$, the result follows for an arbitrary element of $\Gamma_{2}(B_{n}(\mathbb K))$ using the formula given by~(\ref{eq:conjalphat}) and by the normality of $\left\langle \!\left\langle \sigma^{-1}_{2}\sigma_{1}\right\rangle \! \right\rangle_{B_{n}(\mathbb{K})}$. 
We conclude that $\Gamma_{2}(B_{n}(\mathbb{K}))\subset \left\langle \!\left\langle \sigma^{-1}_{2}\sigma_{1}\right\rangle \! \right\rangle_{B_{n}(\mathbb{K})}$, and hence that $\Gamma_{2}(B_{n}(\mathbb{K}))=\left\langle \!\left\langle \sigma^{-1}_{2}\sigma_{1}\right\rangle \! \right\rangle_{B_{n}(\mathbb{K})}$.
\end{proof}

\begin{prop}\label{klein.sol} If $n\geq5$ then $B_{n}(\mathbb{K})$ is not residually soluble. Further, $(B_{n}(\mathbb{K}))^{(1)}=(B_{n}(\mathbb{K}))^{(i)}$ for all $i\geq 2$.
\end{prop}

\begin{proof} Let $n\geq 5$. Once more, the relation numbers will refer to those of Theorem~\ref{total}. As in the case of the torus (Proposition~\ref{toro2}), first consider the following short exact sequence: \[1\longrightarrow \frac{(B_{n}(\mathbb{K}))^{(1)}}{(B_{n}(\mathbb{K}))^{(2)}}\stackrel{i}{\longrightarrow} \frac{B_{n}(\mathbb{K})}{(B_{n}(\mathbb{K}))^{(2)}}\stackrel{p_{\ast}}{\longrightarrow} B_{n}(\mathbb{K})^{\text{Ab}}\longrightarrow 1,\] where $p_{\ast}$ is the canonical projection. Using relations~(\ref{it:full1}) and~(\ref{it:full2}), for all $i=1,\ldots,n-1$, the $\sigma_{i}$ belong to the same $(B_{n}(\mathbb{K}))^{(2)}$-coset, denoted by $\sigma$, in ${B_{n}(\mathbb{K})}/{(B_{n}(\mathbb{K}))^{(2)}}$ (the hypothesis that $n\geq 5$ is used here, see~\cite[Theorem~1.4, p.~3389]{GG}). By Remark~\ref{Ab}, $\sigma$ and the $(B_{n}(\mathbb K)^{(2)}$-cosets of $a$ and $b$ are non trivial. From relations~(\ref{it:full3}) and~(\ref{it:full4}), the $(B_{n}(\mathbb{K}))^{(2)}$-cosets of $a$ and $b$ commute with $\sigma$. Relation~(\ref{it:full7}) implies that $\sigma$ is of order $2$, and so from relation~(\ref{it:full5}), the $(B_{n}(\mathbb{K}))^{(2)}$-coset of $a$ commutes with that of $b$. By relation~(\ref{it:full8}), the $(B_{n}(\mathbb{K}))^{(2)}$-coset of $a$ is of order $2$. So $p_{\ast}$ is a isomorphism, and $(B_{n}(\mathbb{K}))^{(1)}=(B_{n}(\mathbb{K}))^{(2)}$. The second part then follows.
\end{proof}

\begin{proof}[Proof of Theorem~\ref{klein}]
By Proposition~\ref{klein.nilp} (resp.\ Proposition~\ref{klein.sol}), if $n\geq 3$ (resp.\ $n\geq 5$), $B_{n}(\mathbb{K})$ is not residually nilpotent (resp.\ not residually soluble). Conversely, using Theorem~\ref{TG}, the short exact sequence~(\ref{seq1}), 
and the fact that $P_{2}(\mathbb{K})$ is residually $2$-finite by Theorem~\ref{gammaP2}, $B_{2}(\mathbb{K})$ is residually $2$-finite. In particular, $B_{2}(\mathbb{K})$ is residually nilpotent. By~(\ref{eq:pi1}), $B_{1}(\mathbb{K})=\pi_{1}(\mathbb K)$ is residually nilpotent. 
The fact that $P_{n}(\mathbb{K})$ is residually soluble for all $n\geq1$ by Theorem~\ref{soluvel} implies that $B_{n}(\mathbb{K})$ is residually soluble for all $n\leq4$ using Theorem~\ref{TG}. 
\end{proof}

\section{The case of non-orientable surfaces of higher genus}\label{sec:higher}


In this short section, we prove Theorem~\ref{final}, by generalising Propositions~\ref{klein.nilp} and~\ref{klein.sol} to non-orientable surfaces of higher genus.

\begin{teo}\label{not} Let $M$ be a compact, connected non-orientable surface without boundary and of genus $g\geq 3$. Then $B_{n}(M)$ is not residually nilpotent if $n\geq 3$, and is not residually soluble if $n\geq5$.
\end{teo}

\begin{proof}
The relation numbers will refer to those of Theorem~\ref{Bel}. Using Theorem~\ref{Bel} it is easy to see that\[B_{n}(M)^{\text{Ab}}=\left\langle \sigma,a_{i},i=1,\ldots,g\,:\,a^{2}_{1}\cdots a^{2}_{g}=\sigma^{2(n-1)}, [\sigma,a_{i}]=[a_{i},a_{j}]=1, i,j=1,\ldots,n\right\rangle,\] where $\sigma$ denotes the $\Gamma_{2}(B_{n}(M))$-coset of $\sigma_{i}$ for all $i=1,\ldots,n-1$. As in the proof of Proposition~\ref{klein.nilp} (resp.\ Proposition~\ref{klein.sol}), using relations~(\ref{it:full1}) and~(\ref{it:full2}), one may show that for all $i=1,\ldots, n-1$, the $\sigma_{i}$ belong to the same $\Gamma_{3}(B_{n}(M))$-coset (resp.\ $(B_{n}(M))^{(2)}$-coset), which we also denote by $\sigma$, in $B_{n}(M)/\Gamma_{3}(B_{n}(M))$ (resp.\ in $B_{n}(M)/(B_{n}(M))^{(2)}$). If $n\geq 3$ (resp.\ $n\geq 5$), by relation~(\ref{it:full3}), the $\Gamma_{3}(B_{n}(M))$-coset (resp.\ $(B_{n}(M))^{(2)}$-coset) of  $a_{r}$ commutes with $\sigma$ for all $1\leq r \leq g$. By relation~(\ref{it:full5}), the $\Gamma_{3}(B_{n}(M))$-coset (resp.\ $(B_{n}(M))^{(2)}$-coset) of $a_{r}$ commutes with that of $a_{s}$ for all $1\leq r,s\leq g$. By~(\ref{it:full6}), $a^{2}_{1}\cdots a^{2}_{g}=\sigma^{2(n-1)}$, and relation~(\ref{it:full4}) does not give any new information. Thus $B_{n}(M)/\Gamma_{3}(B_{n}(M))$  (resp.\ $B_{n}(M)/(B_{n}(M))^{(2)}$) is isomorphic to $B_{n}(M)^{\text{Ab}}$. Hence $B_{n}(M)$ is not residually nilpotent if $n\geq 3$ (resp.\ is not residually soluble if $n\geq5$).
\end{proof}



\begin{proof}[Proof of Theorem~\ref{final}]
If $M=\mathbb{K}$, the result follows from Theorem~\ref{klein}, and if $M$ is a compact surface without boundary of genus $g\geq 3$, the conclusion follows from~\cite{BG} and from Theorem~\ref{not}. If $M=\mathbb RP^{2}$, by~\cite{GG1}, $B_{n}(\mathbb RP^{2})$ is residually nilpotent if $n\leq3$, and if $n\neq4$, $B_{n}(\mathbb RP^{2})$ is residually soluble if $n<4$.

The case $n=4$ may be obtained by using Theorem~\ref{TG} and the following Fadell-Neuwirth short exact sequence:\[1\longrightarrow P_{m}(\mathbb RP^{2}\setminus\left\{x_{1},\ldots,x_{n}\right\})\longrightarrow P_{m+n}(\mathbb RP^{2})\longrightarrow P_{n}(\mathbb RP^{2})\longrightarrow 1,\] where $n\geq 2$. Note that if $n=2$, $P_{2}(\mathbb RP^{2})$ is the quaternion group of order $8$~\cite{VB}, which is $2$-finite, and by~\cite{BG}, $P_{m}(\mathbb RP^{2}\setminus\left\{x_{1},\ldots,x_{n}\right\})$ is residually $2$-finite. Therefore $P_{m+2}(\mathbb RP^{2})$ is residually $2$-finite for all $m\geq 1$, in particular $P_{4}(\mathbb RP^{2})$ is residually soluble. Applying Theorem~\ref{TG} to the short exact sequence~(\ref{seq1}), we see that $B_{4}(\mathbb RP^{2})$ is residually soluble.
\end{proof}

\appendix
\section*{Appendix}

\renewcommand{\theequation}{A\arabic{equation}}

\setcounter{teo}{0}
\renewcommand{\theteo}{A\arabic{teo}}

Let $M$ be the M\"obius band, and let $n\geq 1$. The braid groups of $M$ are those of $\mathbb{R}P^{2}$ with a single point removed~\cite[proof of Theorem~2(a)]{GG2}, and so $P_{n}(M)$ is the group $\Gamma_{n,1}(\mathbb{R}P^{2})$ of~\cite[Proposition~11]{GG3}. We use the notation and results of that proposition in what follows. In particular, $P_{n}(M)$ is generated by the set $\{ A_{i,j}, \rho_{j} \, \mid \, \text{$1\leq i<j$ and $2\leq j\leq n+1$}\}$.

\begin{prop}\label{prop:centremobius}
Let $n\geq 1$. Then $Z(P_{n}(M))=Z(B_{n}(M))$ is infinite cyclic, generated by $\rho_{2}$ if $n=1$, and by the full twist $\Delta^{2}_{n+1}$ if $n\geq 2$. 
\end{prop}

\begin{proof}
If $n=1$ then $P_{1}(M)=B_{1}(M)$ is infinite cyclic, generated by $\rho_{2}$, and the result follows in this case. So suppose that $n\geq 2$. We consider the following short exact sequence:
\begin{equation}\label{eq:sespnbn}
1\longrightarrow P_{n}(M)\longrightarrow B_{n}(M)\longrightarrow S_{n}\longrightarrow 1. 
\end{equation}
Equation~(\ref{eq:sespnbn}) implies that $B_{n}(M)$ is generated by $\{ \sigma_{2},\ldots, \sigma_{n}, \rho_{2},\ldots, \rho_{n+1} \}$. The braid $\Delta^{2}_{n+1}$ generates the centre of $B_{n+1}$, thus $\sigma_{i}$ commutes with $\Delta^{2}_{n+1}$ for all $i=2,\ldots,n$. Further, using~\cite[Proposition~11, relation~(V)]{GG3}, $\Delta^{2}_{n+1}$ may be written in the following form:
\begin{align}
\Delta^{2}_{n+1}&=(A_{1,2})(A_{1,3}A_{2,3})\cdots(A_{1,n}A_{2,n}\cdots A_{n-1,n})(A_{1,n+1}A_{2,n+1}\cdots A_{n,n+1})\notag\\
&=(A_{1,2})(A_{1,3}A_{2,3})\cdots(A_{1,n}A_{2,n}\cdots A_{n-1,n})\rho^{-2}_{n+1}\label{eq:deltan1}
\end{align}
Since $\rho_{n+1}$ commutes with $A_{i,j}$ for all $1\leq i<j<n+1$~\cite[Proposition~11, relation~(II)]{GG3}, it follows that $\rho_{n+1}$ commutes with $\Delta^{2}_{n+1}$. Now the relation $\rho_{i+1}=\sigma_{i}^{-1} \rho_{i}\sigma_{i}^{-1}$ of~\cite[p.~83]{VB} for $B_{n+1}(\mathbb{R}P^{2})$ also holds in $B_{n}(M)$ for all $i=2,\ldots,n$, so $\rho_{i}=\sigma_{i}\cdots \sigma_{n}\rho_{n+1}\sigma_{n} \cdots \sigma_{i}$, from which we conclude that $\rho_{i}$ commutes with $\Delta^{2}_{n+1}$. Thus $\Delta^{2}_{n+1}$ commutes with all of the elements of a generating set of $B_{n}(M)$, hence $\left\langle \Delta^{2}_{n+1}\right\rangle\subset Z(B_{n}(M))$, and $\left\langle \Delta^{2}_{n+1}\right\rangle\subset Z(P_{n}(M))$ since $\Delta^{2}_{n+1}\in P_{n}(M)$. We now  consider the following two cases:
\begin{enumerate}[(a)]
\item\label{it:parta} $n=2$. Let $p\colon\thinspace P_{2}(M)\longrightarrow P_{1}(M)$ denote the surjective homomorphism given geometrically by forgetting the second string. The kernel of $p$ is a free group of rank $2$ for which $(A_{2,3},\rho_{3})$ is a basis, and $P_{1}(M)=\left\langle \rho_{2}\right\rangle$ is infinite cyclic. Let $Q=\langle \rho_{2}^{2}\rangle$ be the index $2$ subgroup of $P_{1}(M)$, and let $G=p^{-1}(Q)$ be the index $2$ subgroup of $P_{2}(M)$. Then we have the following commutative diagram of short exact sequences:
\begin{equation}\label{eq:commdiagM}
\begin{tikzcd}
& & 1\arrow{d} & 1\arrow{d} &\\
1 \arrow{r} & \operatorname{\text{Ker}}(p) \arrow{r}\arrow[equal]{d} & G \arrow{r}{p\bigl\lvert_{G}\bigr.} \arrow{d} & Q \arrow{r} \arrow{d} & 1\\
1 \arrow{r} & \operatorname{\text{Ker}}(p) \arrow{r} & P_{2}(M) \arrow{r}{p} \arrow{d}{q'} & P_{1}(M) \arrow{r} \arrow{d}{q} & 1,\\
& & \mathbb{Z}_{2}\arrow{d}  \arrow[equal]{r} & \mathbb{Z}_{2}\arrow{d} &\\
& & 1 & 1 &
\end{tikzcd}
\end{equation}
where $q\colon\thinspace P_{1}(M) \longrightarrow \mathbb{Z}_{2}$ is defined by $q(\rho_{2})=\overline{1}$, and $q'\colon\thinspace P_{2}(M) \longrightarrow \mathbb{Z}_{2}$ is given by $q'=q\circ p$. From~(\ref{eq:deltan1}), we have $\Delta_{3}^{2}=A_{1,2}\rho_{3}^{-2}$, and since $A_{1,2}=\rho_{2}^{-1} A_{2,3}\rho_{2}^{-1}$ by~\cite[Proposition~11, relation~(V)]{GG3}, we see that:
\begin{equation}\label{eq:reldelta3}
\Delta_{3}^{2}=\rho_{2}^{-1} A_{2,3}\rho_{2}^{-1}\rho_{3}^{-2}=\rho_{2}^{-2} \ldotp \underbrace{\rho_{2} A_{2,3}\rho_{2}^{-1}\rho_{3}^{-2}}_{\in \operatorname{\text{Ker}}(p)}
\end{equation}
using exactness of~(\ref{eq:commdiagM}). So the restriction $p\bigl\lvert_{G}\bigr.\colon\thinspace G \longrightarrow Q$ admits a section given by sending $\rho_{2}^{2}$ to $\Delta_{3}^{-2}$, and from this and the fact that $\Delta^{2}_{3}\in Z(P_{2}(M))$, the upper row of~(\ref{eq:commdiagM}) splits as a direct product. In particular, $G$ is the internal direct product of $\operatorname{\text{Ker}}(p)$ and $\left\langle \Delta^{2}_{3}\right\rangle$, and $Z(G)=\left\langle \Delta^{2}_{3}\right\rangle$. Now $P_{2}(M)$ (resp.\ $G$) is an index~$2$ subgroup of $B_{2}(M)$ (resp.\ of $P_{2}(M)$), and a transversal is given by $\{ 1,\sigma_{2}\}$ (resp.\ by $\{ 1,\rho_{2}\}$). So every element $z$ of $B_{2}(M)$ may be written as $z=a\rho_{2}^{i} \sigma_{2}^{j}$, where $a\in G$ and $i,j\in \{ 0,1\}$ are unique. We shall refer to this expression as the \emph{normal form} of $z$. Let $z\in Z(B_{2}(M))$, and assume first that $z\notin P_{2}(M)$, so $j=1$. Consider the basis $(u,v)$ of $\operatorname{\text{Ker}}(p)$, where $u=A_{2,3}\rho_{3}$ and $v=\rho_{3}$. By~\cite[Proposition~11, relations~(III) and~(IV)]{GG3}, we have:
\begin{equation}\label{eq:conjrho2}
\text{$\rho_{2}u \rho_{2}^{-1}= \rho_{3}^{-1}A_{2,3}^{-1}\rho_{3}\ldotp A_{2,3}\rho_{3}=u^{-1}vu$ and $\rho_{2}v \rho_{2}^{-1}=u$.}
\end{equation}
Hence the action by conjugation of $\rho_{2}$ on $\operatorname{\text{Ker}}(p)$ is given by composing the involution that exchanges $u$ and $v$ with conjugation by $u^{-1}$. The relation $\rho_{3}=\sigma_{2}^{-1} \rho_{2}\sigma_{2}^{-1}$ implies that:
\begin{equation*}
\sigma_{2} v \sigma_{2}^{-1} =\sigma_{2} \rho_{3} \sigma_{2}^{-1}= \rho_{2} \sigma_{2}^{-2}=\rho_{2} A_{2,3}^{-1}=\rho_{2}vu^{-1}=\rho_{2}vu^{-1}\rho_{2}^{-1} \ldotp \rho_{2}= v^{-1}u \rho_{2}
\end{equation*}
by~(\ref{eq:conjrho2}). Since $z\in Z(B_{2}(M))$, $z$ and $v$ commute, so:
\begin{equation}\label{eq:vwv}
a\rho_{2}^{i}\sigma_{2}=z= vzv^{-1}= v a\rho_{2}^{i}\sigma_{2} v^{-1} \sigma_{2}^{-1} \ldotp \sigma_{2}= v a\rho_{2}^{i}\rho_{2}^{-1} u^{-1}v \sigma_{2}.
\end{equation}
If $i=1$ then the left- and right-hand sides of~(\ref{eq:vwv}) are in normal form, and they clearly differ. If $i=0$ then using the fact that $\rho_{2}^{2}=u^{-1}v^{-1} \Delta^{-2}_{3}\in G$ by~(\ref{eq:reldelta3}) and~(\ref{eq:conjrho2}), equation~(\ref{eq:vwv}) may be written as:
\begin{equation*}
a\sigma_{2}=v a\rho_{2}^{-1} u^{-1}v \sigma_{2}=v a\rho_{2}^{-2} \ldotp \rho_{2} u^{-1}v \rho_{2}^{-1} \ldotp\rho_{2} \sigma_{2}=
v a 
\Delta^{2}_{3} u^{2} \rho_{2} \sigma_{2}.
\end{equation*}
Again the left- and right-hand sides are in normal form, and they differ also. In both cases, this contradicts the fact that $z\in Z(B_{2}(M))$, and so we conclude that $j=0$. Hence $z\in P_{2}(M)$, and $\left\langle \Delta^{2}_{3}\right\rangle \subset Z(B_{2}(M)) \subset Z(P_{2}(M))$. It remains to show that $z\in \left\langle \Delta^{2}_{3}\right\rangle$. Suppose that $i=1$. Since $z\in Z(P_{2}(M))$, $z$ commutes with $u$, so:
\begin{equation*}
a\rho_2=z =uzu^{-1}=ua\rho_2 u^{-1}=ua\rho_2 u^{-1} \rho_2^{-1}\ldotp \rho_2=uau^{-1}v^{-1}u \rho_{2}
\end{equation*}
by~(\ref{eq:conjrho2}). Both sides are in normal form, and thus $a=uau^{-1}v^{-1}u$ in $\operatorname{\text{Ker}}(p)$, which gives rise to a contradiction under Abelianisation in this free group. Hence $i=0$, and thus $z\in G$. So $z\in Z(G)$, and therefore $z\in \left\langle \Delta^{2}_{3}\right\rangle$. We conclude that  $\left\langle \Delta^{2}_{3}\right\rangle = Z(B_{2}(M)) = Z(P_{2}(M))$ as required.

\item $n\geq 3$. Then $Z(S_{n})=\{\operatorname{\text{Id}}\}$, and since the homomorphism $B_{n}(M)\longrightarrow S_{n}$ of~(\ref{eq:sespnbn}) is surjective, it follows that $\left\langle \Delta^{2}_{n+1}\right\rangle\subset Z(B_{n}(M))\subset Z(P_{n}(M))$. The proof of the fact that $Z(P_{n}(M))\subset \left\langle \Delta^{2}_{n+1}\right\rangle$ is analogous to that of the inclusion $Z(P_{n}(\mathbb{K}))\subset Z_{n}$ given in the second paragraph of Proposition~\ref{centro}, where the kernel of the Fadell-Neuwirth short exact sequence involving the pure braid groups of $M$ is a free group with trivial centre.\qedhere
\end{enumerate}
\end{proof}


\end{document}